\documentclass[11pt,a4paper,reqno]{amsart}
\usepackage{amsmath,amscd}
\usepackage{amssymb}
\usepackage{amsthm}

\usepackage[abbrev,nobysame,alphabetic]{amsrefs}
\usepackage[ocgcolorlinks, linkcolor=blue]{hyperref}

\usepackage{graphicx}

\usepackage{mathrsfs}
\usepackage{mathtools}
\mathtoolsset{showonlyrefs}
\usepackage[fleqn]{nccmath}

\usepackage{calc}
{\begin{list}{\arabic{enumi}.}{\usecounter{enumi}%
			\setlength{\labelsep}{0.5em}%
			\settowidth{\labelwidth}{(\arabic{enumi})}%
			\setlength{\leftmargin}{\labelwidth+\labelsep}}}%
	{\end{list}}

\usepackage{comment}
\usepackage{xcolor}

\DeclareRobustCommand{\SkipTocEntry}[5]{}

\newcommand{\I}{\mathrm{i}}
\DeclareMathOperator{\tr}{\mathrm{tr}}

\newcommand{\PD}{\partial}
\newcommand{\N}{\mathbb{N}}

\newcommand{\Beq}{\begin{equation}}
	\newcommand{\Eeq}{\end{equation}}
\newcommand{\beq}{\begin{equation*}}
	\newcommand{\eeq}{\end{equation*}}
\newcommand{\bal}{\begin{align}}
	\newcommand{\eal}{\end{align}}

\usepackage{mathtools}

\newcommand{\tblue}[1]{{\color{black}{#1}}}

\newcommand{\p}{\partial}

\newtheorem{theorem}{Theorem}[section]

\newtheorem{lemma}[theorem]{Lemma}
\newtheorem{proposition}[theorem]{Proposition}

\theoremstyle{definition}
\newtheorem{definition}{Definition}[section]

\newtheorem{remark}{Remark}[section]

\newcommand{\Lc}{\mathcal{L}}
\newcommand{\R}{\mathbb{R}}

\newcommand{\Rn}{\mathbb{R}^n}
\newcommand{\norm}[1]{\lVert #1 \rVert}
\newcommand{\abs}[1]{\lvert #1 \rvert}

\newcommand{\qum}[1]{\quad\mbox{#1}}
\newcommand{\lr}{\langle}
\newcommand{\rn}{\rangle}

\newcommand{\supp}{\mathop{\rm supp}}

\newcommand{\df}{\mathrm{d}}

\newcommand{\agl}[1][\cdot]{ \langle {#1} \rangle}
\allowdisplaybreaks

\DeclarePairedDelimiter\pair{\langle}{\rangle}
\DeclareMathOperator{\vspan}{span}

\numberwithin{equation}{section}

\usepackage{comment}
\usepackage{float}
\usepackage{tikz}
\usetikzlibrary{shapes,arrows.meta,quotes,angles,positioning,patterns,decorations.pathreplacing,calc}
\def\MarkRightAngle[size=#1](#2,#3,#4){
	\draw ($(#3)!#1!(#2)$) -- 
	($($(#3)!#1!(#2)$)!#1!90:(#2)$) --
	($(#3)!#1!(#4)$);
}

\usepackage{todonotes}

\title[]{An Inverse problem for a fourth order nonlinear Schr\"odinger equation (NLS)}

\author[R. K. Mishra, S. K. Sahoo, and C. Thakkar]{Rohit Kumar Mishra,  Anamika Purohit and Suman Kumar Sahoo}
\address{Department of Mathematics, IIT Gandhinagar, Gujarat, India}
\email{rohit.m@iitgn.ac.in, rohittifr2011@gmail.com}
\address{Department of Mathematics, IIT Bombay}
\email{suman@math.iitb.ac.in,sumansahootifr@gmail.com}
\address{Department of Mathematics, IIT Bombay}
\email{anamika.purohit@iitgn.ac.in, purohitanamika18@gmail.com}

\begin{document}
	\begin{abstract}
    We study an inverse problem for the time-dependent nonlinear fourth-order Schr\"odinger equation on both compact Euclidean domains and compact Riemannian manifolds, (say) $M$. This model arises in nonlinear fiber optics and the theory of optical solitons in gyrotropic media. Our main objective is the identification of unknown coefficients from the associated source-to-solution map, which assigns to each source term \( f \), supported in \( (0, T)\times \Gamma \), the corresponding solution \( u \) restricted to the same set, where \( \Gamma \subset M \) is a neighborhood of \( \partial M \). We prove that the zeroth-order term, the second-order coefficient, and the nonlinear coefficient are uniquely determined by this map. Moreover, the recovery of the symmetric second-order tensor reduces to the inversion of a divergent beam transform.
    
	\end{abstract}
	\maketitle
	\section{Introduction}
	\subsection{Main results.}Let $M$ be a convex and compact subset of $\mathbb{R}^n$ ($n\geq 2$) with a smooth boundary $\partial M$. The set $\Gamma\subset M$ be a neighborhood of $\partial M$ and $T > 0$. Suppose $A:=(A_{ij})_{1\leq i,j \leq n}$ be a symmetric 2-tensor field defined on $(0,T) \times M$ whose components $A_{ij}\in C_c^{\infty}\left((0,T)\times M\setminus\Gamma;\mathbb{R}\right)$, $q\in C_c^{\infty}\left((0,T)\times M\setminus\Gamma;\mathbb{R}\right)$. Let $\beta\in C_c^{\infty}\left((0, T)\times M\setminus\Gamma;\mathbb{R}\right)$ which is non-vanishing almost everywhere in the support of $A$ and $q$. We consider the following initial boundary value problem (IBVP) for a nonlinear fourth-order Schr\"odinger partial differential equation:
	\begin{equation}\label{main_IBVP}
		\begin{aligned}
			\begin{cases}
				\left(\I \partial_t +\Delta^2 -\nabla\cdot(A(t,x)\nabla) +q(t,x)\right)u(t,x)+\beta(t,x) u^2(t,x)=f(t,x) \quad &\mbox{in} \quad (0,T) \times M,\\
				u(t,x)= \frac{\partial u}{\partial \nu}(t,x)  = 0 \quad \mbox{on} \quad (0,T)\times \partial M,\\
				u(0,x)=0 \quad  \hspace{1.7cm} \,\mbox{in} \quad x\in M.
			\end{cases}
		\end{aligned}
	\end{equation}

The inverse problem considered here is the unique recovery of the coefficients $A(t,x)$, $q(t,x)$, and $\beta(t,x)$ from the source-to-solution map $L_{A,q,\beta}$ defined by
\begin{align}\label{eq_sos_map}
	L_{A,q,\beta}f := u|_{(0,T)\times \Gamma},
\end{align}
where $u$ is the unique solution to \eqref{main_IBVP} corresponding to a source term $f$ supported in $(0,T)\times \Gamma$. The existence and uniqueness of $u$ follow from the well-posedness of the initial boundary value problem \eqref{main_IBVP}.
To state the well-posedness of the nonlinear problem \eqref{main_IBVP}, we introduce the relevant function spaces (see, e.g., \cite[Chapter 5]{Evans_Book}). Let $X$ be a Banach space with norm $\|\cdot\|$. For $1 \leq p < \infty$, we define
\begin{align}
    L^p((0,T);X) := \left\{ u : [0,T] \to X \ \text{measurable} \; : \; \int_0^T \|u(t)\|^p \, \df t < \infty \right\},
\end{align}
equipped with the norm
\begin{align}
    \|u\|_{L^p((0,T);X)} = \left( \int_0^T \|u(t)\|^p \, \df t \right)^{1/p}.
\end{align}
  For any non-negative integer $\kappa$, the function space $H^{4\kappa}((0,T)\times \Gamma)$ is defined in the following way
     \begin{align*}
   H^{4\kappa}((0,T)\times \Gamma):=  H^{4\kappa}((0,T);L^2 (\Gamma)) \cap  L^2((0,T); H^{4\kappa} (\Gamma))
   \end{align*}
   with the norm
   \begin{align*}
   \norm{u}_{H^{4\kappa}((0,T)\times \Gamma)}=  \norm{u}_{H^{4\kappa}((0,T);L^2 (\Gamma))}+  \norm{u}_{L^2((0,T); H^{4\kappa} (\Gamma))}.
    \end{align*}
Further, we introduce the function space $H_{00}^{4\kappa}$ by  
	\begin{align}
		H_{00}^{4\kappa}:= \left\{f\in H^{4\kappa}((0,T)\times \Gamma): \partial^m_t f(0,\cdot) =0 \quad \mbox{for \, $0\le m \le 4\kappa-1$} \right\}.
	\end{align}
    Let $\delta>0$ be sufficiently small. For $\kappa > \frac{n+1}{2}$, define
\begin{align*}
\mathcal{H}
:= \left\{ f \in H_{00}^{4\kappa} \;:\;
\operatorname{supp} f \subset (0,T)\times \Gamma,\;
\|f\|_{H^{4\kappa}((0,T)\times \Gamma)} < \delta \right\}.
\end{align*}
With these definitions, the initial boundary value problem \eqref{main_IBVP} is well posed in the sense that, for each $f \in \mathcal{H}$, there exists a unique solution $u \in H_{00}^{4\kappa}$ depending smoothly on $f$. In particular, the source-to-solution map \eqref{eq_sos_map} is well defined and smooth. The details of this well-posedness result are provided in Section \ref{Well_posedness}. We now proceed to state our first main result.
    
	\begin{theorem}\label{th: Eulidean}
		Let $T>0$, $M\subset \mathbb{R}^n, n\geq 2$ be a convex and compact set with a smooth boundary $\partial M$. Let $\Gamma\subset M$ be a neighborhood of $\partial M$. 
        Furthermore, we assume $\beta^{(k)}$ is non-vanishing almost everywhere in the support of $A^{(k)}$ and $q^{(k)}$ for $k=1,2$, then one has 
		\begin{align}
			L_{A^{(1)},q^{(1)},\beta^{(1)}}f= L_{A^{(2)},q^{(2)},\beta^{(2)}}f, \quad \forall\ f\in \mathcal{H} \implies  (A^{(1)},q^{(1)},\beta^{(1)})=(A^{(2)},q^{(2)},\beta^{(2)}) \ \mbox{in} \ (0,T) \times M. 
		\end{align}
	\end{theorem}
\noindent 
Our second main result is formulated in a Riemannian geometric setting. To this end, we introduce the class of Riemannian manifolds, referred to as admissible manifolds, on which the result is established. The following definition coincides with that in \cite{LOSST_NLS_2025}.
Throughout, let $(M,g)$ denotes a Riemannian manifold, and $T_pM$ its tangent space at a point $p \in M$. For $\xi \in T_pM$, we denote by $\gamma_{p,\xi}$ the geodesic issued from $p$ in the direction $\xi$. We assume that each $\gamma_{p,\xi}$ is non-trapping, non-tangential, and does not self-intersect at $p \in M$.

\begin{definition}[\cite{LOSST_NLS_2025}]\label{def_admissible}
We say that a Riemannian manifold $(M, g)$ is admissible if every point $p\in M$ is generated by admissible geodesics. A point $p\in M$ is said to be generated by admissible geodesics if there exist unit tangent vectors $\xi_1, \xi_2 \in T_pM$ with $\lr\xi_1,\xi_2 \rn_g = 0$ such that 
$$\gamma_{p,\xi_\lambda} \cap \gamma_{p,\xi_1} \cap \gamma_{p,\xi_2} = \{p\},$$
for any $ \displaystyle \xi_\lambda = \lambda'\xi_1 + \lambda\xi_2$, where $\lambda > 0$ (small real number) and $\lambda'^4 = 1-\lambda^4$.
\end{definition}
\noindent One may check that any compact domain in $\mathbb{R}^n$ and simple manifolds are examples of admissible manifolds. For more details on this, we refer the author to check the discussion given in \cite[Definition 1.1]{LOSST_NLS_2025}.
However, we note that not every admissible manifold needs to be a simple manifold or compact domain in $\mathbb{R}^n$. For example, given any finite interval \( I \), the cylinder \( M = I \times \mathbb{S}^1 \) is admissible, but it has trapped geodesics and conjugate points; see, for instance, \cite{PSU_book}. 

We now introduce a coordinate-invariant nonlinear fourth-order Schr\"odinger operator on manifolds. Let $T>0$ and $(M,g)$ be a compact Riemannian manifold with $\Gamma \subset M$ be a neighbourhood of $\p M$. For a positive definite symmetric metric $h:=(h_{ij})_{1\leq i,j \leq n}$ with $h_{ij}\in C_c^{\infty}\left((0,T)\times M\setminus\Gamma\right)$, potential $q\in C^\infty_0((0,T)\times M\setminus\Gamma)$, and a coupling coefficient $\beta\in C^\infty_0((0,T)\times M\setminus\Gamma)$ such that $\beta$ is non-zero almost everywhere in $\supp(h)$ and $\supp(q)$. In this setting, we consider the following initial-boundary value problem (IBVP) on $(0, T) \times M$:
	\begin{align}\label{IBVP_manifold}
		\begin{cases}
			(\I \p_t + \Delta^2_g + \Delta_h + q)u + \beta u^2= f & \text{in $(0,T) \times M$,}
			\\
			u = \p_{\nu_g} u=0 \qum{on} \quad  (0,T)\times \p M,
			\\
			u(0,x)= 0 \hspace{4mm} \qum{in $\Omega$}.
		\end{cases}
	\end{align}
\begin{theorem}\label{th:Manifold}
		Let $(M, g)$ be an admissible manifold and $\Gamma$ be a neighborhood of $\partial M$. 
     Suppose that $\beta_k$ is non-zero everywhere in the support of $q_k$ and $h_{k}$ for $k=1,2$. Then one has 
		\begin{align}
			L_{h_{1},q_1,\beta_1}f= L_{h_{2},q_2,\beta_2}f \quad \forall  f\in \mathcal{H}  \implies (h_{1},q_1,\beta_1)=(h_{2},q_2,\beta_2) \quad \mbox{in} \ (0,T) \times M.
		\end{align}
	\end{theorem}
\noindent To the best of our knowledge, this is the first work in which the unique recovery of a symmetric two-tensor is achieved via a higher-order linearization method, where the second-order term corresponds to a linear perturbation of the operator $i\partial_t + \Delta^2$. Motivated by \cite{LOSST_NLS_2025}, we study an inverse problem for a time-dependent nonlinear fourth-order Schr\"odinger equation posed on smooth, bounded domains in $\mathbb{R}^n$, as well as on compact manifolds. We work in the \emph{source-to-solution} map framework, which, in many settings, is equivalent to the more commonly used Dirichlet-to-Neumann map formulation, and either viewpoint can be adopted. The equation \eqref{main_IBVP} appears naturally in a range of physical models, including the theory of thin elastic plates, nonlinear fiber optics, and the propagation of optical solitons in gyrotropic media; see \cites{KARPMAN2000194,Karpman1995PRL,KARPMAN1994355} for background.

We briefly outline the method of proof. The argument relies on a second-order linearization, which naturally produces terms involving products of three solutions to the associated linearized equation $ \left(\I \partial_t +\Delta^2 -\nabla\cdot(A(t,x)\nabla) +q(t,x)\right)(\cdot) =0$. This creates an additional layer of difficulty, since one must construct suitable approximate geometric optics (GO) solutions for this higher-order operator; we refer to Section \ref{e_s3} for the details. 
The key idea, in the spirit of \cite{LOSST_NLS_2025}, is to carefully combine three such GO solutions (or quasimodes in the manifold setting) at the level of second-order linearization. Each of these solutions is concentrated along a single ray (or geodesic), and by appropriately choosing them, we obtain an integral identity that is effectively localized near a single point. In this way, although three different rays are involved, their interaction is concentrated near a common point, say $p$. By exploiting homogeneity and employing an averaging argument, the problem reduces to a  \emph{divergent beam transform} along a single ray, which is known to be injective on tensor fields, see for instance \cite{kuchment_cone_trans}. To the best of our knowledge, this is the first instance in which a symmetric two-tensor field is uniquely recovered via a higher-order linearization method in which the second-order term remains linear. The same general strategy extends to the manifold setting, although the analysis becomes more delicate; see Section \ref{sec:geometric case} for a detailed treatment.
\subsection{State of the art.}	
 The problem of determining coefficients in the Schr\"odinger equation with time-dependent coefficients has been a central topic of research over the past few decades. The inverse problem considered in this work can be viewed as a natural extension of Calder\'on's inverse conductivity problem \cite{Calderon_problem}.  In his seminal work \cite{Calderon_problem}, Calder\'on formulated the inverse problem of recovering the conductivity $\gamma$ in the equation $\nabla \cdot (\gamma \nabla u) = 0 \quad \text{in } \Omega, \  u = f  \text{ on } \partial \Omega,$ from the knowledge of the Dirichlet-to-Neumann map (also known as the DN map)
$\gamma \mapsto \Lambda_{\gamma}:=\gamma \, \partial_{\nu} u \big|_{\partial \Omega}.$
He established uniqueness for the linearized problem near constant conductivities. Subsequently, in another pioneer work \cite{sylvester1987global}, Sylvester and Uhlmann solved Calder\'on's problem for smooth conductivities by reducing the conductivity equation to a Schr\"odinger equation of the form $(-\Delta + q)u = 0,$ assuming $\gamma \in C^2(\overline{\Omega})$. For further classical results, we refer to the survey \cite{Uhl_survey} and the references cited therein.

The study of inverse problems for determining time-dependent potentials in the linear Schr\"odinger equation has attracted considerable attention from many authors. A classical result goes back to the work of Eskin \cite{Eskin_JMP}, who first proved the unique recovery of a time-dependent electromagnetic potential from the Dirichlet-to-Neumann map. Building on this result, several stability estimates have been obtained. These include logarithmic stability results by A\"{\i}cha \cite{Ibtissem_dynamical_Schroedinger} and Choulli et al. \cite{CKS_SIAM}, as well as H\"older-type stability estimates by Kian and Soccorsi \cite{Kian_Soccorsi_Holder_stability_Scrodinger}. Further developments include the work of Kian and Tetlow \cite{Kian_Tetlow_Holder_stability_dynamical_Scrodinger}, who proved H\"older-type stability for the dynamical anisotropic Schr\"odinger equation, and Bellassoued and Fraj \cite{Bellassoued_Ben_IPI}, who established partial data recovery results with logarithmic and double-logarithmic stability. All the aforementioned results fall into the category of overdetermined problems, which require repeated observations. We also refer to the work \cite{Baudouin_schrodinger}, where it is shown that one can recover the potential from boundary measurements without repeated observations, under suitable geometric assumptions related to exact controllability. Finally, we mention \cites{Klibanov,Rakesh_Salo_IP,Rakesh_Salo_Siam,BansalKrishnanPattar} for further related results.


The literature on inverse problems related to the higher order Schr\"odinger operator is very limited. The inverse problems for biharmonic and polyharmonic operators  (stationary case) were initiated in the works \cites{KRU1,KRU2}. After that, several authors have investigated the inverse problem for biharmonic and polyharmonic operators. In some recent works \cites{BKS_MRT_polyharmonic,SS_linearized,BKSU_biharmonic_nonlinear} the authors introduced a generalized momentum ray transform (MRT) techniques to deal with inverse problems for  linear and nonlinear (both) polyharmonic operators.  
We refer to the related works where the authors have studied various inverse problems for biharmonic and polyharmonic operators   \cites{
BG_JFAA,BG_MAA,GK_applicable_analysis,Bhattacharyya-Kumar,KMS_fractional_nonlinear,AJS_bihar_nonlinear} and the references cited therein.

In recent years, inverse problems for nonlinear PDEs have attracted significant attention, leading to a substantial and growing body of literature. This line of research can be traced back to the foundational work of Krylov, Lassas, and Uhlmann \cite{KLU_invention}, where an inverse problem for a nonlinear wave equation was solved using higher-order linearization and the nonlinear interaction of non-smooth plane waves.
Related approaches appear in \cites{FO_jde,FO_semilinear_elliptic,Hintz_survey,Uhl_Zhai_math_annalen}, where quasimode constructions and microlocal techniques are employed to study nonlinear inverse problems. An alternative perspective was developed by Stefanov at el. \cites{Barreto_Stefanov_nonlinear,Nikolas_Stefanov_nonlinear}, who introduced nonlinear geometric optics methods to address inverse problems for wave equations.
The higher order linearization method for elliptic PDEs was introduced in the work \cite{LLLS_JMPA}. Subsequently, this method was extended in the works \cites{LLLS_Revista,LLS_Math_ann,LLST_fractional_power,FO_jde,Kruchyk_Uhl,Krupchyk_Uhlmann_magnetic,Kian_Kru_Uhlmann,lai2024partialdatainverseproblems,Lai_Zhou}. In all the above works, the authors have considered a nonlinearity $a(x,u)$ under some analytic assumption in the $u$ variable, for example  $ a(x,u)=q(x) u^{k}$, $k\ge 2$.
The study of inverse problems involving nonlinearities via first-order linearization dates back to the work of Isakov \cite{isakov1993uniqueness} for parabolic equations. In \cite{isakov1993uniqueness}, nonlinear inverse problems for parabolic equations were investigated under the assumptions that $\partial_u a(x,u) > 0$ and $a(x,0) = 0$. The condition $a(x,0)=0$ ensures that $u=0$ is a common solution to the nonlinear equations under consideration. Removing this assumption in one of the equations leads to a natural gauge invariance in the recovery of the nonlinear term; see \cite{Sun_EJDE} for further discussion.
In a recent work \cite{JNS2023}, both assumptions $\partial_u a(x,u) > 0$ and $a(x,0)=0$ were removed, and the authors established recovery of a general nonlinear term up to a natural gauge. This approach has been further extended in \cites{JNS_low_regu,nurminen_sahoo}. Another important class of nonlinear inverse problems arises in the study of porous medium equations, where the nonlinearity appears in the highest-order term (e.g., equations of the form $\partial_t u + \Delta(u^m) + \cdots = 0$); see, for instance, \cites{Catalin_Ghosh_Uhlmann_poros_media,Catalin_Ghosh_Nakamura_aniso_poros_media}. More broadly, there is a substantial body of literature on inverse problems for nonlinear PDEs; we refer to \cites{Isakov_Sylvester_CPAM,Isakov_Nachman_TAMS,Janne_ms_nonlinearity,Janne_ms_nonlinear_ana,LLPT_apde,Kian_forum,FATY_ann_pde,CLLO_jde,Kian_Uhl_arma,Lai_arma} and the references therein for further developments. 
The rest of the article is organized as follows. Section \ref{Well_posedness} establishes the well-posedness of the IBVP associated with the considered nonlinear partial differential equation, which is used to ensure that the source-to-solution map is well-defined. In Section \ref{e_s3}, we construct geometric optics solutions to the linear fourth-order Schr\"odinger equation. Section \ref{main_theorem} contains the proof of the main theorem. Section \ref{sec:geometric case} deals with the Gaussian beam construction and related results. Finally, we present the proof of Theorem \ref{th:Manifold} in Section \ref{sec:main_result_thm_1.2}.
\vspace{-3mm}
\section{Well-posedness for nonlinear equations and the source-to-solution map} \label{Well_posedness}
\noindent This section aims to show that the source-to-solution map, corresponding to the IBVP for the nonlinear fourth-order Schr\"odinger PDE \eqref{main_IBVP}, is well-defined and smooth in a neighborhood around zero. More precisely, we prove the following:
\begin{theorem}[Well-posedness]\label{e_prop}
Let $A,q,\beta$ be as in equation \eqref{main_IBVP}. For every given $f\in\mathcal{H}$, there exists a unique solution $u\in H^{4\kappa}_{00}$ satisfying $u=L_{A,q,\beta}f$. Moreover, the operator $L_{A,q,\beta}:\mathcal{H}\rightarrow H^{4\kappa}_{00}$ is smooth.
\end{theorem}
\noindent To this result, we will need the well-posedness of the linearized version of equation \eqref{main_IBVP}. We first address this in the next subsection. 
\subsection{Well-posedness of the linearized equation}
We start with a general inhomogeneous linear fourth-order Schr\"odinger equation
\begin{equation}\label{linear_main_IBVP}
\begin{aligned}
\begin{cases}
\left(\I \partial_t +\Delta^2 +L\right)u(t,x)=f(t,x) \quad &\mbox{in} \quad (0,T) \times M,\\
u(t,x)= \frac{\partial u}{\partial \nu}(t,x) = 0, \quad & \mbox{on} \quad (0,T)\times \partial M,\\
u(0,x)=0 \quad & \,\mbox{in} \quad\, x\in M,
\end{cases}
\end{aligned}
\end{equation}
where $L$ is a second-order elliptic self-adjoint operator with coefficients in $C_c^{\infty}\left((0,T)\times M\setminus\Gamma\right)$. 

\noindent As mentioned before, we aim to show that the initial boundary value problem \eqref{linear_main_IBVP} is well-posed. We achieve this by using the classical Galerkin approximation argument; a similar approach can be found in \cites{Kian_Soccorsi_Holder_stability_Scrodinger, Evans_book_KAM}. For the well-posedness result (see the lemma discussed below), we do not need $C_c^{\infty}\left((0, T)\times M\setminus\Gamma\right)$ regularity of the coefficients; however, for simplicity, we took the coefficients to be $C_c^{\infty}\left((0, T)\times M\setminus\Gamma\right)$.  
\begin{lemma}
\tblue{With notations as in equation \eqref{linear_main_IBVP} and }let $f\in H^1(0, T; L^2(M))$ satisfies $f(0,\cdot)=0$ almost everywhere in $M$. Then there exists a unique solution $v\in \mathcal{C}([0, T]; H^2_0(M)\cap H^4(M))\cap \mathcal{C}^1([0, T]; L^2(M)) $ of the system \eqref{linear_main_IBVP} satisfying the following estimate
\begin{align}\label{stability_estimate}
\|v\|_{\mathcal{C}([0, T]; H^4(M))}+\|v\|_{\mathcal{C}^1([0, T]; L^2(M))}\leq C\|f\|_{H^1(0,T;L^2(M))}.
\end{align}
\end{lemma}
\begin{proof}
The proof of this result is a bit long, and therefore, we divided it into several steps.
\subsubsection*{\textbf{Step 1: (Set up)}}
We consider the sesquilinear form associated with IBVP \eqref{linear_main_IBVP}
\begin{align}\label{sesquilinear_form}
a(t;u,v):= \int_{M} \Delta u(t,x) \ \overline{\Delta v(t,x)} \,\df x + \int_M Lu(t,x) \ \overline{v(t,x)} \,\df x, \quad u,v \in H^2_{0}(M).
\end{align}
Then we have
\begin{align}\label{estimate_sesquilinear}
a(t;u,u) + \delta \|u\|^2_{L^2(M)}\geq \alpha \|u\|^2_{H^2(M)},\quad \delta, \alpha >0,  u\in H^2_{0}(M), t\in (0,T).
\end{align}
\noindent For a Hilbert basis $\{e_k : k\in \mathbb{N}\}$ of $H^2_{0}(M)$, consider an approximated solution of size $m\in \mathbb{N}$ for the equation \eqref{linear_main_IBVP}
\begin{align}\label{approx_solution}
v_m(t,x) : = \sum_{k=1}^{m} g_{k,m}(t)e_k(x), \quad (t,x)\in (0,T)\times M.
\end{align}
 The functions $g_{k,m}$, appearing in the above relation, are chosen so that the function $v_m$ satisfies the following initial value problem:
\begin{align}\label{approx_v_mequation}
\begin{cases}
\I\langle \partial_tv_m(t,\cdot),e_k\rangle_{L^2(M)} + a(t;v_m(t,\cdot),e_k) = \langle f(t,\cdot),e_k\rangle_{L^2(M)}, \quad t\in (0,T),\\
v_m(0,\cdot) = 0
\end{cases}
\end{align}
for all $k=1,2,\dots,m$.
\subsubsection*{\textbf{Step 2}:} In this step, our goal is to prove that the equation \eqref{approx_v_mequation} admits a unique solution $v_m\in W^{1,\infty}(0,T;H^2_0(M))$. We begin by showing that a solution $v_m$ of \eqref{approx_v_mequation}  belongs to   $L^{\infty}(0,T;L^2(M))$. For $k\in\{1,2,\dots,m\}$ and a fixed $t\in(0,T)$, multiply the equation \eqref{approx_v_mequation} by $\overline{g_{k,m}(t)}$ and sum over $k=1,2,\dots,m$ to get
$$\I\langle \partial_tv_m(t,\cdot),v_m(t,\cdot)\rangle_{L^2(M)} + a(t;v_m(t,\cdot),v_m(t,\cdot)) = \langle f(t,\cdot),v_m(t,\cdot)\rangle_{L^2(M)}, \quad t\in (0,T).$$
Equating the imaginary parts of this relation, we have
$$\frac{d}{ds}\|v_m(s,\cdot)\|^2_{L^2(M)}=2 \mathfrak{Im} \langle f(s,\cdot),v_m(,s\cdot)\rangle_{L^2(M)},\quad s\in (0,T).$$
Integrating the above with respect to $s$ over the interval $(0, t)$ and  using $v_m(0,\cdot) = 0,$ we conclude
$$\|v_m(t,\cdot)\|^2_{L^2(M)}=2 \mathfrak{Im}\int_0^t \langle f(s,\cdot),v_m(s,\cdot)\rangle_{L^2(M)} \,\df s \leq \int_0^t \|f(s,\cdot)\|^2_{L^2(M)}\,\df s + \int_0^t \|v_m(s,\cdot)\|^2_{L^2(M)}\,\df s.$$
Using Gr\"onwall's inequality \cite[Appendix B]{Evans_Book}, we obtain a constant $C_0$ that depends only on $T$ and $M$ such that
\begin{align}\label{estimate_L_2_norm}
\|v_m\|_{L^\infty(0,T;L^2(M))}\leq C_0\|f\|_{L^2((0,T)\times M)}.
\end{align}
\tblue{This proves that $v_m$ is in $L^{\infty}(0,T;L^2(M))$ and to prove the boundedness of the derivative, we consider the function $w_m:=\partial_tv_m$}. Observe that this $w_m$ solves the following Cauchy problem for $k=1, \dots, m$:
\begin{align}\label{approx_equation}
\begin{cases}
\I\langle \partial_tw_m(t,\cdot),e_k\rangle_{L^2(M)} + a(t;w_m(t,\cdot),e_k) = a'(t;v_m(t,\cdot),e_k)+\langle \partial_tf(t,\cdot),e_k\rangle_{L^2(M)}, \quad t\in (0,T),\\
w_m(0,\cdot) = 0.
\end{cases}
\end{align}
Here
\begin{align}\label{a_prime}
a'(t;u,v):=   -\int_M (\partial_tL)u(t,x)\ \overline{v(t,x)} \,\df x, \mbox{ for } t\in(0,T) \mbox{ and for } u, v\in H^2_0(M),
\end{align}
where \(\partial_t L\) is the same operator as \(L\), but with the coefficients changed to the time derivative of the coefficients of \(L\). Next, we multiply $\overline{g'_{k,m}(t)}$ to the first line of \eqref{approx_equation} and then sum over $k = 1,\dots,m$ to obtain
$$\I\|\partial_tv_m(t,\cdot)\|^2_{L^2(M)}+a(t;v_m(t,\cdot),\partial_tv_m(t,\cdot)) = \langle f(t,\cdot),\partial_tv_m(t,\cdot)\rangle_{L^2(M)}, \quad t\in (0,T).$$
\tblue{This time, we equate the real parts of the above identity and get}
$$a(s;v_m(s,\cdot),\partial_tv_m(s,\cdot))+a(s;\partial_tv_m(s,\cdot),v_m(s,\cdot))=2\mathfrak{Re}\langle f(s,\cdot),\partial_tv_m(s,\cdot)\rangle_{L^2(M)},  \quad s\in (0,T). $$
We can rewrite this as follows:
\begin{align}\label{estimate_a}
\frac{d}{ds}a(s;v_m(s,\cdot),v_m(s,\cdot))  &= -a'(s;v_m(s,\cdot),v_m(s,\cdot)) + 2\mathfrak{Re}\langle f(s,\cdot),\partial_tv_m(s,\cdot)\rangle_{L^2(M)},  \quad s\in (0,T)\nonumber \\
\Longrightarrow 	a(t;v_m(t,\cdot),v_m(t,\cdot)) &=-\int_{0}^{t}a'(s;v_m(s,\cdot),v_m(s,\cdot))\,\df s+2\mathfrak{Re}\int_0^t \langle f(s,\cdot),\partial_tv_m(s,\cdot)\rangle_{L^2(M)}\,\df s.
\end{align}
We use $\displaystyle \int_0^t \langle f(s,\cdot),\partial_tv_m(s,\cdot)\rangle_{L^2(M)} \,\df s = \langle f(t,\cdot),\partial_tv_m(t,\cdot)\rangle_{L^2(M)}-\int_0^t \langle \partial_tf(s,\cdot),v_m(s,\cdot)\rangle_{L^2(M)} \,\df s$ to get
\begin{align*}
\left|\int_0^t \langle f(s,\cdot),\partial_tv_m(s,\cdot)\rangle_{L^2(M)} \,\df s\right|\leq& \|f(t,\cdot)\|_{L^2(M)}\|v_m(t,\cdot)\|_{L^2(M)}+\int_0^t \|\partial_t f(s,\cdot)\|_{L^2(M)}\|v_m(s,\cdot)\|_{L^2(M)}\,\df s. 
\end{align*}
With the help of equations \eqref{estimate_sesquilinear}, \eqref{a_prime}, and \eqref{estimate_a}, we get the following upper-bound of $\|v_m(t,\cdot)\|^2_{H^2(M)}$
\begin{align*}
& C\left(\int_0^t \|v_m(s,\cdot)\|^2_{H^1(M)} \,\df s + \int_0^t \|\partial_tf(s,\cdot)\|^2_{L^2(M)} \,\df s  \right) + \frac{2}{\alpha} \|f(t,\cdot)\|_{L^2(M)}\|v_m(t,\cdot)\|_{L^2(M)}+ \frac{\delta}{\alpha}\|v_m(t,\cdot)\|^2_{L^2(M)}, 
\end{align*}
where $C$ is a positive constant. Using this upper bound, inequalities
\begin{align*}
&\|f(t,\cdot)\|_{L^2(M)}\|v_m(t,\cdot)\|_{L^2(M)}\leq \frac{\alpha}{4}\|v_m(t,\cdot)\|^2_{L^2(M)}+\frac{1}{\alpha}\|f(t,\cdot)\|^2_{L^2(M)},\\
&\|f(t,\cdot)\|_{L^2(M)} = 2\mathfrak{Re}\int_0^t \langle f(s,\cdot),\partial_tf(s,\cdot)\rangle_{L^2(M)} \,\df s\leq  \int_0^t \left(\|f(s,\cdot)\|^2_{L^2(M)} + \|\partial_tf(s,\cdot)\|^2_{L^2(M)}  \right)\,\df s,
\end{align*}
and the estimate \eqref{estimate_L_2_norm}, we get
\begin{align*}
\|v_m(t,\cdot)\|^2_{H^2(M)} \leq 2C \int_0^t \|v_m(s,\cdot)\|^2_{H^1(M)} \,\df s + 2\left(C+\frac{2}{\alpha^2}+ \frac{C_0^2}{2}+ \frac{\delta C_0^2}{\alpha}\right)\|f\|^2_{H^1(0,T;L^2(M))}.
\end{align*}
Next, we apply Gr\"onwall's inequality to find a positive constant $C'=C'(T,\alpha )$ such that
\begin{align}\label{estimate_H_2_norm}
\|v_m\|_{L^\infty(0,T;H^2(M))}\leq C'\|f\|_{H^1(0,T;L^2(M))}.
\end{align} 
To gain more regularity in the time variable, we rewrite \eqref{approx_equation} as follows:
\begin{align}\label{new_approx_equation}
\begin{cases}
\I\langle \partial_tw_m(t,\cdot),e_k\rangle_{L^2(M)} + a(t;w_m(t,\cdot),e_k) = \langle f_m(t,\cdot),e_k\rangle_{L^2(M)}, \quad t\in (0,T),\\
w_m(0,\cdot) = 0
\end{cases}
\end{align}
for all $k=1,\dots,m$, where $ f_m(t,x) :=- (\partial_tL)v_m(t,x) + \partial_tf(t,x), \ (t,x)\in(0,T)\times M.$
\vspace{2mm}\\
Following a similar strategy as above, multiply $\overline{g'_{k,m}(t)}$ to the first line of \eqref{new_approx_equation} and summing over $k=1,\dots,m$ yields
$$\I\langle \partial_tw_m(t,\cdot),w_m(t,\cdot)\rangle_{L^2(M)} + a(t;w_m(t,\cdot),w_m(t,\cdot)) = \langle f_m(t,\cdot),w_k(t,\cdot)\rangle_{L^2(M)}, \quad t\in (0,T).$$
Taking the imaginary part of both sides gives
$$\frac{d}{ds}\|w_m(s,\cdot)\|^2_{L^2(M)} = 2\mathfrak{Im}\langle f_m(s,\cdot),w_m(s,\cdot)\rangle, s\in(0,T). $$
Integrating with respect to $s$ over $(0, t)$, we get 
$$\|w_m(t,\cdot)\|^2_{L^2(M)} \leq \int_0^t \|w_m(s,\cdot)\|^2_{L^2(M)} \, ds + \int_0^t \|f_m(s,\cdot)\|^2_{L^2(M)} \,\df s,$$
for each $t\in(0, T)$ and combining this with \eqref{a_prime}, gives
$$\|w_m(t,\cdot)\|^2_{L^2(M)} \leq \int_0^t \|w_m(s,\cdot)\|^2_{L^2(M)} \,\df s + C\left( \|\partial_t f\|^2_{L^2((0,T)\times M)}+ \|v_m\|^2_{L^{\infty}(0,T;H^2(M))}\right),\ t\in(0, T).$$
Applying estimate \eqref{estimate_H_2_norm} together with Gr\"onwall's lemma, we have
\begin{align}\label{estimate_w_m_norm}
\|w_m\|_{L^{\infty}(0,T;L^2(M))}\leq C \|f\|_{H^1(0,T;L^2(M))},
\end{align}
where $C$ is another positive constant that depends only on $T, \alpha$. This completes the goal of \textit{\textbf{Step 2.}}

\subsubsection*{\textbf{Step 3}} \tblue{This is the final step, where we show the existence of a solution to our linearized IBVP \eqref{linear_main_IBVP}}. We will be using the estimates \eqref{estimate_H_2_norm} and \eqref{estimate_w_m_norm} to achieve our aim. Note from equation \eqref{estimate_H_2_norm}, the sequence $v_m$ is uniformly bounded in $L^\infty(0,T;H_0^2(M))$. Since this space is the dual of $L^1(0,T;H^{-2}(M))$, the Banach-Alaoglu theorem guarantees the existence of a subsequence $v_{m'}$ and a function $ v \in L^\infty(0,T;H_0^2(M))$ such that $v_{m'} \overset{*}{\rightharpoonup} v \ \text{in} \  L^\infty(0,T;H_0^2(M))$. Next, each $v_{m'}$ satisfies the Galerkin formulation given in equation \eqref{approx_v_mequation}. Passing to the limit as $m'\to\infty$ and using the weak-* convergences, we obtain 
$\displaystyle (i\partial_t + \Delta^2 + L)v = f \ \text{in } \ L^\infty(0,T;H^{-2}(M))$ together with the initial condition $v(0)=0$, since $v_{m'}(0)=0$ for all $m'$ and the weak-* limit preserves the value at $t=0$. In the weak formulation, $v$ satisfies
\begin{align}\label{v_solution}
\left(\I \partial_t +\Delta^2 +L\right)v=f \quad &\mbox{in} \quad (0,T) \times M,\\
v(0,\cdot)=0 \quad & \,\mbox{in} \quad\, M.
\end{align}
Furthermore, from estimate \eqref{estimate_w_m_norm} the sequence $w_m=\partial_t v_m$ is uniformly bounded in $L^\infty(0,T;L^2(M))$. The Banach-Alaoglu theorem gives a subsequence $w_{m'}$ and a function $w\in L^\infty(0,T;L^2(M))$ such that $w_{m'} \overset{*}{\rightharpoonup} w$. Since $w_{m}=\partial_t v_{m}$ for all $m$, it follows that $ \partial_t v = w \in L^\infty(0,T;L^2(M))$.
Consequently, $v \in L^\infty(0,T;H_0^2(M))\cap W^{1,\infty}(0,T;L^2(M))$. Next, by applying the same argument as in the derivation of \cite[Theorem 8.3, Remark 10.2, Chapter 3]{Lion-Magenes2}, the regularity $v\in L^\infty(0, T; H_0^2(M))$ and $\partial_tv\in L^\infty(0, T; L^2(M))$ together with $H_0^2(M)$ is continuously and densely embedded in $L^2(M)$, implies that $v$ (after possibly being redefined on a set of measure zero) is continuous function of $[0, T]\to H_0^2(M)$ such that $v\in \mathcal{C}([0, T]; H^2_0(M))$. Moreover, equation \eqref{new_approx_equation} gives $\partial^2_tv\in L^2(0, T; H^{-2}(M))$, combined with $\partial_tv\in L^\infty(0, T; L^2(M))$, the same Lions-Magenes result yields $\partial_tv \in \mathcal{C}([0, T]; L^2(M))$. Hence, we conclude $v\in \mathcal{C}([0, T]; H^2_0(M))\cap \mathcal{C}^1([0, T]; L^2(M))$.
Therefore, we have from equation \eqref{v_solution} that for each $t\in [0, T]$, $v(t,\cdot)$ is the solution to the fourth-order boundary value problem 
\begin{align*}
\left( \Delta^2 +L\right)v(t,\cdot)=f(t,\cdot)-\I \partial_t v(t,\cdot)\quad &\mbox{in} \quad  M,\\
v(t,\cdot)=  \frac{\partial v}{\partial \nu}(t,\cdot) = 0, \quad & \mbox{on} \quad  \partial M.
\end{align*}
Since $f-\I \partial_t v \in \mathcal{C}([0, T]; L^2(M))$, this implies $v\in \mathcal{C}([0, T]; H^2_0(M)\cap H^4(M))$ and we conclude our claim directly from this and equations \eqref{sesquilinear_form}, \eqref{estimate_H_2_norm} and \eqref{estimate_w_m_norm}. This completes the proof of existence part, and the uniqueness follows by considering the energy $\displaystyle E(t) =\int_{M} u\, \Bar{u}\, \df x$, and showing that $E'(t)=0$.
\end{proof}
\subsection{Source-to-solution map}
To prove Theorem \ref{e_prop}, we need higher-order energy estimates for the linearized problem, which is encoded in the following lemma.
\begin{lemma}
For each non-negative integer $\ell\geq 0$, one has the following estimate 
\begin{align}\label{Hiher_energy_estimate}
\|\partial^\ell_t v_m\|_{L^{\infty}(0,T;H^2(M))}\leq C \|f\|_{H^{\ell+1}(0,T;L^2(M))},
\end{align}
where $v_m$ is given by the relation \eqref{approx_solution} and solves the IBVP \eqref{approx_v_mequation}.
\end{lemma}
\begin{proof}
\tblue{The proof of this lemma is also along the same lines as we have seen above, together with mathematical induction.} For $\ell=0$ this result follows directly from equation \eqref{estimate_H_2_norm}. In order to prove the estimate for $\ell=1$, let us multiply the first line of equation \eqref{approx_equation} by $\overline{g''_{k,m}(t)}$ and summing up over $k=1,\dots,m$ to get
\begin{align*}
\I\langle \partial_tw_m(t,\cdot),\partial_tw_m(t,\cdot)\rangle_{L^2(M)} + a(t;w_m(t,\cdot),\partial_tw_m(t,\cdot)) &= a'(t;v_m(t,\cdot),\partial_tw_m(t,\cdot))\\ &\qquad+\langle \partial_tf(t,\cdot),\partial_tw_m(t,\cdot)\rangle_{L^2(M)}, &\quad t\in (0,T).
\end{align*}
Next, considering the real part of both sides of the above identity gives
\begin{align*}
&a(t;w_m(t,\cdot),\partial_tw_m(t,\cdot)) + a(t;\partial_tw_m(t,\cdot),w_m(t,\cdot)) \\ &\qquad \qquad = a'(t;v_m(t,\cdot),\partial_tw_m(t,\cdot)) + a'(t;\partial_tw_m(t,\cdot),v_m(t,\cdot)) + 2\mathfrak{Re} \langle \partial_tf(t,\cdot),\partial_tw_m(t,\cdot)\rangle_{L^2(M)}.
\end{align*}
This identity can be rewritten in the following form
\begin{align*}
\frac{d}{ds} a(s;w_m(s,\cdot),w_m(s,\cdot)) &= -a'(s;w_m(s,\cdot),w_m(s,\cdot)) + a'(s;v_m(s,\cdot),\partial_tw_m(s,\cdot)) \\&\qquad + a'(s;\partial_tw_m(s,\cdot),v_m(s,\cdot)) + 2\mathfrak{Re} \langle \partial_tf(s,\cdot),\partial_tw_m(s,\cdot)\rangle_{L^2(M)}.
\end{align*}
Integrating with respect to $s$ over the interval $(0, t)$ yields
\begin{align*}
a(t;w_m(t,\cdot),w_m(t,\cdot)) &= -\int_0^t a'(s;w_m(s,\cdot),w_m(s,\cdot))\,\df s + \int_0^t a'(s;v_m(s,\cdot),\partial_tw_m(s,\cdot))\,\df s \\&\qquad+ \int_0^t a'(s;\partial_tw_m(s,\cdot),v_m(s,\cdot)) \,\df s+ 2\mathfrak{Re} \int_0^t \langle \partial_tf(s,\cdot),\partial_tw_m(s,\cdot)\rangle_{L^2(M)} \,\df s.
\end{align*}
Using equation \eqref{a_prime} and together with estimates \eqref{estimate_H_2_norm} and \eqref{estimate_w_m_norm}, we get
\begin{align*}
a(t;w_m(t,\cdot),w_m(t,\cdot)) \leq  C\int_0^t \|w_m(s,\cdot)\|^2_{H^2(M)}  \,\df s + C' \|f\|^2_{H^2(0,T;L^2(M))}. 
\end{align*}
The above relation together with estimate \eqref{estimate_sesquilinear} gives an upper-bound for $\|w_m(t,\cdot)\|^2_{H^2(M)}$ as follow
\begin{align*}
\|w_m(t,\cdot)\|^2_{H^2(M)} \leq  C_1\int_0^t \|w_m(s,\cdot)\|^2_{H^2(M)}  \,\df s + C_2 \|f\|^2_{H^2(0,T;L^2(M))} + \frac{\delta}{\alpha}\|w_m\|^2_{L^2(M)}, \quad t\in(0, T). 
\end{align*}
From \eqref{estimate_w_m_norm} and Gr\"onwall's inequality we have
\begin{align}\label{new_estimate_w_m_norm}
\|w_m\|_{L^{\infty}(0,T;H^2(M))}\leq C \|f\|_{H^2(0,T;L^2(M))},
\end{align}
where $C$ is another positive constant depending only on $T, \alpha $. This proves the inequality \eqref{Hiher_energy_estimate} for $\ell=1$.
For the general $\ell\geq 2$, one can argue by induction and prove the following estimate
\begin{align*}
\|\partial^\ell_t v_m\|_{L^{\infty}(0,T;H^2(M))}\leq C \|f\|_{H^{\ell+1}(0,T;L^2(M))}.
\end{align*} 
\end{proof}
\noindent Next, we consider inhomogeneous linear fourth-order Schr\"odinger equation related to our main IBVP \eqref{main_IBVP}
\begin{align}\label{e_LS}
\begin{cases}
(\I\partial_t+\Delta^2-\nabla\cdot (A\nabla)+q) u = f & \text{on $(0,T) \times M$,}\\
u|_{x \in \partial M}=\frac{\partial u}{\partial \nu}|_{x \in \partial M} = 0,
\\
u|_{t = 0} = 0.
\end{cases}
\end{align}
The solution operator related to the above equation is denoted by $\mathcal{S}$ and defined as $\mathcal{S}(f):=u$.
Let us recall the following energy estimates for our linearized problem (\ref{e_LS}), which hold under the assumption that $f|_{t=0}=0$: 
\begin{align}\label{e_NRG1}
&\norm{u}_{L^\infty(0,T;L^2(M))}\leq C\norm{f}_{L^2((0,T)\times M)}, \ \ \norm{u}_{L^\infty(0,T;H^2(M))}\leq C\norm{f}_{H^1(0,T;L^2(M))},
\\
\label{e_NRG2}
&\norm{\Delta^2 u}_{L^2((0,T)\times M)}\leq C\norm{f}_{H^{1}(0, T;L^2(M))},\ \ \norm{\partial^m_tu}_{L^\infty(0,T;L^2(M))}\leq C\norm{f}_{H^{m}(0,T;L^2(M))},\\ 
\label{e_NRG3}
&\|\partial^m_tu\|_{L^{\infty}(0,T;H^2(M))}\leq C \norm{f}_{H^{m+1}(0,T; L^2(M))} \ \mbox{for} \ m\geq 1.
\end{align}
Now, we derive higher-order energy estimates for the linearized problem \eqref{e_LS}, which is needed to obtain an existence result for the nonlinear equation \eqref{main_IBVP} and for the fact that $H^{s}((0, T)\times M)$ is an algebra, when $\displaystyle s> \frac{n+1}{2}$. 
\begin{lemma}\label{e_NRG_Higher}
The system \eqref{e_LS} satisfies the following estimate
\begin{equation}\label{e_NRG4}
\norm{u}_{H^{4\kappa}((0,T)\times M)}\leq C\norm{f}_{H^{4\kappa}((0,T)\times M)},
\end{equation}
for any source $f\in H^{4\kappa}_{00}$.
\end{lemma}
\begin{proof}
Let $u$ be a solution of IBVP \eqref{e_LS} then we have $$\I\partial_tu=f-\Delta^2 u + \nabla\cdot(A\nabla u) -qu.$$
Since $f\in H^{4\kappa}_{00}$, this implies that $\partial_t^mf|_{t=0}=0$ for $m \leq 4\kappa-1$. Using this together with $u|_{t=0}=0$, we obtain
\begin{equation}\label{e_IV}
\partial_t^mu|_{t=0}=0,\ {\rm when}\ m\leq 4\kappa.
\end{equation}
For $\kappa=0$, the required estimate holds by \eqref{e_NRG1}.  We next argue by mathematical induction on $\kappa$. Assume that \eqref{e_NRG4} holds true for $\kappa\leq K-1$. Then, our aim is to prove that the estimate holds for $\kappa= K$. In that case, it is enough to show that for $\rho,\sigma\in\N$ such that $\rho+4\sigma=4K$, we have
\begin{equation}\label{aim}\norm{{\partial_t}^\rho\Delta^{2\sigma} u}_{L^2((0, T)\times M)}\leq\norm{f}_{H^{4K}((0, T)\times M)}.\end{equation}
By applying $\partial_t$ to (\ref{e_LS}), we get
\begin{equation}\label{d_t}(\I\partial_t+\Delta^2-\nabla\cdot(A\nabla)+q)\partial_tu= \tilde{f}=\partial_tf+\nabla\cdot((\partial_t A)\nabla u) -(\partial_t q)u.\end{equation}
Since $u|_{t=0}=0, \partial_tf|_{t=0}=0$, this implies $\tilde{f}(0,x)=0$ and we have $\partial_tu|_{t=0} =0$. As a result, we can apply the estimate \eqref{e_NRG2} for $m=1$ to the equation (\ref{d_t})  and deduce
\[\norm{\partial_t^2u}_{L^\infty(0,T;L^2(M))}\leq C\norm{f}_{H^{2,0}((0,T)\times M)}+C\norm{(\nabla\cdot((\partial_t A)\nabla u)}_{H^{1,0}((0,T)\times M)}+C\norm{(\partial_tq)u}_{H^{1,0}((0,T)\times M)}.\]
Using $A_{ij}\in C^\infty_0((0,T)\times M\setminus\Gamma)$, $q\in C^\infty_0((0,T)\times M\setminus\Gamma)$, and estimates \eqref{e_NRG1}, \eqref{e_NRG2}, \eqref{e_NRG3} for $m=1$, we conclude that
\[\norm{\partial_t^2u}_{L^\infty(0,T;L^2(M))}\leq C\norm{f}_{H^{2}((0,T)\times M)}.\]
Next, we apply $\partial_t$ to \eqref{d_t} and use the estimations that were previously generated to get the bound
\[\norm{\partial_t^3u}_{L^\infty(0,T;L^2(M))}\leq C\norm{f}_{H^{3}((0,T)\times M)}.\]
In a similar way, one can prove that
\begin{equation}\label{e_b0}\norm{\partial_t^mu}_{L^\infty(0, T; L^2(M))}\leq C\norm{f}_{H^{m}((0, T)\times M)} \quad \mbox{for all integers  $m\ge 0$}.
\end{equation}
Specifically, we observe that estimate \eqref{e_b0} establishes \eqref{aim} in the case when $\sigma=0$ and that it holds for all $m\leq 4K$, not simply when $m$ is even. Now, let us consider the case in which $\sigma=1$. The fourth-order Schr\"odinger equation is satisfied by $u$, therefore it follows that
\begin{align} \label{e_b1}
\norm{\partial_t^{4K-4}\Delta^2 u}_{L^2((0,T)\times M)}&\leq C\norm{\partial_t^{4K-4}(f-\I\partial_t u+\nabla\cdot( A\nabla u)-qu)}_{L^2((0,T)\times M)} \nonumber \\ &\leq C\norm{f}_{H^{4K-4}((0,T)\times M)}+C\norm{\partial_t^{4K-3}u}_{L^2((0,T)\times M)}\nonumber \\  &\quad+ C\norm{\partial_t^{4K-4}(\nabla\cdot( A\nabla u))}_{L^2((0,T)\times M)} +C\norm{\partial_t^{4K-4}(qu)}_{L^2((0,T)\times M)}.
\end{align}
We infer from \eqref{e_b0} that $\norm{f}_{H^{4K-3}((0, T)\times M)}$ bounded by the second term on the right-hand side of \eqref{e_b1}. Since $A_{ij}\in C^\infty_0((0, T)\times M\setminus\Gamma)$, estimate \eqref{e_NRG3} suggests that the third term is bounded by $\norm{f}_{H^{4K-3}((0, T)\times M)}$, and estimate \eqref{e_NRG2} suggests that the fourth term is bounded by $\norm{f}_{H^{4K-4}((0, T)\times M)}$ as $q\in C^\infty_0((0, T)\times M\setminus\Gamma)$. Consequently, it follows that
\[\norm{\partial_t^{4K-4}\Delta^2 u}_{L^2((0,T)\times M)}\leq C\norm{f}_{H^{4K-3}((0,T)\times M)}.\]
This gives, when $\sigma=1$, (\ref{aim}). Only the situation where $\sigma\geq2$ needs to be addressed. Please note that
\[\begin{split}{\partial_t}^\rho\Delta^{2\sigma} u&=\, {\partial_t}^\rho \Delta^{2(\sigma-1)}(f-\I\partial_tu+\nabla\cdot( A\nabla u)-qu)\\&=\, {\partial_t}^\rho\Delta^{2(\sigma-1)}(f+\nabla\cdot( A\nabla u)-qu)-\I{\partial_t}^{\rho+1}\Delta^{2(\sigma-2)}\Delta^2 u\\&=\, {\partial_t}^\rho\Delta^{2(\sigma-1)}(f+\nabla\cdot( A\nabla u)-qu))-\I{\partial_t}^{\rho+1}\Delta^{2(\sigma-2)}(f-\I\partial_tu+\nabla\cdot( A\nabla u)-qu).\end{split}\]
The induction hypothesis then suggests that since all of the derivatives of $u$ and $f$ in the final expression are of order $4K-2$ or less,
\[\norm{{\partial_t}^\rho\Delta^{2\sigma} u}_{L^2((0,T)\times M)}\leq C\norm{f}_{H^{4K-2}((0,T)\times M)}.\]
Consequently, the proof of (\ref{aim}) for $\sigma\geq2$ is completed.
\end{proof}
\noindent With all this preparation, now we are ready to prove our well-posedness result, Theorem \ref{e_prop}.
\begin{proof}[Proof of Theorem \ref{e_prop}]
First, we address the uniqueness part. Let $u$ and $v$ be two solutions of \eqref{main_IBVP} for a given $f$. This gives the following relation:
$$(i\partial_t+\Delta^2-\nabla\cdot A\nabla + q)(u-v)+\beta(u+v)(u-v)=0.$$ 
We can determine that $u-v=0$ by using the energy estimate \eqref{e_NRG1} for the linear problem. This proves that $L_{A, q, \beta}$ is well-defined. Next, we fix a sufficiently large $\kappa\in\N$ such that $H^{4\kappa}$ is a Banach algebra. The map\[\mathcal{K}: H^{4\kappa}_{00}\times H^{4\kappa}_{00}\rightarrow H^{4\kappa}_{00}\quad \mbox{defined as} \quad \mathcal{K}(u,f)=f-\beta u^2\] is well-defined since $H^{4\kappa}_{00}$ is Banach algebra, as follows from the trace theorem.
We define the map $\phi(u,f):= u - \mathcal{S}\mathcal{K}(u,f)$, and it is noted that $\phi(u,f)=0$ indicates that the nonlinear fourth-order Schr\"odinger equation \eqref{main_IBVP} has a solution $u$. Further, note that $ \displaystyle \mathcal{S}:H^{4\kappa}_{00}\rightarrow H^{4\kappa}_{00}$ by the outcome of Lemma \eqref{e_NRG_Higher} together with \eqref{e_IV}. Therefore, it follows that
\[\phi:H^{4\kappa}_{00}\times H^{4\kappa}_{00}\longrightarrow H^{4\kappa}_{00}.\]
\textcolor{black}{\noindent We observe that the mapping $\phi$ depends smoothly on both $u$ and $f$, because $\mathcal{K}(u,f)$ is polynomial in nature while $\mathcal{S}$ acts linearly. By applying the chain rule, one finds that $\partial_i u (0,0)=\mathrm{Id}$. Hence, the implicit function theorem yields a smooth map $f\mapsto u$, from a neighborhood $\mathcal{H} \subset H^{4\kappa}_{00}$ of the zero function into $H^{4\kappa}_{00}$, satisfying $\phi(u(f), f)=0$ for all $f\in\mathcal{H}$. According to the previously proved uniqueness, this mapping agrees with $L_{A, q, \beta}$ in $\mathcal{H}$, which completes the proof of the theorem.}
\end{proof}
\vspace{-4mm}
\section{Geometric optics construction}\label{e_s3}
\noindent In this section, we present the construction of geometric optics solutions to the linear fourth-order Schr\"odinger equation. We begin by considering the homogeneous linear fourth-order Schr\"odinger equation
\begin{equation}\begin{split}\label{e_HLS} \I\partial_t u + \Delta^2 u - \nabla\cdot(A\nabla u) + q u &= 0 \quad\:\mbox{in}\quad (0,T)\times M\\
			u|_{t=0}& = 0\quad \mbox{ in }\quad M.\end{split}\end{equation}
Let $\tau>0$ be a large parameter. We look for an ansatz of the following form $u= U+R_{\tau}$, where $R_{\tau}$ is an error and $U$ is an approximate solution given by
\[U(t,x)=e^{\I\tau(\xi\cdot x+c\tau^3 t)}a(\tau;t,x)=e^{\I\tau(\xi\cdot x+c\tau^3 t)}\Bigg(\sum_{k=0}^N\frac{a_k(t,x)}{\tau^k}\Bigg), \]
where $\xi\in \mathbb{R}^n$, $c$ is a positive constant, and $a_k$ $(0\leq k\leq N)$ are the amplitude terms that need to be determined. Let us compute the left-hand side of equation \eqref{e_HLS} for this choice of $U(t,x)$. 
\begin{align}\label{Su}
\left(\I\partial_t+\Delta^2 -\nabla\cdot(A\nabla ) + q \right)U &=  e^{\I\tau(\xi\cdot x+c\tau^3t)}\left[(\I\partial_t+\Delta^2 -\nabla\cdot(A\nabla )+ q)a + \tau^4(|\xi|^4-c)a  - 4\I\tau^3 |\xi|^2(\mathcal{T}_\xi a) \right.\nonumber \\ &\quad\left. - \tau^2\left(2|\xi|^2\Delta a + 4\xi\cdot\nabla(\mathcal{T}_\xi a)-(\xi\cdot A\xi)a\right) + \tau \left(2\I\xi\cdot\nabla(\Delta a) \right.\right. \nonumber \\ & \quad \left.\left. +2\I\Delta (\mathcal{T}_\xi a)-\I(A\xi)\cdot \nabla a - \I (\nabla\cdot A\xi)a - \I \xi\cdot(A\nabla a)\right)\right],\end{align}
where $\mathcal{T}_\xi=\sum_{l=1}^n\xi^l\partial_{x^l}$ is the transport operator \tblue{(directional derivative)} in the direction of $\xi$. Further, we need the right-hand side of the above equation to vanish in the power of $\tau$. Thus we get
\begin{equation}\label{eikonal}
c=|\xi|^4,
\end{equation}
and the amplitude terms $a_k$ satisfies the following transport equations
\begin{align}\label{transports}
\mathcal{T}_\xi a_0 &=0, \nonumber \\ 
\mathcal{T}_\xi a_1 &= \frac{\I}{4|\xi|^2}(2|\xi|^2 \Delta a_0 - (\xi\cdot A\xi)a_0), \nonumber \\
\mathcal{T}_\xi a_2 &= \frac{\I}{4|\xi|^2}\left[2|\xi|^2 \Delta a_1 + 4\xi\cdot\nabla(\mathcal{T}_\xi a_1) - (\xi\cdot A\xi)a_{1}-2\I\xi\cdot\nabla(\Delta a_{0}) + \I(A\xi \cdot \nabla a_0)+\I(\nabla\cdot A\xi)a_0 \right. \nonumber \\ & \left.  \qquad \qquad + \I\xi\cdot(A\nabla a_0)\right], \nonumber \\
\mathcal{T}_\xi a_N &=\frac{\I}{4|\xi|^2}\left[2|\xi|^2 \Delta a_{N-1} +4\xi\cdot\nabla(\mathcal{T}_\xi a_{N-1})- (\xi\cdot A\xi)a_{N-1}-2\I\xi\cdot\nabla(\Delta a_{N-2})-2\I\Delta (\xi\cdot \nabla a_{N-2}) \right. \nonumber \\ & \left.  \qquad + \I(A\xi)\cdot \nabla a_{N-2} +\I(\nabla\cdot A\xi)a_{N-2} + \I\xi\cdot(A\nabla a_{N-2})-(\I\partial_t+\Delta^2-\nabla\cdot(A\nabla) + q)a_{N-3} \right],  \nonumber  \\ & \qquad \mbox{for $N=3,4,\dots$}.
\end{align}
We will fix a point $y \in M$ and let $\gamma_{y,\xi}$ denote the straight line passing through $y$ and in the direction $\xi$, that is, $\gamma_{y,\xi}(s) = s\xi + y$, $s \in \mathbb{R}$. We then select unit vectors $\omega_l \in \mathbb{S}^{n-1}$ ($l = 1, 2, \dots , (n-1)$), so that the set $\displaystyle \left\{ \frac{\xi}{|\xi|}, \omega_1, \ldots, \omega_{n-1} \right\}$ constitutes an orthonormal basis of $\mathbb{R}^n$ with respect to the Euclidean metric. Now, for a sufficiently small parameter $\delta > 0$, we define the zeroth amplitude as
\begin{equation}\label{amp0}
a_0(t,x)=\phi(t)\prod_{l=1}^{n-1}\chi_\delta(\omega_l\cdot(x-y)),
\end{equation}
where $\chi_\delta\in C^\infty_0(-\delta,\delta)$ and $\phi\in C^\infty_0((0,T))$ is a smooth cutoff function. Consequently, for every $t\in(0,T)$, the support of $a_0(t,\cdot)$ is contained in a $\delta$-tubular neighborhood of the line $\gamma_{y,\xi}(\mathbb{R})$.
	
\noindent We can use the transport equations appearing in \eqref{transports} with initial conditions set to zero on the subset $\displaystyle \Sigma_{y,\xi} = \left\{x\in \mathbb{R}^n: \xi\cdot(x-y)=0\right\}$ to iteratively compute the remaining amplitudes $a_k$ for $k \geq 1$. Therefore, we have that
\begin{align}\label{ampsk}
& a_1 (s\xi+y)= \frac{\I}{4|\xi|^2}\int_0^s (2|\xi|^2 \Delta a_0 - (\xi\cdot A\xi)a_0)(\tilde{s}\xi+y)\df \tilde{s}, \nonumber \\
&a_2 (s\xi+y) = \frac{\I}{4|\xi|^2}\int_0^s\left[2|\xi|^2 \Delta a_1 + 4\xi\cdot\nabla(\mathcal{T}_\xi a_1) - (\xi\cdot A\xi)a_{1}-2\I\xi\cdot\nabla(\Delta a_{0}) \right. \nonumber \\ & \qquad \left. \qquad \qquad \qquad \qquad+ \I(A\xi \cdot \nabla a_0)+\I(\nabla\cdot A\xi)a_0 + \I\xi\cdot(A\nabla a_0)\right](\tilde{s}\xi+y)\df \tilde{s}, \nonumber \\
&a_k(s\xi+y)=\frac{\I}{4|\xi|^2}\int_0^s \left[2|\xi|^2 \Delta a_{k-1} +4\xi\cdot\nabla(\mathcal{T}_\xi a_{k-1})- (\xi\cdot A\xi)a_{k-1}-2\I\xi\cdot\nabla(\Delta a_{k-2})\right. \nonumber \\ & \qquad \qquad \qquad \qquad \left.  \qquad -2\I\Delta (\xi\cdot \nabla a_{k-2}) +\I(A\xi)\cdot \nabla a_{k-2} +\I(\nabla\cdot A\xi)a_{k-2} + \I\xi\cdot(A\nabla a_{k-2})\right. \nonumber \\ & \qquad \left.\qquad \qquad \qquad \qquad -(\I\partial_t+\Delta^2-\nabla\cdot(A\nabla) + q)a_{k-3} \right](\tilde{s}\xi+y)\df \tilde{s}, \quad \mbox{for} \ k=3,4,\dots.
\end{align}
From the above discussion, we get for all $t \in (0,T)$, $U(t,\cdot)$ is compactly supported within a $\delta$-neighborhood of $\gamma_{y,\xi}(\mathbb{R})$. Then, substituting \eqref{eikonal} and \eqref{transports} into \eqref{Su}, it follows that
\begin{align*}
\left(\I\partial_t+\Delta^2-\nabla\cdot(A\nabla) + q\right)U& =e^{\I\tau(\xi\cdot x+|\xi|^4\tau^3 t)}\left[(\I\partial_t+\Delta^2-\nabla\cdot(A\nabla) + q)\left(\frac{a_{N-2}}{\tau^{N-2}}+\frac{a_{N-1}}{\tau^{N-1}}+\frac{a_{N}}{\tau^{N}}\right) \right. \\ &\left. \quad +\tau\Big(2\I\xi\cdot\nabla\Delta + 2\I\Delta \mathcal{T}_\xi - \I A\xi\cdot \nabla - \I (\nabla\cdot A\xi)-\I \xi\cdot A\nabla\Big) \left(\frac{a_{N-1}}{\tau^{N-1}}+\frac{a_{N}}{\tau^{N}} \right) \right.\\ &\qquad \left. -\tau^2\Big(2|\xi|^2\Delta-\xi\cdot A\xi+ 4\xi\cdot\nabla\mathcal{T}_\xi\Big)  \frac{a_{N}}{\tau^{N}}\right].
\end{align*}
As a result of direct computation, we find that
\[\norm{\left(\I\partial_t+\Delta^2-\nabla\cdot(A\nabla) + q\right)U}_{H^\sigma((0,T)\times M)}\lesssim\tau^{-N+2+4\sigma}.\]
Now, we convert the approximate solution $U(t,x)$ of equation \eqref{e_HLS} into an exact solution by adding a remainder term $R_\tau$ to $U$, writing $u=U+R_\tau$. For $u$ to be a solution of equation \eqref{e_HLS}, the remainder term $R_\tau$ solves the following system:
\begin{align}\left\{\begin{array}{rll}
\label{remainder}
\left(\I\partial_t+\Delta^2-\nabla\cdot(A\nabla) + q\right)R_\tau&=-\left(\I\partial_t+\Delta^2-\nabla\cdot(A\nabla) + q\right)U &\text{in $(0,T) \times M$,}\\
R_\tau|_{x \in \partial M} &= \frac{\partial R_\tau}{\partial \nu}|_{x \in \partial M}= 0, & \text{on  } (0, T) \times \partial M \\
R_\tau|_{t = 0} &= 0 &\text{in } M.
\end{array}\right.\end{align}
Consequently, the energy estimate \eqref{e_NRG4} directly indicates that
\begin{equation}\label{e_NRG_rem}
\norm{R_{\tau}}_{H^\sigma((0,T)\times M)} \lesssim \norm{\left(\I\partial_t+\Delta^2-\nabla\cdot(A\nabla) + q\right)U}_{H^\sigma((0,T)\times M)}\lesssim \tau^{-N+2+ 4\sigma},
\end{equation}
for any $\sigma = 4\kappa$ where $\kappa\in\N$.
\subsection{Determination of the Boundary Sources}\label{e_s3.2}
In this subsection, we show that the amplitudes of geometric optics solutions in $(0, T)\times \Gamma$ are determined by the source-to-solution map $L_{A, q, \beta}$. To this end, we recall that the $\Gamma$ is a neighbourhood of $\partial M$, and the linear fourth-order Schr\"odinger equation
\begin{align*}
	\left\{\begin{array}{rll}
	\left(\I\partial_t+\Delta^2-\nabla\cdot(A\nabla) + q\right)u &= f  &\text{in $(0,T) \times M$,}\\
	u|_{x \in \partial M}&= \frac{\partial u}{\partial \nu}|_{x \in \partial M} = 0, & \text{on } (0, T)\times \partial M \\
	u|_{t = 0} &= 0 & \text{in $ M$}.
	\end{array}\right.
	\end{align*}
Define the source-to-solution map $\mathcal{L}_{A,q}$ for the linear system as follows:
\[\mathcal{L}_{A,q}f=u|_{(0,T)\times\Gamma},\]
where source term $f$ is supported in $(0,T)\times\Gamma$ and $u$ solves the above linear system. Note that for any $A, q, \beta$, $L_{A,q,\beta}$ determines $\mathcal{L}_{A,q}$ via the expression $\mathcal{L}_{A,q}  f=\partial_\varepsilon L_{A, q, \beta}(\varepsilon f)|_{\varepsilon=0}$. This implies that 
\begin{align*}
L_{A^{(1)},q^{(1)}, \beta^{(1)}}= L_{A^{(2)}, q^{(2)}, \beta^{(2)}}\implies \mathcal{L}_{A^{(1)}, q^{(1)}}=\mathcal{L}_{A^{(2)}, q^{(2)}}.
\end{align*}
Next, we consider the fourth-order Schr\"odinger equation, which is backward in time.
\begin{align*}
	\left\{\begin{array}{rll}
		\left(\I\partial_t+\Delta^2-\nabla\cdot(A\nabla) + q\right) w &= h & \text{in $(0,T) \times M$,}\\
		w|_{x \in \partial M}&=\frac{\partial w}{\partial \nu}|_{x \in \partial M} = 0,& \text{on $(0,T) \times \partial M$,}\\
		w|_{t = T} &= 0& \text{in $M$}.
	\end{array}\right.
\end{align*}
Then, it holds that $\mathcal{L}_{A, q}^\ast h=w|_{(0,T)\times\Gamma}$ for $h$ supported in $(0,T)\times\Gamma$. In fact, for source terms $f,\,h$ supported in $(0,T)\times\Gamma$, we have
\begin{equation}\label{adjoint}\pair{\mathcal{L}_{A, q}f,h}_{L^2((0,T)\times\Gamma)}=\pair{f,w}_{L^2((0,T)\times\Gamma)}.
\end{equation}
For $\ell=1,2$, we consider the potentials $A^{(\ell)}$ and $q^{(\ell)}$, where $A^{(\ell)}_{ij}$ (for $1 \leq i,j \leq n$) and $q^{(\ell)}$ belong to $C^\infty_0((0,T) \times M \setminus \Gamma)$. For any line $\gamma_{r,\xi}$ with an initial point $r \in \partial M$ and an initial direction $\xi \in \mathbb{R}^n$, we can define a sequence of functions $a_k$ and $b_k$ corresponding to the pairs $(A^{(1)}, q^{(1)})$ and $(A^{(2)}, q^{(2)})$, respectively. To define these sequences, we first choose vectors \(\omega_1, \ldots, \omega_{n-1} \in \mathbb{S}^{n-1}\) (as before) such that the set \(\left\{\frac{\xi}{|\xi|}, \omega_1, \ldots, \omega_{n-1}\right\}\) forms an orthonormal basis of \(\mathbb{R}^n\). Again, as before, we choose the zeroth amplitude for some small \(\delta > 0\) as follows
\begin{align}\label{eq_a0}
a_0(t,x) = b_0(t,x) =\phi(t)\prod_{l=1}^{n-1}\chi_\delta(\omega_l\cdot(x-r)),
\end{align}
where $\chi_\delta\in C^\infty_0(-\delta,\delta)$ and $\phi\in C^\infty_0(0,T)$ is a smooth cutoff function. The subsequent functions $a_k$ are defined by solving the following transport equations
\begin{align}\label{sequence}
\mathcal{T}_\xi a_1 &= \frac{\I}{4|\xi|^2}\left[2|\xi|^2 \Delta a_0 - (\xi\cdot A^{(1)}\xi)a_0\right], \quad a_{1}|_{\Sigma_{r,\xi}}=0. \nonumber \\
\mathcal{T}_\xi a_2 &= \frac{\I}{4|\xi|^2}\left[2|\xi|^2 \Delta a_1 + 4\xi\cdot\nabla(\mathcal{T}_\xi a_1) - (\xi\cdot A^{(1)}\xi)a_{1}-2\I\xi\cdot\nabla(\Delta a_{0}) + \I(A^{(1)}\xi \cdot \nabla a_0) \right. \nonumber \\ & \left.  \qquad \qquad +\I(\nabla\cdot A^{(1)}\xi)a_0+ \I\xi\cdot(A^{(1)}\nabla a_0)\right], \quad a_{2}|_{\Sigma_{r,\xi}}=0.\nonumber \\
\mathcal{T}_\xi a_N &=\frac{\I}{4|\xi|^2}\Big[2|\xi|^2 \Delta a_{N-1} +4\xi\cdot\nabla(\mathcal{T}_\xi a_{N-1})- (\xi\cdot A^{(1)}\xi)a_{N-1}   -2\I\xi\cdot\nabla(\Delta a_{N-2})\nonumber \\ & \qquad \qquad-2\I\Delta (\xi\cdot \nabla a_{N-2})  + \I(A^{(1)}\xi)\cdot \nabla a_{N-2} +\I(\nabla\cdot A^{(1)}\xi)a_{N-2} + \I\xi\cdot(A^{(1)}\nabla a_{N-2})\nonumber  \\ & \qquad \qquad-(\I\partial_t+\Delta^2-\nabla\cdot(A^{(1)}\nabla) + q^{(1)})a_{N-3} \Big], \quad a_{N}|_{\Sigma_{r,\xi}}=0,\quad \mbox{for $N=3,4,\dots$}.
\end{align}
Likewise, we can define the subsequent functions \(b_k\) by solving the same transport equations as before, but replacing the potentials \(A^{(1)}, q^{(1)}\) with \(A^{(2)}, q^{(2)}\).
\vspace{2mm}\\	
\noindent For each $\tau > 0$ and $N \in \mathbb{N}$, we can define an approximate geometric optics solutions $U^{(\ell)}$ for $\ell=1,2$ associated with the amplitudes $a_k$ and $b_k$, respectively, by the expressions
\[U^{(1)}(t,x)=e^{\I(\tau\xi\cdot x+|\xi|^4\tau^4t)}\sum_{k=0}^N\tau^{-k}a_k(t,x) \quad \mbox{and} \quad U^{(2)}(t,x)=e^{\I(\tau\xi\cdot x+|\xi|^4\tau^4t)}\sum_{k=0}^N\tau^{-k}b_k(t,x).
\]
We denote \( u^{(\ell)} \) as the exact solution corresponding to the fourth-order Schr\"odinger equation
\begin{align}
	\left\{\begin{array}{rlc}
	\left(\I\partial_t+\Delta^2-\nabla\cdot(A^{(\ell)}\nabla) + q^{(\ell)}\right) u^{(\ell)} &= 0 \quad  \text{ in $(0,T) \times M$,}\\
	u^{(\ell)}|_{t = 0} &= 0,\quad \text{in}\  M
\end{array}\right.
\end{align}
as constructed above. Observe that $U^{(\ell)}$ approximates $u^{(\ell)}$ in $L^2$ up to an error of order $\mathcal{O}(\tau^{-N+2})$.
\vspace{2mm}\\
Let $\eta\in C^\infty_0(M)$ be such that $\eta=1$ in $M\setminus\Gamma$ and set $f_\ell=\left(i\partial_t+\Delta^2-\nabla\cdot(A^{(\ell)}\nabla) + q^{(\ell)}\right)(\eta u^{(\ell)})$ so that the source $f_\ell$ is supported in $(0,T)\times\Gamma$. It follows that the function $\eta u^{(\ell)}$ solves the fourth-order Schr\"odinger equation
    \begin{align}
		\left\{\begin{array}{rll}
			\left(\I\partial_t+\Delta^2-\nabla\cdot(A^{(\ell)}\nabla) + q^{(\ell)}\right) \mathcal{U}^{(\ell)}&= f_\ell,  &\text{in $(0,T) \times M$,}
			\\
			\mathcal{U}^{(\ell)}|_{x \in \partial M}&=\frac{\partial \mathcal{U}^{(\ell)}}{\partial \nu}|_{x \in \partial M} = 0, & \mbox{on  } (0, T) \times \partial M
			\\
			\mathcal{U}^{(\ell)}|_{t = 0} &= 0, & \mbox{in  } M.
		\end{array}\right.
	\end{align}
Similarly, a geometric-optics solution can be considered for the backward-in-time problem 
\begin{align*}
	\left\{\begin{array}{rll}
	\left(\I\partial_t+\Delta^2-\nabla\cdot(A^{(2)}\nabla) + q^{(2)}\right)w = 0 & \text{in $(0,T) \times M$,}\\
	w|_{t = T} = 0 & \text{in $M$}.
	\end{array}\right.
\end{align*}
Note that \( w \) can also be approximated to \( \mathcal{O}(\tau^{-N+2}) \) using the expression 
\[ e^{i(\tau \xi \cdot x + |\xi|^4 \tau^4 t)} \sum_{k=0}^N \tau^{-k} w_k(t,x), \] 
where \( w_0 \) is defined by equation \eqref{eq_a0} and for \( k \ge 1 \), \( w_k \) represents a sequence that satisfies equation \eqref{sequence} with coefficients \( A^{(2)} \) and \( q^{(2)} \). Moreover, if $\tilde{\eta}\in C^\infty_0(M)$ is chosen so that $\tilde{\eta}=1$ in $M\setminus\Gamma$, then for $h=(\I\partial_t+\Delta^2-\nabla\cdot(A^{(2)}\nabla) + q^{(2)})(\tilde{\eta}w)$, the function $\tilde{\eta}w$ solves
\begin{align*}
	\left\{\begin{array}{rll}
	\left(\I\partial_t+\Delta^2-\nabla\cdot(A^{(2)}\nabla) + q^{(2)}\right) \mathcal{W} &= h & \text{in $(0,T) \times M$,}
	\\
	\mathcal{W}|_{x \in \partial M}=\frac{\partial \mathcal{W}}{\partial \nu}|_{x \in \partial M} &= 0 & \text{on $(0,T) \times \partial M$,}\\
	\mathcal{W}|_{t = T} &= 0 & \text{in $ M$}.
	\end{array}\right.
	\end{align*}
\begin{lemma}\label{e_lemma}
Suppose that $\mathcal{L}_{A^{(1)},q^{(1)}}=\mathcal{L}_{A^{(2)},q^{(2)}}$. Then $a_k=b_k$ in $(0,T)\times\Gamma$ for all $k\in\N$. 
\end{lemma}
\begin{proof}
The proof of this lemma is very similar to \cite[Lemma 3.1]{LOSST_NLS_2025}. Therefore, we skip it.
\end{proof}
\vspace{-4mm}
\section{Proof of main result: Euclidean case}\label{main_theorem}
\noindent The aim of this section is to present a proof of Theorem \ref{th: Eulidean}. The argument is based on a second-order linearization technique and on the nonlinear interactions among the GO solutions constructed above.
\begin{proof}[Proof of Theorem \ref{th: Eulidean}]
The proof is long and involves various steps, which we discuss one by one in separate parts.
\vspace{-3mm}
\subsubsection*{\textbf{Step 1:}}
In this step, we select geometric-optics solutions to the linearized equation corresponding to distinct source terms. \tblue{For a fixed point $p\in M\setminus\Gamma$ and a small $\lambda>0$, we choose three pairwise linearly independent vectors $\xi_0$, $\xi_1$, and $\xi_2$ such that 
\[\xi_0 = \xi_1 + \xi_2\]
and 
\[|\xi_0|^4 = 1,\quad |\xi_1|^4 = 1 - \lambda^4,\quad |\xi_2|^4 = \lambda^4.\]}
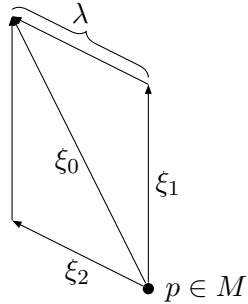
\begin{figure}[ht]
	\begin{center}
		\begin{tikzpicture}[scale=0.9]
		\coordinate [circle,fill,inner sep=1.5pt,label={[label distance=0.1mm]0:$p\in M$}](A) at (0,0);
		\coordinate (B) at (0,3);
		\coordinate (C) at (-2,4);
		\coordinate (D) at (-2,1);
		\draw [->,>={Triangle[scale width=0.6]}](A) -- (C) node[midway,yshift=-1.5mm,xshift=-1.5mm]{$\xi_0$};
		\draw [->,>={Triangle[scale width=0.6]}](A) -- (B) node[midway,xshift=2.6mm]{$\xi_1$};
		\draw [->,>={Triangle[scale width=0.6]}](B) -- (C) node[midway,yshift=-2.5mm]{};
		\draw [->,>={Triangle[scale width=0.6]}](A) -- (D) node[midway,yshift=-2.5mm]{$\xi_2$};
		\draw [->,>={Triangle[scale width=0.6]}](D) -- (C) node[midway,xshift=2.6mm]{};
		\draw [decorate,decoration={brace,amplitude=5pt,mirror}]([yshift=1mm]B) -- ([yshift=1mm]C) node[midway,yshift=4mm]{$\lambda$};
		\end{tikzpicture}
		\caption{The vectors $\xi_0$, $\xi_1$ and $\xi_2$.}
		\end{center}
	\end{figure}
\tblue{For $j \in \{0,1,2\}$, let $r_j\in\partial M$ be the point where the line $\gamma_{p,\xi_j}$ intersects $\partial M$ and the vector $p-r_j$ is a positive multiple of $\xi_j$.} Next, we  define the amplitude functions $a_k^{(j)}$ and $b_k^{(j)}$ for the pairs of coefficients $(A^{(1)},q^{(1)})$ and $(A^{(2)},q^{(2)})$, respectively as follows. We start by constructing $a_0^{(j)}$ and $b_0^{(j)}$ by substituting $\xi=\xi_j$ into expression \eqref{amp0}. Then for a large $N\in\N$, we define the remaining amplitudes $a_1^{(j)},\dots,a_N^{(j)}$ and $b_1^{(j)},\dots,b_N^{(j)}$ by substituting $\xi=\xi_j$ into equation \eqref{ampsk} for all $y$ in the hyperplane $\displaystyle \Sigma_{r_j,\xi_j}$. For $\tau>0$, we can define the approximate geometric optics solutions $U^{(1)}_j(t,x)$ and $U^{(2)}_j(t,x)$ for the pairs of coefficients $(A^{(1)},q^{(1)})$ and $(A^{(2)},q^{(2)})$, respectively, on the following form:
\begin{align*}
U^{(1)}_j(t,x)=e^{\I(\tau\xi_j\cdot x+\tau^4|\xi_j|^4t)}\Bigg(\sum_{k=0}^N\tau^{-k}a_k^{(j)}(t,x)\Bigg),\\
U^{(2)}_j(t,x)=e^{\I(\tau\xi_j\cdot x+\tau^4|\xi_j|^4t)}\Bigg(\sum_{k=0}^N\tau^{-k}b_k^{(j)}(t,x)\Bigg).
\end{align*} 
\tblue{The corresponding exact solutions are $u^{(\ell)}_j=U^{(\ell)}_j+R^{(\ell)}_{\tau,j}$, where the remainder $R^{(\ell)}_{\tau,j}$ solves equation \eqref{remainder} for $\ell=1,2$. Take the source term $f_j=\left(\I \partial_t+\Delta^2-\nabla\cdot(A^{(\ell)}\nabla) + q^{(\ell)}\right)(\eta u^{(\ell)}_j)$ , where $\eta\in C_0^\infty(M)$ with $\eta=1$ in $M\setminus\Gamma$.} Then, for $j=1,2$ the function $\eta u^{(\ell)}_j$ solves
\begin{align}
\left\{\begin{array}{rll}
\I \partial_t \mathcal{U}^{(\ell)}_j + \Delta^2 \mathcal{U}^{(\ell)}_j - \nabla\cdot(A^{(\ell)}\nabla \mathcal{U}^{(\ell)}_j)+q^{(\ell)}  \mathcal{U}^{(\ell)}_j &= f_j & \text{in $(0,T) \times M$,}\\
\mathcal{U}^{(\ell)}_j\left|_{x \in \partial M}\right.&=\frac{\partial  \mathcal{U}^{(\ell)}_j}{\partial \nu}|_{x \in \partial M} = 0, &\text{on $(0,T) \times \partial M$,}\\
\mathcal{U}^{(\ell)}_j|_{t = 0} &= 0 & \text{in $ M$}.
\end{array}\right.
\end{align} 
For $j=0$, take $\tilde{\eta}\in C_0^\infty(M)$ with $\tilde{\eta}=1$ in $M\setminus\Gamma$ then the function $\tilde{\eta} u^{(\ell)}_0$ solves 
\begin{align}
\left\{\begin{array}{rll}
\I \partial_t \mathcal{U}^{(\ell)}_0 + \Delta^2 \mathcal{U}^{(\ell)}_0 - \nabla\cdot(A^{(\ell)}\nabla \mathcal{U}^{(\ell)}_0)+q^{(\ell)}  \mathcal{U}^{(\ell)}_0 &= f_0 & \text{in $(0,T) \times M$,}\\
\mathcal{U}^{(\ell)}_0|_{x \in \partial M}&=\frac{\partial  \mathcal{U}^{(\ell)}_0}{\partial \nu}|_{x \in \partial M} = 0, & \text{on $(0,T) \times \partial M$,}\\
\mathcal{U}^{(\ell)}_0|_{t = T} &= 0& \text{in $ M$}.
\end{array}\right.
\end{align}
\subsubsection*{\textbf{Step 2:}} In this step, we derive an integral identity involving $A^{(\ell)},q^{(\ell)}$, and $\beta^{(\ell)}$ for $\ell=1,2$. We start with linearizing the following fourth-order nonlinear Schrödinger equation
\begin{align}\label{NLS}
\left\{\begin{array}{rll}
\I \partial_t u_\ell + \Delta^2 u_\ell -\nabla\cdot(A^{(\ell)}\nabla u_\ell) +q^{(\ell)} u_\ell + \beta^{(\ell)} u_\ell^2 &= f & \text{in $(0,T) \times M$,}\\
u_\ell|_{x \in \partial M} &= \frac{\partial u_\ell}{\partial \nu}|_{x \in \partial M} = 0,& \text{on $(0,T) \times \partial M$,}\\
u_\ell|_{t = 0} &= 0& \text{in $M$}.
\end{array}\right.
\end{align}
Let $\varepsilon=(\varepsilon_1,\varepsilon_2)$, where $\varepsilon_1,\varepsilon_2>0$ be small values and let \( f_1, f_2 \in H^{4k}_{00} \) be supported in the interval \( (0, T) \times \Gamma \). \tblue{Then for sufficiently small $\varepsilon$, the source term $f=\varepsilon_1f_1+\varepsilon_2f_2$ lies in $\mathcal{H}$ and we have}
\[-\frac{1}{2}\partial_{\varepsilon_1}\partial_{\varepsilon_2}L_{A^{(\ell)},q^{(\ell)},\beta^{(\ell)}}f|_{\varepsilon=0}=w_\ell|_{(0,T)\times\Gamma}.\]
The $w_\ell$ appearing in the above relation solves
\begin{align}
\left\{\begin{array}{rll}
\I\partial_tw_\ell+\Delta^2 w_\ell-\nabla\cdot(A^{(\ell)}\nabla w_\ell)+q^{(\ell)} w_\ell&=\beta^{(\ell)}\mathcal{U}^{(\ell)}_1\mathcal{U}^{(\ell)}_2 & \text{in $(0,T) \times M$}\\
w_\ell|_{x \in \partial M}&=\frac{\partial w_\ell}{\partial \nu}|_{x \in \partial M} = 0, & \text{on $(0,T) \times \partial M$,}\\
w_\ell|_{t=0} &= 0, & \text{in $M$}.
\end{array}\right.
\end{align}
This implies
\begin{align*}\int_{(0,T)\times M}w_\ell\ \overline{(\I\partial_t+\Delta^2- \nabla\cdot(A^{(\ell)}\nabla)+q^{(\ell)})\mathcal{U}^{(\ell)}_0}\ \df x\, \df t &=\int_{(0,T)\times\Gamma} w_\ell \overline{f_0}\ \df x\, \df t\\ &= -\frac{1}{2}\int_{(0,T)\times\Gamma}\partial_{\varepsilon_1}\partial_{\varepsilon_2}L_{A^{(\ell)},q^{(\ell)},\beta^{(\ell)}}f|_{\varepsilon=0}\ \overline{f_0}\ \df x\,\df t.
\end{align*}
We now integrate the left-hand side of the above equation by parts to get
\begin{align}\int_{(0,T)\times M}(\I\partial_t+\Delta^2- \nabla\cdot(A^{(\ell)}\nabla)+q^{(\ell)})w_\ell \ \overline{\mathcal{U}^{(\ell)}_0}\ \df x\, \df t &=\int_{(0,T)\times M} \beta^{(\ell)}\ \mathcal{U}^{(\ell)}_1\  \mathcal{U}^{(\ell)}_2 \ \overline{\mathcal{U}^{(\ell)}_0}\ \df x\, \df t \nonumber\\ \label{e_before} &= -\frac{1}{2}\int_{(0,T)\times\Gamma}\partial_{\varepsilon_1}\partial_{\varepsilon_2}L_{A^{(\ell)},q^{(\ell)},\beta^{(\ell)}}f|_{\varepsilon=0}\ \overline{f_0}\ \df x\,\df t.
\end{align}
Next, subtracting the above equation for $\ell =1$ from $\ell =2$ and using $L_{A^{(1)},q^{(1)},\beta^{(1)}} = L_{A^{(2)},q^{(2)},\beta^{(2)}}$ entails
\begin{align}\label{e_after}\int_{(0,T)\times M} \left(\beta^{(1)}\ \mathcal{U}^{(1)}_1\  \mathcal{U}^{(1)}_2 \ \overline{\mathcal{U}^{(1)}_0}-\beta^{(2)}\ \mathcal{U}^{(2)}_1\  \mathcal{U}^{(2)}_2 \ \overline{\mathcal{U}^{(2)}_0}\right)\ \df x\, \df t = 0.
\end{align}
We now like to approximate $\mathcal{U}_j$ in the integral \eqref{e_after} by $\eta U_j$ for $j=1,2$ and $\mathcal{U}_0$ by $\tilde{\eta} U_0$. Therefore, choose $\kappa$ sufficiently large so that  $H^{4\kappa}_{00}$ is a Banach algebra, and let $N\geq4\kappa+4$. Then, up to an error of order $\mathcal{O}(\tau^{-4})$, the integral \eqref{e_after} becomes equivalent to the integral
\[\int_{(0,T)\times M}\eta^2\tilde{\eta} \left(\beta^{(1)}\ U^{(1)}_1\  U^{(1)}_2 \ \overline{U^{(1)}_0}-\beta^{(2)}\ U^{(2)}_1\  U^{(2)}_2 \ \overline{U^{(2)}_0}\right)\, \df x\, \df t=0.\]
Since $\phi(t)$ appearing in the definition of $a_0^{(j)}$ and $b_0^{(j)}$ is arbitrary, it follows that the above integral is determined for all $t\in(0, T)$ 
\begin{equation}\label{e_I}\int_M\eta^2\tilde{\eta} \left(\beta^{(1)}(t,\cdot)\ U^{(1)}_1(t,\cdot)\  U^{(1)}_2(t,\cdot) \ \overline{U^{(1)}_0}(t,\cdot)-\beta^{(2)}(t,\cdot)\ U^{(2)}_1(t,\cdot)\  U^{(2)}_2(t,\cdot) \ \overline{U^{(2)}_0}(t,\cdot)\right)\, \df x\,=0
\end{equation}
up to a small error $\mathcal{O}(\tau^{-4})$. \tblue{For simplicity, we will omit the explicit $t$-dependence from our notation moving forward. Next, we expand the preceding integral \eqref{e_I} in powers of $\tau$ as }
\[\mathcal{I}_0+\mathcal{I}_1\tau^{-1}+ \mathcal{I}_2\tau^{-2}+\mathcal{I}_3\tau^{-3}+\mathcal{O}(\tau^{-4}) = 0.\]
\tblue{It is important to note that taking the product $\overline{U_0} U_1 U_2$ exactly cancels the phases associated with the approximate geometric optics solutions, and additionally $\eta=\tilde{\eta}=1$ in $\supp\big(\overline{U_0}U_1U_2\big)$.} 
Therefore, it follows that 
\[\mathcal{I}_0=\int_M\left(\beta^{(1)} a_0^{(0)}a_0^{(1)}a_0^{(2)}-\beta^{(2)} b_0^{(0)}b_0^{(1)}b_0^{(2)}\right)\, \df x = 0.\]
Since \( a_0^{(j)} = b_0^{(j)} \) for \( j = 0, 1, 2 \), the above identity simplifies to
\[\int_M\left(\beta^{(1)} -\beta^{(2)}\right)a_0^{(0)}a_0^{(1)}a_0^{(2)}\, \df x = 0.\]
\tblue{If we now allow $\chi_\delta$ in the definition of $a^{(j)}_0$ given by \eqref{amp0} to converge to the indicator function of $(-\delta, \delta)$, it follows that}
\[\int_{P_\delta}\left(\beta^{(1)}(x) -\beta^{(2)}(x)\right)\, \df x = 0\]
\tblue{where $P_\delta$ is a small neighborhood of $p$ contained within a ball of radius $\delta$. By taking the limit as $\delta \to 0$, we conclude that} $\beta^{(1)}(p) = \beta^{(2)}(p).$ \tblue{Because the choice of $p \in M \setminus \Gamma$ was arbitrary, we uniquely recovered the function $\beta(t, x)$ everywhere.}
\vspace{-2mm}
\subsubsection*{\textbf{Step 3:}} In this step, we recover a symmetric two-tensor field with the help of properties of \textbf{divergent beam transform}. 
By using $\beta^{(1)} = \beta^{(2)}$ in the expression for $\mathcal{I}_1$, we obtain
$$ \mathcal{I}_1=\sum_{|e|=1}\int_M\beta \left(a_{e_0}^{(0)}a_{e_1}^{(1)}a_{e_2}^{(2)}-b_{e_0}^{(0)}b_{e_1}^{(1)}b_{e_2}^{(2)}\right)\,\df x = 0$$
where $e$ is a multi-index, $e=(e_0,e_1,e_2)\in(\N \cup\{0\})^3$.\vspace{2mm}\\ 
We recall from equation (\ref{ampsk}) that $a_1^{(j)}$ and  $b_1^{(j)}$ are of the form $a_1^{(j)}=c_1^{(j)}+d_1^{(j)}$ and $b_1^{(j)}=e_1^{(j)}+f_1^{(j)}$, respectively. Here 
\begin{equation}\begin{gathered}\label{e_split}
c_1^{(j)}(s\xi_j+y) =\frac{\I}{2}\int_0^s\Big[\Delta a_0^{(j)}\Big](\tilde{s}\xi_j+y)\df\tilde{s},\quad
d_1^{(j)}(s\xi_j+y)=\frac{-\I}{4|\xi_j|^2}\int_0^s\Big[(\xi_j\cdot A^{(1)}\xi_j)a_0^{(j)}\Big](\tilde{s}\xi_j+y)\df\tilde{s},\\
e_1^{(j)}(s\xi_j+y) =\frac{\I}{2}\int_0^s\Big[\Delta b_0^{(j)}\Big](\tilde{s}\xi_j+y)\df\tilde{s},\quad
f_1^{(j)}(s\xi_j+y)=\frac{-\I}{4|\xi_j|^2}\int_0^s\Big[(\xi_j\cdot A^{(2)}\xi_j)b_0^{(j)}\Big](\tilde{s}\xi_j+y)\df\tilde{s},\end{gathered}\end{equation}
for all $y$ in the hyperplane $\Sigma_{r_j,\xi_j}=\{x\in\R^n:\xi_j\cdot(x-r_j)\}$. 
In particular, $a_0^{(j)}$ and $b_0^{(j)}$ are independent of the potentials, so are $c_1^{(j)}$ and $e_1^{(j)}$. Since $a_0^{(j)}=b_0^{(j)}$, we have $c_1^{(j)}= e_1^{(j)}$ and this further implies
\[\int_M \beta \left[ \left(d_1^{(0)}- f_1^{(0)}\right)a_0^{(1)}a_0^{(2)}+a_0^{(0)}\left(d_1^{(1)}- f_1^{(1)}\right)a_0^{(2)}+a_0^{(0)}a_0^{(1)}\left(d_1^{(2)}- f_1^{(2)}\right)\right]\, \df x = 0.\]
Again (as above) by letting $\chi_\delta$ converge to the indicator function of the interval $(-\delta,\delta)$, we deduce that  
\begin{align}
&\sum_{j=0}^2\int_{P_\delta}\beta \left(d_{1,j}-f_{1,j}\right)\,\df x=0, \quad \mbox{where} \nonumber\\
& d_{1,j}(s\xi_j+y)=\frac{-\I} {4|\xi_j|^2}\int_0^s\Big(\xi_j\cdot A^{(1)}\xi_j\Big)(\tilde{s}\xi_j+y)\df\tilde{s},\quad \mbox{and} \nonumber\\ \label{f_1j_eqn`}
&f_{1,j}(s\xi_j+y)=\frac{-\I}{4|\xi_j|^2}\int_0^s\Big(\xi_j\cdot A^{(2)}\xi_j\Big)(\tilde{s}\xi_j+y)\df\tilde{s},
\end{align}
for a small neighborhood $P_\delta$ of $p$. \tblue{Knowing that $\beta(p)$ is non-zero for almost every $p$, we can send $\delta \to 0$ to recover the quantity}
\begin{equation}\label{2-to-1}
\sum_{j=0}^2(d_{1,j}(p)-f_{1,j}(p)) = 0.
\end{equation}
\tblue{To show that $A$ can be uniquely recovered from equation \eqref{2-to-1}, we first let $\widehat{\xi}$  denote the unit vector along $\xi_2$, yielding $\xi_2 = \lambda \widehat{\xi}$. We then let $s_0 \in \mathbb{R}$ be the parameter value satisfying $\gamma_{r_2, \widehat{\xi}}(s_0) = p$.}
From \eqref{f_1j_eqn`}, we have
\[d_{1,2}(s\xi_2+r_2)=\frac{-\I}{4|\xi_2|^2}\int_0^s\Big(\xi_2\cdot A^{(1)}\xi_2\Big)(\tilde{s}\xi_2+r_2)\df\tilde{s}\]
and 
\[f_{1,2}(s\xi_2+r_2)=\frac{-\I}{4|\xi_2|^2}\int_0^s\Big(\xi_2\cdot A^{(2)}\xi_2\Big)(\tilde{s}\xi_2+r_2)\df\tilde{s}.\]
Since $\gamma_{r_2,\widehat{\xi}}(s_0)=p$, then it holds that
\begin{align*}
d_{1,2}(p)-f_{1,2}(p)&= \frac{-\I}{4|\xi_2|^2}\int_0^{s_0/\lambda}\Big(\xi_2\cdot (A^{(1)}-A^{(2)})\xi_2\Big)(\tilde{s}\xi_2+r_{2})\df\tilde{s}  \\
&=\frac{-\I}{4\lambda}\int_0^{s_0} \Big(\widehat{\xi}\cdot (A^{(1)}-A^{(2)})\widehat{\xi}\Big)(s\widehat{\xi}+r_{2})\,\df s,
\end{align*} 
and we can similarly check that $d_{1,j}(p)-f_{1,j}(p)=\mathcal{O}(1)$ as $\lambda\rightarrow0$ for $j\neq2$. Therefore, we have shown that we can recover the quantity
\[\lim_{\lambda\rightarrow0}\Bigg(\lambda\sum_{j=0}^2\left(d_{1,j}(p)-f_{1,j}(p) \right) \Bigg)=\int_0^{s_0} \Big(\widehat{\xi}\cdot (A^{(1)}-A^{(2)})\widehat{\xi}\Big)(s\widehat{\xi}+r_{2})\,\df s = 0.\]
Hence, we have the following relation  for $A := A^{(1)} - A^{(2)}$
\begin{align}\label{eq:injectivity}
\int_0^{\infty}\langle A(p - s\xi), \xi^{\odot 2}\rangle \ \df s = 0, \quad \mbox{ for all unit vectors } \xi,  
\end{align}
where $\odot$ denotes the symmetric tensor product.  Since the choice of $p\in M\setminus\Gamma$ was arbitrary, the relation \eqref{eq:injectivity} is true for all $p\in M\setminus\Gamma$. Now apply operator $\xi\cdot \nabla_p$ on \eqref{eq:injectivity} to get
\begin{align}\label{A_identity}
0=(\xi\cdot \nabla_p)  \int_0^{\infty}\langle A(p - s\xi), \xi^{\odot 2}\rangle \ \df s&=\int_0^{\infty}(\xi\cdot \nabla_p)\langle A(p - s\xi), \xi^{\odot 2}\rangle \ \df s \nonumber\\ 
&= -\int_0^{\infty}\frac{d}{ds}\left(\langle A(p - s\xi), \xi^{\odot 2}\rangle\right) \ \df s \nonumber\\&= \langle A(p), \xi^{\odot 2}\rangle, \quad \mbox{ for all unit vectors } \xi.
\end{align}
The claim is that $A(p) \equiv 0$, to see that let $\{\xi_1, \xi_2, \dots, \xi_n\}$ be an orthonormal basis of $\mathbb{R}^n$ then by \cite[Lemma 5.4]{Venky_Rohit}, the collection of $ \binom {n+1}2=\frac{n(n+1)}{2}$ symmetric two tensor of the form $\displaystyle \xi_{i}\odot\xi_{j}, \ 1\leq i,j\leq n,$ is linearly independent.
	%
Then, by \eqref{A_identity}, we have
$\langle A(p), \xi_i\odot \xi_j\rangle = 0$ for all $1 \leq i, j \leq n$, which finally implies $A(p) \equiv 0$. Thus, we show that $A^{(1)}(p) = A^{(2)}(p)$, which shows the unique recovery of $A(t,x)$ everywhere.
\subsubsection*{\textbf{Step 4:}}  \tblue{In this final step of the proof, we work on the unique recovery of $q$, that is, $q^{(1)}=q^{(2)}$ everywhere.} To achieve this,  we now consider the integral $\mathcal{I}_3$. By Using $\beta^{(1)} = \beta^{(2)}$ in the expression $\mathcal{I}_3$, we obtain
\begin{align*}
\mathcal{I}_3 = \sum_{|e|=3}\int_M\beta \left(a_{e_0}^{(0)}a_{e_1}^{(1)}a_{e_2}^{(2)}-b_{e_0}^{(0)}b_{e_1}^{(1)}b_{e_2}^{(2)}\right)\,\df x = 0, \quad \mbox{ where }e=(e_0,e_1,e_2)\in(\N \cup\{0\})^3.
\end{align*}
We recall from equation \eqref{ampsk} that $a_3^{(j)}$ and $b_3^{(j)}$ are of the form $a_3^{(j)}=c_3^{(j)}+d_3^{(j)}$ and $b_3^{(j)}=e_3^{(j)}+f_3^{(j)}$, respectively. \tblue{These amplitudes satisfy the following relations for all $y \in \Sigma_{r_j,\xi_j}=\{x\in\R^n:\xi_j\cdot(x-r_j)\}$: }
\begin{align*}
&c_3^{(j)}=  \frac{\I}{4|\xi_j|^2}\int_0^s \left[\left(2|\xi_j|^2 \Delta  +4\xi_j\cdot\nabla(\tau_{\xi_j} )- (\xi_j\cdot A^{(1)}\xi_j)\right)a^{(j)}_{2} + \left(\I(A^{(1)}\xi_j)\cdot \nabla -2\I\xi_j\cdot\nabla\Delta  \right.\right. \nonumber \\ & \left.\left.  \qquad -2\I\Delta (\xi_j\cdot \nabla) +\I(\nabla\cdot A^{(1)}\xi_j) + \I\xi_j\cdot(A^{(1)}\nabla) \right)a^{(j)}_{1}-(\I\partial_t+\Delta^2-\nabla\cdot(A^{(1)}\nabla))a^{(j)}_{0} \right](\tilde{s}\xi_j+y)\df\tilde{s}\\
&d_3^{(j)}=  \frac{\I}{4|\xi_j|^2}\int_0^s \left[q_{1} a^{(j)}_{0} \right](\tilde{s}\xi_j+y)\df\tilde{s}\\
&e_3^{(j)}=  \frac{\I}{4|\xi_j|^2}\int_0^s \left[\left(2|\xi_j|^2 \Delta  +4\xi_j\cdot\nabla(\xi_j\cdot\nabla )- (\xi_j\cdot A^{(2)}\xi_j)\right)b^{(j)}_{2} + \left(\I(A^{(2)}\xi_j)\cdot \nabla -2\I\xi_j\cdot\nabla(\Delta )\right. \right. \nonumber \\ & \left. \left.  \qquad -2\I\Delta (\xi_j\cdot \nabla) +\I(\nabla\cdot A^{(2)}\xi_j) + \I\xi_j\cdot(A^{(2)}\nabla) \right)b^{(j)}_{1}-(\I\partial_t+\Delta^2-\nabla\cdot(A^{(2)}\nabla))b^{(j)}_{0} \right](\tilde{s}\xi_j+y)\df\tilde{s}\\
&f_3^{(j)}=  \frac{\I}{4|\xi_j|^2}\int_0^s \left[q^{(2)} b^{(j)}_{0} \right](\tilde{s}\xi_j+y)\df\tilde{s}.
\end{align*}
Since $a_0^{(j)}=b_0^{(j)}$ and $A^{(1)}=A^{(2)}$, we get $a_1^{(j)}=b_1^{(j)}$, $a_2^{(j)}=b_2^{(j)}$ and $c_3^{(j)}=e_3^{(j)}$. Thus, using these, we get 
\begin{align*}
\int_M \beta \left[ \left(d_3^{(0)}- f_3^{(0)}\right)a_0^{(1)}a_0^{(2)}+a_0^{(0)}\left(d_3^{(1)}- f_3^{(1)}\right)a_0^{(2)}+a_0^{(0)}a_0^{(1)}\left(d_3^{(2)}- f_3^{(2)}\right)\right]\, \df x = 0.
\end{align*}
\tblue{Again, if  we let $\chi_\delta$ in the definition of $a^{(j)}_0$ converge to the indicator function of $(-\delta, \delta)$ then we have}
\[\sum_{j=0}^2\int_{P_\delta}\beta \left(d_{3,j}-f_{3,j}\right)\,\df x=0,\]
where $\displaystyle d_{3,j}(s\xi_j+y)= \frac{\I}{4|\xi_j|^2}\int_0^s q^{(1)}(\tilde{s}\xi_j+y)\,\df\tilde{s}$ and $\displaystyle f_{3,j}(s\xi_j+y)= \frac{\I}{4|\xi_j|^2}\int_0^s q^{(2)}(\tilde{s}\xi_j+y)\,\df\tilde{s}$
for a small neighbourhood $P_\delta$ of $p$. Since $\beta(p)$ is known and non-zero for almost all $p$, we can let $\delta\rightarrow0$ to recover the quantity
\begin{equation}\label{3-to-1}\sum_{j=0}^2(d_{3,j}(p)-f_{3,j}(p)) = 0.\end{equation}
To conclude the proof, we use the above equation \eqref{3-to-1} to get the unique recovery of $q$. \tblue{We use a similar setup as discussed in the previous steps.} Let $\widehat{\xi}$ to be the unit vector in the direction of $\xi_2$ and $\xi_2=\lambda\widehat{\xi}$. Further, assume $s_0\in\R$ be such that $\gamma_{r_2,\widehat{\xi}}(s_0)=p$. Then, we have
\begin{align*}
d_{3,2}(p)-f_{3,2}(p)&= \frac{\I}{4|\xi_2|^2}\int_0^{s_0/\lambda}(q^{(1)}-q^{(2)})(\tilde{s}\xi_2+r_{2})\,\df \tilde{s} =\frac{\I}{4\lambda^3}\int_0^{s_0} (q^{(1)}-q^{(2)})(s\widehat{\xi}+r_{2})\,\df s,
\end{align*} 
One may again check as in the previous section that $d_{3,j}(p)-f_{3,j}(p)=\mathcal{O}(1)$ as $\lambda\rightarrow0$ for $j\neq2$. Therefore, we have
\[\lim_{\lambda\rightarrow0}\Bigg(\lambda^3\sum_{j=0}^2\left(d_{3,j}(p)-f_{3,j}(p) \right) \Bigg)=\int_0^{s_0} (q^{(1)}-q^{(2)})(s\widehat{\xi}+r_{2})\,\df s = 0.\]
This is precisely the truncated ray transform of $q$, which we can differentiate with respect to $s_0$ to get $q^{(1)}(p)=q^{(2)}(p)$. Since the choice of $p\in M\setminus\Gamma$ was arbitrary, we uniquely recover the coefficient $q(t,x)$ everywhere. This completes the proof.
\end{proof}

\section{geometric case}\label{sec:geometric case}
\subsection{Gaussian beam construction}\label{subsec:Gaussian_beam}
In this section, we present the construction of special solutions of  $  (\I \PD_t +\Delta^2_g + \Delta_h + q)u=0$ in the manifold settings. We closely follow the ideas from \cite{Dos_Jems}*{Section 3}. These special solutions concentrate near a given geodesic and are usually referred to as Gaussian beams in the literature. We next embed the manifold $(M,g)$ into a closed manifold $(N,g)$. To this end, we recall some results without their proofs.
\begin{lemma}[\cite{Dos_Jems}*{Lemma 3.4}]
Let $F$ be a $C^{\infty}$ diffeomorphism of $(a,b) \times \{0\} \subset \Rn$ to a smooth $n$ dimensional manifold such that $F|_{(a,b) \times \{0\}}$ is injective and $DF(t,0)$ is invertible for $ t\in (a,b)$. If $[a_0,b_0]$ is a closed subinterval of $(a,b)$, then $F$ is a $C^{\infty}$ diffeomorphism in some neighbourhood of $[a_0,b_0]\times \{0\} \in \Rn$. 
\end{lemma}
\begin{lemma}[\cite{Dos_Jems}*{Lemma 3.5}] \label{lem:Fermi}
Let $(N,g)$ be a closed Riemannian manifold. Let $ \gamma$ be a unit speed geodesic with no loops. Given a closed subinterval $[a_0,b_0]$ of $(a,b)$ such that $\gamma|_{[a_0,b_0]}$ intersects itself only at finitely many times $\{r_j\}_{j = 1}^{N}$ with $r_0 = a_0 < r_1 < \cdots < r_{N} < r_{N+1} = b_0$,
there exists an open cover $\{(U_j,\phi_j)\}_{j=0}^{N+1}$ of $\gamma[a_0,b_0]$ consisting of coordinate neighbourhoods having the following properties:
\begin{itemize}
\item [1.] $\phi_j(U_j)= I_{j}\times B$ where $I_0 = (a, r_0 + \epsilon)$, $I_j = (r_{j-1} - 2\epsilon, r_j + \epsilon)$ with $j = 1,\cdots,N$, $I_{N+1} = (r_N - 2\epsilon, b)$, and $B=B(0,\delta)$ is an open ball in $\mathbb{R}^{n-1}$, where $\delta$ can be taken arbitrarily  small.
\item [2.] $\phi_j(\gamma)=(r,0,\cdots,0)$ for $ r\in I_j$.
\item [3.] $r_j\in I_j$ and $ \Bar{I}_j\cap \Bar{I}_k$ is empty unless $ |j-k|\le 1$.
\item [4.] $\phi_j=\phi_k$ on $ \phi^{-1}_j\left((I_j\cap I_k) \times B\right)$.
\item [5.] Further, the metric in these coordinates satisfies $g^{jk}|_{\gamma(t)} = \delta^{jk}$, $\partial_i g^{jk}|_{\gamma(t)} = 0$. 
\end{itemize}

\end{lemma}
\noindent We now present the construction of special solutions of 
\begin{align}\label{!LS}
\left\{\begin{array}{rll}
(\I \PD_t +\Delta^2_g + \Delta_h + q)u&=0 & \mbox{in $(0, T)\times  M$}\\
u(0,x)&=0 & \mbox{in $M$} 
\end{array}\right.
\end{align}
known as Gaussian beams along a geodesic $\gamma$. \tblue{We first assume that $\gamma$ is non-self-intersecting. Therefore, it can be entirely contained within a single coordinate chart $(W, \Phi)$ with local coordinates $(r, y)\in I \times B(0, \delta)$ for some $\delta>0$ and some open interval $I=(a,b)$. We now consider an approximate Gaussian beam of the following form:}
\[U_\tau(t,r,y)=e^{i(\tau\psi(r,y)+\tau^4 t)}a_\tau(t,r,y).\] 
The phase function $\psi\in C^\infty(M)$ and amplitude $a_\tau\in C^\infty_0((0,T)\times M)$ are unknown, and we will determine these below. We write $P_{h,q} := \I \PD_t +\Delta^2_g +\Delta_h + q$ and compute
\begin{equation}\label{!SU}P_{h,q} \big(e^{\I(\tau\psi+\tau^4 t)}a_\tau\big)=e^{\I(\tau\psi+\tau^4 t)}\Big(\tau^4(\mathcal{E}\psi)a_\tau-2i\tau^3\mathcal{T}_1a_\tau - \tau^2 \mathcal{T}_2a_\tau + \I\tau \mathcal{T}_3 a_\tau + P_{h,q}a_\tau\Big),\end{equation}
where the operators $\mathcal{E}: C^\infty(M)\rightarrow C^\infty(M),\ \mathcal{T}_1, \mathcal{T}_2, \mathcal{T}_3:C^\infty((0,T)\times M)\rightarrow C^\infty((0,T)\times M)$ are defined by
\begin{align}
\mathcal{E}\psi:=&|\df \psi|_g^2|\df \psi|_g^2-1,\\
\mathcal{T}_1a_\tau:=&|\df \psi|_g^2\Delta_g\psi a_\tau + 2|\df \psi|_g^2\lr \df \psi,\df a_\tau\rn_g+\lr \df \psi,\df (|\df \psi|_g^2)\rn_g a_\tau,\\
\mathcal{T}_2a_\tau:=& 2\lr \df \psi,\df \Delta_g\psi \rn_ga_\tau + 4\lr \df \psi,\df a_\tau \rn_g\Delta_g\psi +|\Delta_g\psi|^2a_\tau + 4\lr \df \psi,\df (\lr \df \psi,\df  a_\tau\rn_g) \rn_g + \Delta_g(|\df \psi|^2_g) a_\tau \nonumber \\& \ + 2\lr \df (|\df \psi|^2_g), \df a_\tau \rn_g + 2\Delta_g a_\tau |\df \psi|^2_g + |\df \psi|^2_h a_\tau,\\
\mathcal{T}_3a_\tau:=& \Delta^2_g\psi a_\tau + 2\lr \df (\Delta_g\psi), \df a_\tau \rn_g + 2 \Delta_g\psi \Delta_g a_\tau + 2\Delta_g(\lr \df \psi, \df a_\tau \rn_g) + 2 \lr \df \psi, \df (\Delta_ga_\tau)\rn_g + \Delta_h\psi a_\tau \nonumber\\ & \  + 2\lr \df \psi, \df a_\tau\rn_h.
\end{align}
\tblue{Our objective is to determine the phase function $\psi$ and amplitude $a_\tau$ on $(M, g)$ by solving the eikonal equation $\mathcal{E}\psi = 0$ and associated transport equations, respectively. Because the exponential map is a global diffeomorphism on simple manifolds (see \cite{DOS}), these functions can be constructed explicitly. Consequently, on simple manifolds, one can utilize a global coordinate system to construct solutions globally on $(M, g)$. This approach, however, fails for non-simple manifolds. Instead, on a non-simple manifold, the strategy is to prescribe the Taylor series for the functions $\psi$ and $a_\tau$ within a tubular neighborhood of the geodesic $\gamma$. We achieve this by imposing the following conditions on the eikonal and transport equations}
\begin{align}\label{!eikonal}
\frac{\PD^\alpha}{\PD y^\alpha}(\mathcal{E}\psi)(r,0,\cdots,0)&=0\quad \forall r\in I,\\ \label{!transport_1}
\frac{\PD^\alpha}{\PD y^\alpha}(\mathcal{T}_1a_\tau)(r,0,\cdots,0)&=0\quad \forall r\in I, 
\end{align}
where $ \alpha\in (\N\cup\{0\})^{n-1}$ is a multi-index with $ \abs{\alpha}\le N.$ 
\subsection{Construction of phase function}
\tblue{Let us construct a phase function $\psi$ that satisfies the eikonal equation \eqref{!eikonal}. We expand $\psi$ as $\psi = \sum_{j=0}^N \psi_j(r, y)$, where each $\psi_j(r, y)$ is homogeneous of degree $j$ with respect to $y$. Setting $|\alpha| = 0$ in equation \eqref{!eikonal}, we obtain the following equation on $\gamma$:}
\[\left(\sum_{k,l=1}^ng^{kl}|_{\gamma}\p_k\psi\p_l\psi\right) \left(\sum_{p,q=1}^ng^{pq}|_{\gamma}\p_p\psi\p_q\psi\right)=1,\quad \forall \,\,r\in I.\] 
Using the identities $ \displaystyle g^{kl}|_\gamma=\delta^{kl}$ and $\displaystyle g^{pq}|_\gamma=\delta^{pq}$, we have 
\[\left(\sum_{l=1}^n(\p_l\psi)^2\right)\left(\sum_{p=1}^n(\p_p\psi)^2\right)=1.\]
Next, take $|\alpha|=1$ in (\ref{!eikonal}) and use  $g^{kl}=\delta^{kl}$, $\p_ig^{kl}=0$ and $g^{pq}=\delta^{pq}$, $\p_ig^{pq}=0$ on $\gamma$ to deduce that 
$$\left(\sum_{l=1}^n\p_{il}^2\psi\p_l\psi \right)\left(\sum_{p=1}^n(\p_p\psi)^2\right)=0\quad\forall r\in I\, \mbox{ and } i = 2, \dots n.$$
These equations are satisfied by setting \begin{equation}\label{!psi01}
\psi_0=r\quad\textrm{and}\quad\psi_1=0.
\end{equation}
\tblue{Proceeding in similar fashion, we note that evaluating equation \eqref{!eikonal} at $|\alpha| = 2$ yields the equivalent form}
\begin{align}
&\sum_{k,l=1}^n \Big( 2g^{kl}\p_{ijk}^3\psi\p_l\psi+2g^{kl}\p_{ik}^2\psi\p_{jl}^2\psi+\p_{ij}^2g^{kl}\p_k\psi\p_l\psi+4\p_ig^{kl}\p_{jk}^2\psi\p_l\psi\Big) \Big(\sum_{p,q=1}^n\p_p\psi \p_q\psi\Big) \nonumber\\&+ \Big(\sum_{k,l=1}^n \p_jg^{kl}\p_{k}\psi\p_l\psi + g^{kl}\p_{jk}^2\psi\p_l\psi + g^{kl}\p_k\psi\p_{jl}^2\psi\Big)\Big(\sum_{p,q=1}^n \p_ig^{pq}\p_{p}\psi\p_q\psi + g^{pq}\p_{ip}^2\psi\p_q\psi + g^{pq}\p_p\psi\p_{iq}^2\psi\Big)=0,
\end{align}
for any $i,j=2,\dots,n$. \tblue{Since the conditions $\partial_i g^{kl} = 0$, $\partial_j \psi = \delta_{j1}$, and $\partial_1 \partial_j \psi = 0$ for $j \neq 1$ hold on $\gamma$, the above equation simplifies to}
\begin{equation}\label{!ric}\p_{ij}^2g^{11}+2\p_{1ij}^3\psi+2\sum_{k=2}^n\p_{ki}^2\psi\p_{kj}^2\psi=0\quad\forall r\in I.
\end{equation}
\tblue{Accordingly, we introduce the quadratic form $\psi_2(r,y) = \frac{1}{2}\sum_{i,j=1}^{n-1}H_{ij}(r)y^iy^j$. Here, $H$ denotes a smooth, symmetric, complex matrix whose imaginary part is strictly positive-definite on $I$ (that is, $\Im H(r) > 0$ for all $r \in I$). Noting that $\partial_{y^i}\partial_{y^j}\psi_2 = H_{ij}$, we deduce that $H(r)$ must satisfy the matrix Riccati equation in order to satisfy \eqref{!ric}}
\begin{equation}\label{!riccati}\frac{d}{dr}H+H^2+D=0\quad \forall r\in I,\quad \mbox{here }D_{ij}=\frac{1}{2}\p_{y^i}\p_{y^j}g^{11}.
\end{equation}
\tblue{We solve this ODE using \cite{Lassas2001boundary}*{Lemma 2.56} to obtain the required $H$. With this $H$, we are able to determine $\psi_2$. Following a similar approach, one can determine the functions $\psi_j$ for $j\ge 3.$}
\subsection{Construction of amplitudes} In this subsection, we construct the amplitudes satisfying \eqref{!transport_1}. Let us write $a_\tau$ as
\begin{equation}\begin{split}\label{eq transport}
a_\tau(t,r,y)=\sum_{k=0}^N\tau^{-k}a_k(t,r,y),\quad a_k(t,r,y)=\phi(t)\chi\left(\frac{|y|}{\delta'}\right)\sum_{j=0}^N\,a_{k,j}(t, r,y).\end{split}
\end{equation}
In the above relation, $a_{k,j}$ is a homogeneous polynomial of degree $j$ in the variable $y$, $\chi\in C^\infty_0(\R)$ satisfies $\displaystyle \chi(s)= \left\{\begin{array}{cc}
    1, &  |s|\leq 1/4\\
    0, &  |s|\geq 1/2
\end{array}\right.$ and $\phi\in C^\infty_0(0,T)$ is a smooth cutoff function.
From the definition of $\mathcal{T}_1a_\tau$, we see that $a_0$ satisfies the following relation for any multi-index $\alpha$ with $ \abs{\alpha}\le N$ on $\gamma$
\begin{align}\label{eq_transport_v0}
\PD^{\alpha}_{y}\left( | \df\psi|^2_g\Delta_g\psi a_0 + 2|\df\psi|_g^2\lr \df\psi,\df a_0\rn_g+\lr \df\psi,\df(|\df\psi|_g^2)\rn_g a_0 \right) = 0.
\end{align}
By Lemma \ref{lem:Fermi} we have  that $g^{kl}|_{\gamma(r)}=\delta^{kl}$ and for $ |\alpha|=0$, this preceding equation reduces to $$\frac{d}{dr}a_{0,0}+\frac{1}{2}\tr(H)a_{0,0}=0\quad\forall r\in I. $$ 
This implies 
\begin{equation}\label{!amp00}
a_{0,0} (r) =c_0e^{-\frac{1}{2}\int_{r_0}^r\tr(H)(s)\df s}, \quad c_0 = a_{0,0}(r_0).
\end{equation}
\noindent \tblue{We can construct the subsequent terms $a_{0,j}$ (for $j=1,\dots,N$) by evaluating equation \eqref{eq_transport_v0} at $|\alpha|=j$, which reduces the problem to solving a sequence of linear first-order ODEs along $\gamma$. For instance, taking $|\alpha|=1$ in \eqref{eq_transport_v0} and applying the definition of $\mathcal{T}_1$, we obtain an equation of the form}
\begin{align}\label{eq_transport_v_0j}
\p_ra_{0,1}+\frac{1}{2}\tr(H)a_{0,1}+ \nabla_y\psi_2\cdot \nabla_y a_{0,1}+\Pi_1=0\quad\forall r\in I\quad \mbox{and} \quad a_{0,1}(r_0)=0,
\end{align}
where $\Pi_1$ is a homogeneous polynomial of degree $1$ in $y$, whose coefficients depend only upon $a_{0,0}$ and $\{\psi_l\}_{l=0}^{3}$. To construct the remaining amplitude terms $a_k$, we proceed by solving the equations 
\begin{align}\label{eq_transport_v_1}
2\I \mathcal{T}_1a_1+\mathcal{T}_2a_0 &= 0\\ \label{eq_transport_v_2}
2\I \mathcal{T}_1a_2+\mathcal{T}_2a_1-\I \mathcal{T}_3a_0 &= 0\\ \label{eq_transport_v_k}
2\I \mathcal{T}_1a_k+\mathcal{T}_2a_{k-1}-\I \mathcal{T}_3a_{k-2} - P_{h,q}a_{k-3} &= 0 \ \mbox{ for $k\ge 3$ up to $N$-th order on $\gamma$.}
\end{align}
\tblue{The proof is largely identical to the argument for $a_0$, so we omit the details and instead refer the reader to \cites{Dos_Jems, Kenig_Salo_Survey}. To establish an analogue of \eqref{e_split} for approximate Gaussian beams, we evaluate equation \eqref{eq_transport_v_1} at $|\alpha|=0$ along $\gamma$. This yields}
\begin{equation}
\begin{aligned}
4\I\left(\frac{d}{dr}a_{1,0}+\frac{1}{2}\tr(H)a_{1,0}\right)=&-2\frac{d}{dr}(\Delta_g\psi)a_{0,0}-4\tr(H)\frac{d}{dr}a_{0,0}-(\tr(H))^2a_{0,0}-4\frac{d^2}{dr^2}a_{0,0} \nonumber\\& \ -\Delta_g(|\df\psi|^2_g)a_{0,0}-2\Delta_g a_0-h^{11}a _{0,0}.
\end{aligned}
\end{equation}
Therefore, we may choose
\begin{align}\label{!split_1}
  a_{1,0}(r)=c_{1,0}(r)+d_{1,0}(r)  
\end{align}
where 
\begin{align*}
   c_{1,0}(r)&=e^{-\frac{1}{2}\int_{r_0}^r\tr(H)(s)\df s}\int_{r_0}^r\frac{\I}{4}e^{\frac{1}{2}\int_{r_0}^s\tr(H)\df \tilde{s}}\left(2\frac{d}{dr}(\Delta_g\psi)a_{0,0}+4\tr(H)\frac{d}{dr}a_{0,0}+(\tr(H))^2a_{0,0} \right. \nonumber \\ & \left. \qquad+4\frac{d^2}{dr^2}a_{0,0} +\Delta_g(|\df\psi|^2_g)a_{0,0}+2\Delta_g a_0\right)(s) \df s \nonumber\\
   d_{1,0}(r) &=e^{-\frac{1}{2}\int_{r_0}^r\tr(H)(s)\df s}\int_{r_0}^r\frac{\I}{4}e^{\frac{1}{2}\int_{r_0}^s\tr(H)\df \tilde{s}}h^{11}a_{0,0}(s)\df s. 
\end{align*}
We consider equation \eqref{eq_transport_v_2} with $|\alpha|=0$,  on $\gamma$. This gives 
\begin{equation}\begin{aligned}
4\I\left(\frac{d}{dr}a_{2,0}+\frac{1}{2}\tr(H)a_{2,0}\right)=&-2\frac{d}{dr}(\Delta_g\psi)a_{1,0}-4\tr(H)\frac{d}{dr}a_{1,0}-(\tr(H))^2a_{1,0}-4\frac{d^2}{dr^2}a_{1,0} +2\I h^{1j} \p_ja_0 \\
&  \quad-\Delta_g(|\df\psi|^2_g)a_{1,0}-2\Delta_g a_1-h^{11}a _{1,0}+\I\Delta_g^2\psi a_{0,0} + 2\I \lr \df (\Delta_g\psi), \df a_0 \rn_g  \\ & \quad +2\I \tr(H) \Delta_g a_0   \quad+2\I \Delta_g(\lr \df\psi, \df a_0 \rn)+2\I \frac{d}{dr}(\Delta_ga_0) \quad+ \I(\Delta_h\psi) a_{0,0}.
\end{aligned}
\end{equation}
As above, we choose 
\begin{align}\label{!split_2}
    a_{2,0}(r)=c_{2,0}(r)+d_{2,0}(r),
\end{align}
where
\begin{align*}
c_{2,0}(r)&=e^{-\frac{1}{2}\int_{r_0}^r\tr(H)(s)\df s}\int_{r_0}^r\frac{\I}{4}e^{\frac{1}{2}\int_{r_0}^s\tr(H)\df \tilde{s}}\left(2\frac{d}{dr}(\Delta_g\psi)a_{1,0}+4\tr(H)\frac{d}{dr}a_{1,0}+(\tr(H))^2a_{1,0} \right.\\ & \left.\quad +4\frac{d^2}{dr^2}a_{1,0} + \Delta_g(|\df\psi|^2_g)a_{1,0}+2\Delta_g a_1+h^{11}a_{1,0}-2\I \lr \df (\Delta_g\psi), \df a_0 \rn_g \right.\\ & \left. \qquad- \I\Delta_g^2\psi a_{0,0} -2\I \tr(H) \Delta_g a_0 -2\I \Delta_g(\lr \df\psi, \df a_0 \rn)-2\I \frac{d}{dr}(\Delta_ga_0) \right)(s) \df s\\
d_{2,0}(r)&=e^{-\frac{1}{2}\int_{r_0}^r\tr(H)(s)\df s}\int_{r_0}^r\frac{1}{4}e^{\frac{1}{2}\int_{r_0}^s\tr(H)\df \tilde{s}}\left((\Delta_h\psi) a_{0,0} + 2h^{1j} \p_ja_0\right)(s)\df s.
\end{align*}
We consider equation \eqref{eq_transport_v_k} with $|\alpha|=0$, on $\gamma$. This gives \begin{equation}
\begin{aligned}
4\I\left(\frac{d}{dr}a_{3,0}+\frac{1}{2}\tr(H)a_{3,0}\right)=&-2\frac{d}{dr}(\Delta_g\psi)a_{2,0}-4\tr(H)\frac{d}{dr}a_{2,0}-(\tr(H))^2a_{2,0}-4\frac{d^2}{dr^2}a_{2,0} + 2\I h^{1j} \p_ja_1 \nonumber\\
& -\Delta_g(|\df \psi|^2_g)a_{2,0}-2\Delta_g a_2-h^{11}a _{2,0}+\I\Delta_g^2\psi a_{1,0} +2 \I \lr \df(\Delta_g\psi), \df a_1 \rn_g  \nonumber\\ & +2\I \tr(H) \Delta_g a_1  +2\I \Delta_g(\lr \df\psi, \df a_1 \rn)+2\I \frac{d}{dr}(\Delta_ga_1) + \I(\Delta_h\psi) a_{1,0} + P_{h,q}a_{0}.
\end{aligned}
\end{equation}
Choose
\begin{align}\label{!split_3}
    a_{3,0}(r)=c_{3,0}(r)+d_{3,0}(r)
\end{align}
with 
\begin{align*}
 c_{3,0}(r)&=e^{-\frac{1}{2}\int_{r_0}^r\tr(H)(s)\df s}\int_{r_0}^r\frac{\I}{4}e^{\frac{1}{2}\int_{r_0}^s\tr(H)\df \tilde{s}}\bigg(2\frac{d}{dr}(\Delta_g\psi)a_{2,0}+4\tr(H)\frac{d}{dr}a_{2,0}+(\tr(H))^2a_{2,0} + 4\frac{d^2}{dr^2}a_{2,0}  \\ & \left. \qquad +\Delta_g(|\df\psi|^2_g)a_{2,0}+2\Delta_g a_2+h^{11}a_{2,0}-2\I \lr \df(\Delta_g\psi), \df a_1 \rn_g - \I\Delta_g^2\psi a_{1,0} -2\I \tr(H) \Delta_g a_1 \right.\\ & \left. \quad \qquad -2\I \Delta_g(\lr \df\psi, \df a_1 \rn) -2\I \frac{d}{dr}(\Delta_ga_1) -\I (\Delta_h\psi) a_{1,0} -2\I h^{1j} \p_ja_1 \bigg)(s) \df s\right. \\
d_{3,0}(r) &=-e^{-\frac{1}{2}\int_{r_0}^r\tr(H)(s)\df s}\int_{r_0}^r\frac{\I}{4}e^{\frac{1}{2}\int_{r_0}^s\tr(H)\df \tilde{s}}  P_{h,q}a_{0}(s)\df s.   
\end{align*}
We constructed an approximate Gaussian beam of order $N$ of the form $U_\tau=e^{i(\tau\psi+\tau^4t)}a_\tau$, which  satisfies the following properties:
\begin{enumerate}
\item The equations (\ref{!eikonal}), and  (\ref{!transport_1}) are hold.
\item The phase  function $\psi$ satisfies $\Im(\psi)|_\gamma=0$, that is, the imaginary part of $\psi$ vanishes along $\gamma$.
\item There exists a constant $c\geq0$ such that $\Im(\psi)(r,y)\geq c|y|^2$ for all $(r,y)\in \Phi(W)$, where $(W, \Phi)$ is a local coordinate chart.
\end{enumerate}
\begin{proposition}\label{!GBProp}
Let $U_\tau=e^{i(\tau\psi+\tau^4t)}a_\tau$ be an approximate Gaussian beam of order $N$. Then for $\tau\gg1$ we have
\begin{equation}\label{!estimates}\norm{P_{h,q}U_\tau}_{H^s((0,T)\times W)}\lesssim\tau^{\frac{7-N}{2}+4s},\qquad\norm{U_\tau}_{H^s((0,T)\times W)}\lesssim \tau^{4s}.
\end{equation}
\end{proposition}
\begin{proof}
Since $\Im \psi$ satisfies $\Im(\psi)(r,y)\geq c|y|^2$, it follows that $|e^{i(\tau\psi+\tau^4 t)}|\leq Ce^{-\frac{1}{4}c\tau|y|^2}$ for small enough $y$. Therefore, using the smoothness and compact support of $a_\tau$, we obtain the estimate 
\[\begin{split}\norm{U_\tau}_{H^s((0,T)\times W)}\leq& C\tau^{4s}\norm{U_\tau}_{L^2((0,T)\times W)}\lesssim\tau^{4s}\norm{e^{-\frac{1}{4}c\tau|y|^2}\chi(|y|/\delta')}_{L^2((0,T)\times W)}=\mathcal{O}(\tau^{4s}),\end{split}\] 
for small enough $\delta'$. By substituting  \eqref{!eikonal} and \eqref{!transport_1} into the identity given by \eqref{!SU}, a similar deduction shows that
\begin{equation}\label{!1}|\p_z^\sigma P_{h,q}U_\tau|\lesssim\tau^{|\sigma|}|e^{i(\tau\psi+\tau^4 t)}|\Big(C_0\tau^4|y|^{N+1}+C_1\tau^{-N}\Big), \quad \mbox{ for any } \sigma\in\N^{n}.
\end{equation}
This further implies 
\begin{equation}\label{!2}
|\p_t^m P_{h,q}U_\tau|\lesssim \tau^{4m}|e^{i(\tau\psi+\tau^4t)}|\Big(C_0\tau^4|y|^{N+1}+C_1\tau^{-N}\Big).
\end{equation} 
Finally, from \eqref{!1} and \eqref{!2}, we have 
\[\begin{split}\norm{P_{h,q}U_\tau}_{H^{s}((0,T)\times W)}\lesssim\tau^{4s}\norm{e^{-\frac{1}{4}c\tau|y|^2}(\tau^4|y|^{N+1}+\tau^{-N})}_{L^2((0,T)\times W)}= \mathcal{O}(\tau^{\frac{7-N}{2}+4s}).
\end{split}\]
\end{proof}
\tblue{To obtain exact solutions to \eqref{!LS} from the approximate Gaussian beams, we introduce an appropriate remainder term. Writing the exact solution as $u = U_\tau + R_\tau$, we require that the remainder satisfies}
\begin{align}\label{eq_rmdr}
&P_{h,q}R_\tau = -P_{h,q}U_\tau \ \mbox{in} \ (0,T)\times M \nonumber\\
&R_\tau|_{t=0} = 0.
\end{align}
Moreover, by Proposition \ref{!GBProp} we get
\begin{align}
\|R_\tau\|_{H^s((0,T)\times M)} \lesssim \|P_{h,q}U_\tau\|_{H^s((0,T)\times M)} \lesssim \tau^{\frac{7-N}{2}+4s}.
\end{align}
\subsection{Boundary source determination}\label{subsec:boundary_source_determine}
\begin{lemma}\label{lm:determining_amplitudes}
Let $\Lc_{h_{1},q_1}= \Lc_{h_{2}, q_2}$ in $ (0,T) \times \Gamma$, and let $U^{(j)}_{\tau}= e^{\I (\tau \psi +\tau^4 t)} a_\tau^{(j)}$ be an approximate Gaussian beam solution of \eqref{!LS} concentrate near a geodesic $\gamma$.  Then for each integer $k\ge 0$, the amplitudes $a^{(1)}_k=a^{(2)}_k$ on $\Gamma \cap \gamma$  up to higher order terms.
\end{lemma}
\begin{remark}
We follow the argument similar to one used in \cite{LOSST_NLS_2025}. However, one could also use a different argument based on the stationary phase lemma; see \cite[Theorem 3.1]{simon2024gaussianbeamsinverseproblems}.
\end{remark}
\begin{proof}
The proof is based on induction on $k$.  Fix a geodesic $ \gamma$ and let $(r,y)$ be the Fermi coordinates in a tubular neighbourhood of $\gamma$. Let $\ell$ be the length of the geodesic $\gamma$ and we choose $0<r_1<r_2<\ell$ such that $(r,y)\in \Gamma$ when $r<r_1$ or $r>r_2$.  Now proving Lemma \ref{lm:determining_amplitudes} is equivalent to proving the following: for any integer $M\ge 0$ and for all integers $k\ge 0$
\begin{align}\label{induction_argument}
(\PD^{\alpha}_{y} a_k^{(1)})(r, 0)=  (\PD^{\alpha}_{y} a_k^{(2)})(r, 0) \ \mbox{ for $r<r_1$ or $r>r_2$, and  multi-indices $|\alpha|=M$.}
\end{align}
We divide the proof into several steps. 
 \smallskip
    
    \subsubsection*{\textbf{Step 1}} 
We show that 
	$ (\PD^{\alpha}_{y} a_0^{(1)})(r, 0)=  (\PD^{\alpha}_{y} a_0^{(2)})(r, 0) \ \mbox{for any multi-index $|\alpha|=M$.}$\smallskip \\
	For $\abs{\alpha}=0$ we have that $ a^{(1)}_0-a^{(2)}_0$ satisfies the transport equation 
	\[  \mathcal{T}_1(a^{(1)}_0-a^{(2)}_0)=0 \quad \mbox{on $\gamma$ and  $(a^{(1)}_0-a^{(2)}_0)(0)=0$.}\]
	As $a^{(1)}_0-a^{(2)}_0 = 0$ at $\gamma(0)$,
	this implies $(a^{(1)}_0-a^{(2)}_0)(r,0)=0$ by unique solvability of ordinary differential equations. Suppose that 
	\begin{align}\label{eq_induction_1}
		(\PD^{\alpha}_{y} a_0^{(1)})(r, 0)=  (\PD^{\alpha}_{y} a_0^{(2)})(r, 0) \ \mbox{for any multi-index $|\alpha| \le M$. }
	\end{align}
	As $\mathcal T_1a^{(j)}_0$ vanishes to a high order on $\gamma$,
	it follows from \eqref{eq_induction_1} that
	for a multi-index $|\alpha| = M + 1$ there holds
	\begin{align}
		\mathcal T_1\p_y^\alpha (a^{(1)}_0-a^{(2)}_0) = 0 \ \text{on $\gamma$}.
	\end{align}
	As $\p_y^\alpha (a^{(1)}_0-a^{(2)}_0) = 0$ at $\gamma(0)$, we see that 
	\eqref{eq_induction_1} holds for $|\alpha| = M + 1$. This completes the induction argument.
	
	\subsubsection*{\textbf{Step 2}} Here  we show that \eqref{induction_argument} holds true for $k=K$, and $ |\alpha|=0$. We argue by induction and assume that \eqref{induction_argument} is true for all $k\le K-1$. 
	Let $\eta \in C_c^{\infty}(M)$ such that $\eta=1$ in $M\setminus\Gamma$.
	We set $f_j=(\I \PD_t +\Delta^2_g +\Delta_{h_{j}} + q_j)(\eta u_j)$, 
	where $u_j$ solves 
	\begin{equation}
		\begin{cases}
			(\I \PD_t +\Delta^2_g +\Delta_{h_{j}} + q_j)u_j=0\quad &\mbox{in} \quad (0,T) \times M\\
			u_j=0 \quad & \mbox{at} \quad t=0,
		\end{cases}
	\end{equation}
	and coincides with the approximate Gaussian beam 
	\begin{align}
		U^{(j)}_{\tau}= e^{\I (\tau \psi(r,y)+ \tau^4 t)} a_{\tau}^{(j)}(t,r,y)= e^{\I (\tau \psi(r,y)+ \tau^4 t)} \sum_{k=0}^N\, \tau^{-k}a_k^{(j)} (t,r,y)
	\end{align}
	up to a small error of order $\mathcal{O}(\tau^{\frac{7-N}{2}})$ in the sense of $L^2$.
	Clearly,  the function $ \eta u_j$ satisfies 
	\begin{align}
		\begin{cases}
			(\I \PD_t +\Delta^2_g +\Delta_{h_{j}} + q_j)u_j=f_j\quad &\mbox{in} \quad (0,T) \times M\\
			u_j= {\p_{\nu_g} u_j}=0 \quad & \mbox{on} \quad \PD M\\
			u_j=0 \quad & \mbox{at} \quad t=0,
		\end{cases}
	\end{align}
	with sources $f_j$ supported in $(0,T) \times \Gamma$. 
	Further, let us consider a solution of 
	\begin{align}
		\begin{cases}
			(\I \PD_t +\Delta^2_g +\Delta_{h_{2}} + q_2)w=0\quad &\mbox{in} \quad (0,T) \times M\\
			w=0 \quad & \mbox{at} \quad t=T.
		\end{cases}
	\end{align}
	that coincides with the approximate Gaussian beam
	\begin{align}
		e^{\I (\tau \psi(r,y)+ \tau^4 t)} \sum_{k=0}^N\, \frac{w_k (t,r,y)}{\tau^k},
	\end{align}
	up to a small error in the same sense.
	
	Let $\tilde \eta \in C_c^{\infty}(M)$ such that $\tilde \eta=1$ in $M\setminus\Gamma$.
	We set $\tilde{f}=(\I \PD_t +\Delta^2_g +\Delta_{h_{2}} + q_2)(\tilde \eta w)$. Analogously to, we have
	\begin{align}\label{eq_1.7}
		\lr (\Lc_{h_{j},q_j}-\Lc_{h_{2},q_2})f_j,\tilde{f}\rn_{L^2((0,T)\times\Gamma)}= \lr \eta u_j,\tilde{f}\rn_{L^2((0,T)\times\Gamma)} - \lr f_j,\Tilde{\eta} w\rn_{L^2((0,T)\times\Gamma)}.   
	\end{align}
	Subtracting \eqref{eq_1.7}
	for $j=1$ from \eqref{eq_1.7} for $j=2$ we obtain
	\begin{equation}\label{integral_identity}
		\begin{aligned}
			0=& \lr \eta (u_1-u_2),\tilde{f}\rn_{L^2((0,T)\times\Gamma)} - \lr f_1-f_2,\Tilde{\eta} w\rn_{L^2((0,T)\times\Gamma)}\\
			0=  &\agl[ e^{\I (\tau \psi(r,y)+ \tau^4 t)} \eta\tau^{-K}( a^{(1)}_{K}-a^{(2)}_K), e^{\I (\tau \psi(r,y)+ \tau^4 t)} 4(\I \tau)^3|\df\psi|_g^2\lr \df\psi,\df\tilde{\eta}\rn_g \,w_0+\mathcal{O}(\tau^2)]_{L^2((0,T)\times\Gamma)}\\&- \agl[ 4(\I \tau)^3|\df\psi|_g^2\lr \df\psi,\df {\eta}\rn_g e^{\I (\tau \psi(r,y)+\tau^4 t)} \tau^{-K} ( a^{(1)}_{K}-a^{(2)}_K) + \mathcal{O}(\tau^{2}) , e^{\I (\tau \psi(r,y) +\tau^4 t)}\Tilde{\eta} w_0]_{L^2((0,T)\times\Gamma)}.
		\end{aligned}
	\end{equation}
	We next consider the coefficient of $\tau^{-K+3}$ from the above integral identity and simplify it. To this end, recall that we  choose $0 < r_1 < r_2 < \ell$ such that in Fermi coordinates $(r,y) \in \Gamma$ when $r < r_1$ or $r > r_2$ and $\eta(r,y) = \tilde \eta(r,y) = 0$ when $r > \ell$. Let $I$ be the coefficient of $\tau^{-K+3}$ in the above expansion. Since $\eta$ and $\tilde{\eta}$ are real valued and $\mathcal{T}(\cdot)= |\df\psi|_g^2\lr \df\psi,\df{\cdot}\rn_g$ is a complex valued function, this implies
	\begin{align}
		I&= 4\I \int_0^T \int_{\Gamma} e^{-2\tau \Im \psi} (\eta  \:\overline{\mathcal{T} \Tilde{\eta}}+ \Tilde{\eta}\mathcal{T}\eta) w_0 \, ( a^{(1)}_{K}-a^{(2)}_K) \df v_g\,\df t \nonumber\\
		&=4\,\I \int_0^T  \int_0^{r_1} \int_{\R^{n-1}} e^{-2\tau \Im \psi} (\eta  \:\overline{\mathcal{T} \Tilde{\eta}}+ \Tilde{\eta}\mathcal{T}\eta) w_0 \, ( a^{(1)}_{K}-a^{(2)}_K) \sqrt{|g|}\, \df r\,\df y\,\df t \notag\\& \quad+4\,\I \int_0^T  \int_{r_2}^{\ell}\int_{\R^{n-1}} e^{-2\tau \Im \psi} (\eta  \:\overline{\mathcal{T} \Tilde{\eta}}+ \Tilde{\eta}\mathcal{T}\eta) w_0 \, ( a^{(1)}_{K}-a^{(2)}_K) \sqrt{|g|}\, \df r\,\df y\,\df t.
	\end{align}
	Since $ a_K^{(1)}= a_K^{(2)}$ up to higher order before the geodesic enters the domain $M\setminus\Gamma$, that is when $r<r_1$, there holds \[ \int_0^T  \int_0^{r_1} \int_{\R^{n-1}} e^{-2\tau \Im \psi} (\eta  \:\overline{\mathcal{T} \Tilde{\eta}}+ \Tilde{\eta}\mathcal{T}\eta) w_0 \, ( a^{(1)}_{K}-a^{(2)}_K) \sqrt{|g|}\, \df r\,\df y\,\df t =\mathcal{O}(|y|^{\infty}).\] 
	This further entails, up to $\mathcal{O}(|y|^{\infty})$
	\begin{align*}
		I= 4\,\I \int_0^T  \int_{r_2}^{\ell}\int_{\R^{n-1}} e^{-2\tau \Im \psi} (\eta  \:\overline{\mathcal{T} \Tilde{\eta}}+ \Tilde{\eta}\mathcal{T}\eta) w_0 \, ( a^{(1)}_{K}-a^{(2)}_K) \sqrt{|g|}\, \df r\,\df y\,\df t.
	\end{align*}
	Observe that, in the Fermi coordinates $(r,y)$  we have the following expression for $ \psi$:
	\[\psi(r,y)= r +\frac{1}{2} H(r)y\cdot y+ \mathcal{O} (|y|^3)  \implies \mathrm{Im}(\psi) = \mathrm{Im}(\frac{1}{2} H(r)y\cdot y+ \mathcal{O} (|y|^3) ).\] 
	We next perform the change of variable $ y\rightarrow y/\sqrt{\tau}$ and obtain 
	\begin{align*}
		I=  &\frac{4\I}{\tau^{(n-1)/2}} \int_0^T  \int_{r_2}^{\ell}\int_{\R^{n-1}} e^{-\Im H(r) y\cdot y+\tau^{-1/2} \mathcal{O}(|y|^3)} (\eta  \:\overline{\mathcal{T} \Tilde{\eta}}+ \Tilde{\eta}\mathcal{T}\eta)(r,y/\sqrt{\tau})\\& \qquad\qquad\qquad\qquad\times w_0(t,r,y/\sqrt{\tau}) \, ( a^{(1)}_{K}-a^{(2)}_K)(r,y/\sqrt{\tau}) \sqrt{|g|}(r,y/\sqrt{\tau})\, \df y\, \df r\,\df t\\
		= &\frac{4\I}{\tau^{(n-1)/2} } \int_0^T  \int_{r_2}^{\ell} \frac{1}{\sqrt{\det\Im(H(r))}}\,\int_{\R^{n-1}} e^{- |y|^2-\tau^{-1/2} \mathcal{O}(|y|^3)} (\eta  \:\overline{\mathcal{T} \Tilde{\eta}}+ \Tilde{\eta}\mathcal{T}\eta)(r,(\cdot)/\sqrt{\tau})\\& \qquad\qquad\qquad\qquad\times w_0(t,r,\cdot/\sqrt{\tau}) \, ( a^{(1)}_{K}-a^{(2)}_K)(r,\cdot/\sqrt{\tau}) \sqrt{|g|}(r,\cdot/\sqrt{\tau})\, \df y\, \df r\,\df t.
	\end{align*}
	We multiply \eqref{integral_identity} by $\tau^{K-3}$ and  let $ \tau \rightarrow \infty$ to conclude
	\begin{align*}
		0= &4\I\int_0^T  \int_{r_2}^{\ell} \frac{1}{\sqrt{\det\Im(H(r))}}\,\int_{\R^{n-1}} e^{- |y|^2} (\eta  \:\overline{\mathcal{T} \Tilde{\eta}}+ \Tilde{\eta}\mathcal{T}\eta)(r,0)\\& \qquad\times w_0(t,r,0) \, ( a^{(1)}_{K}-a^{(2)}_K)(r,0)  \df y\, \df r\,\df t.
	\end{align*}
	Next choosing 
	$w_0(t,r,0)=\phi_1(t)\, \phi_2(r)$ for some $C^{\infty}_c$ functions in their respective variables, from above we deduce that,
	\begin{align*}
		0= 4\I \int_{r_2}^\ell \frac{1}{\sqrt{\det(\Im H(r))}} \, \phi_2(r)(\eta  \:\overline{\mathcal{T} \Tilde{\eta}}+ \Tilde{\eta}\mathcal{T}\eta)(r,0)\, ( a^{(1)}_{K}-a^{(2)}_K)(r,0) \df r.
	\end{align*}
	Further choosing $\eta=1$ in the support of $\Tilde{\eta}$ we obtain
	\begin{align*}
		0=& 4\I \int_{r_2}^\ell \frac{1}{\sqrt{\det(\Im H(r))}} \, \phi_2(r)\overline{\mathcal{T} \Tilde{\eta}}(r,0) \, ( a^{(1)}_{K}-a^{(2)}_K)(r,0) \df r\\
		=&4\I \int_{r_2}^\ell \frac{1}{\sqrt{\det(\Im H(r))}} \, \phi_2(r)  \PD_r\Tilde{\eta}\, ( a^{(1)}_{K}-a^{(2)}_K)(r,0) \df r.
	\end{align*}
	Let $r_3 \in (r_2, \ell)$. 
	We choose $\Tilde{\eta} \in C_c^{\infty}$ such that $\Tilde{\eta}$ converges to the indicator function of $[r_1, r_3]$ near $\gamma$. Then $ \PD_r \Tilde{\eta}$ will converge to $ \delta_{r_1}-\delta_{r_3}$, and
	\begin{align*}
		\frac{\phi_2(r_3)}{{\sqrt{\det(\Im H(r_3))}}}\,(a^{(1)}_{K}-a^{(2)}_K)(r_3,0)=0.
	\end{align*} 
	Since $ \phi_2$ solves some transport equation with non-zero initial condition on $\gamma$, this implies $\phi_2(r_3)\neq 0$, and  $ \Im H(r)$ is positive definite entails that $ \det(\Im H)(r_3)\neq 0$. This along with preceding equation implies $ (a^{(1)}_{K}-a^{(2)}_K)(r_3,0)=0$. 
	We have shown \eqref{induction_argument} in the case $k = K$ and $\alpha = 0$.
     \subsubsection*{\textbf{Step 3}}  We show that $(a^{(1)}_{K,\alpha}-a^{(2)}_{K,\alpha})(r,0)=0$  for any multi-index $ \alpha$ with $ |\alpha|=p$, where $a^{(j)}_{K}(r,y) = \sum_{\alpha} a^{(j)}_{K,\alpha}(r) y^\alpha$. The proof is by induction in $|\alpha| = p$, and suppose that the claim holds for $p$. Consider a multi-index $\alpha$ such that $|\alpha| = p + 1$. Now, we consider the Gaussian beam solution of 
	\begin{align}
		\begin{cases}
			(\I \PD_t +\Delta^2_g +\Delta_{h_{2}} + q_2)w=0\quad &\mbox{in} \quad (0,T) \times M\\
			w=0 \quad & \mbox{at} \quad t=T
		\end{cases}
	\end{align}
	that coincides with the approximate Gaussian beam
	\begin{align}
		e^{\I (\tau \psi(r,y)+ \tau^4 t)} \sum_{k=0}^N\, \frac{w_k (t,r,y)y^{\beta}}{\tau^k}, \ \mbox{where} \ |\beta| = p + 1
	\end{align}
	up to a small error in the same sense.
	We argue as in \emph{\textbf{Step~2}} to obtain 
	\begin{align}
		I &= 4\,\I \int_0^T  \int_{r_2}^{\ell}\int_{\R^{n-1}} e^{-2\tau \Im \psi} (\eta  \:\overline{\mathcal{T} \Tilde{\eta}}+ \Tilde{\eta}\mathcal{T}\eta) w_{0} y^\beta (a^{(1)}_{K,\alpha}-a^{(2)}_{K,\alpha})(r)\, y^{\alpha} \sqrt{|g|}\, \df r\,\df y\,\df t \\&\quad + \mathcal{O}(|y|^{2p+3}). 
	\end{align}
	We now proceed exactly as above  to conclude
	\begin{align}
		I= &\frac{4\I }{\tau^{(n-1)/2} }\tau^{-(2p+2)} \int_0^T  \int_{r_2}^{\ell} \frac{1}{\sqrt{\det\Im(H(r))}}\,\int_{\R^{n-1}} e^{- |y|^2-\tau^{-1/2} \mathcal{O}(|y|^3)} (\eta  \:\overline{\mathcal{T} \Tilde{\eta}}+ \Tilde{\eta}\mathcal{T}\eta)(r,(\cdot)/\sqrt{\tau})\\& \qquad\times w_{0}(t,r)\, y^\beta  \, ( a^{(1)}_{K,\alpha}-a^{(2)}_{K,\alpha})(r)\,y^{\alpha}\,\sqrt{|g|}(r,\cdot/\sqrt{\tau})\, \df y\, \df r\,\df t + \mathcal{O}(\tau^{-(2p+3)/2}).
	\end{align}
	Multiplying above by $\tau^{K+2p-1}$
	and letting $\tau \rightarrow \infty$ we conclude that
	\begin{align}
		0= 4\I \int_0^T  \int_{r_2}^{\ell} \frac{1}{\sqrt{\det\Im(H(r))}}\,&\int_{\R^{n-1}} e^{- |y|^2} y^\beta y^{\alpha}\, ( a^{(1)}_{K,\alpha}-a^{(2)}_{K,\alpha})(r)\, \\& \times(\eta  \:\overline{\mathcal{T} \Tilde{\eta}}+ \Tilde{\eta}\mathcal{T}\eta)(r,(\cdot)/\sqrt{\tau})\, w_{0}(t,r) \, \df y\, \df r\,\df t.
	\end{align}
	Further choosing $w_{0}(t,r,0)=\phi_1(t)\, \phi_2(r)$ for some $C^{\infty}_c$ functions in their respective variables, from above we deduce that,
	\begin{align}
		0=\int_{\R^{n-1}} e^{- |y|^2} y^{\beta} y^{\alpha}\,\df y\times
		4\I \int_{r_2}^\ell \frac{1}{\sqrt{\det(\Im H(r))}} \,  ( a^{(1)}_{K,\alpha}-a^{(2)}_{K,\alpha})(r) \phi_2(r)(\eta  \:\overline{\mathcal{T} \Tilde{\eta}}+ \Tilde{\eta}\mathcal{T}\eta)(r,0)\,  \df r.    
	\end{align}
	This, along with Lemma \cite{LOSST_NLS_2025} implies 
	\begin{align}
		0= \int_{r_2}^\ell \frac{1}{\sqrt{\det(\Im H(r))}} \,  ( a^{(1)}_{K,\alpha}-a^{(2)}_{K,\alpha})(r) \phi_2(r)(\eta  \:\overline{\mathcal{T} \Tilde{\eta}}+ \Tilde{\eta}\mathcal{T}\eta)(r,0)\,  \df r.    
	\end{align}
	Next, proceed similarly as in \textbf{Step 2}, one can conclude that 
	\begin{align}
		( a^{(1)}_{K,\alpha}-a^{(2)}_{K,\alpha})(r)=0 \quad \mbox{for } \,\,\abs{\alpha}=p+1.  
	\end{align}
	This completes the induction step as well as the proof of Lemma \ref{lm:determining_amplitudes}.
	\end{proof}
    \vspace{-4mm}
\section{Proof of main result: Geometric case}\label{sec:main_result_thm_1.2}
\noindent This section is devoted to proving the main result of this article in the geometric setting. The proof is divided into several steps.
\begin{proof}[Proof of Theorem \ref{th:Manifold}]\tblue{As in the Euclidean case, we again divide the proof into several steps.}
\vspace{-3mm}
\subsubsection*{\textbf{Step 1 (Set-up)}}  
\tblue{Recall that $M$ is an admissible manifold (Definition \ref{def_admissible}). We fix a point $p\in M\setminus\Gamma$ and a constant $\lambda>0$, setting $\lambda'^4:=1-\lambda^4$. In accordance with Definition \ref{def_admissible}, let $\xi_0,\xi_1,\xi_2$ be chosen such that $\xi_0:=\lambda'\xi_1+\lambda\xi_2$. For each $j=0,1,2$, we establish Fermi coordinates along the curve $\gamma_{p,\xi_j}$, centered at $p$, that is, $p=(0,0\cdots,0)$ in these coordinates. Let $U_{j,\tau}$ be an approximate Gaussian beam constructed along $\gamma_{p,\xi_j}$ in these coordinates. The corresponding exact solution to the fourth-order Schrödinger equation is then given by $u_{j,\tau}=U_{j,\tau}+R_{j,\tau}$, where the remainder term $R_{j,\tau}$ is constructed according to \eqref{eq_rmdr}.}

\noindent For $j =0,1,2$, we choose a sufficiently large natural number $N\in\N$ to define the amplitude functions $a_k^{(j)}$ and $b_k^{(j)}$ for the pairs of coefficients $(h_{1},q_{1})$ and $(h_{2},q_{2})$, respectively as follows. For $\tau>0$, we can define the approximate geometric optics solutions $U^{(1)}_{j,\tau}(t,r,y)$ and $U^{(2)}_{j,\tau}(t,r,y)$ for the pairs of coefficients $(h_{1},q_{1})$ and $(h_{2},q_{2})$, respectively, in the following form:
\begin{align*}
U^{(1)}_{j,\tau}(t,r,y)=e^{i(\tau\psi^{(j)}(r,y)+\tau^4 t)}\Bigg(\sum_{k=0}^N\tau^{-k}a_k^{(j)}(t,r,y)\Bigg),\\
U^{(2)}_{j,\tau}(t,r,y)=e^{i(\tau\psi^{(j)}(r,y)+\tau^4 t)}\Bigg(\sum_{k=0}^N\tau^{-k}b_k^{(j)}(t,r,y)\Bigg).
\end{align*} 
Further, for each $U^{(\ell)}_{j,\tau}$, we can define the corresponding exact solution $u^{(\ell)}_{j,\tau}=U^{(\ell)}_{j, \tau}+R^{(\ell)}_{j,\tau}$, where $R^{(\ell)}_{j,\tau}$ satisfies \eqref{eq_rmdr} for $\ell=1,2$. We again let $\eta\in C^\infty_0(M)$ be such that $\eta=1$ in $M\setminus\Gamma$, and consider the sources $f_{j}^\tau=\left(\I \PD_t +\Delta^2_g + \Delta_{h_{\ell}} + q_{\ell}\right)(\eta u^{(\ell)}_{j,\tau})$. Then, for $j=1,2$ the function $\eta u^{(\ell)}_{j,\tau}$ solves
\begin{align}
\left\{\begin{array}{rll}
\I \PD_t \mathcal{U}^{(\ell)}_{j,\tau}  +\Delta^2_g \mathcal{U}^{(\ell)}_{j,\tau}+ \Delta_{h_{\ell}}\mathcal{U}^{(\ell)}_{j,\tau} + q_{\ell}\mathcal{U}^{(\ell)}_{j,\tau} = f_{j}^\tau & \text{in $(0,T) \times M$,}\\
\mathcal{U}^{(\ell)}_{j,\tau}|_{x \in \partial M}={\p_{\nu_g} \mathcal{U}^{(\ell)}_{j,\tau}}|_{x \in \partial M} = 0, & \text{on $(0,T) \times \partial M$,}\\
\mathcal{U}^{(\ell)}_{j,\tau}|_{t = 0} = 0 & \text{in $M$}.
\end{array}\right.
\end{align} 
For $j=0$, the function $\tilde{\eta} u^{(\ell)}_{0,\tau}$ solves
\begin{align}
\left\{\begin{array}{rll}
\I \PD_t \mathcal{U}^{(\ell)}_{0,\tau}  +\Delta^2_g \mathcal{U}^{(\ell)}_{0,\tau}+ \Delta_{h_{\ell}}\mathcal{U}^{(\ell)}_{0,\tau} + q_{\ell}\mathcal{U}^{(\ell)}_{0,\tau} = f_0^\tau & \text{in $(0,T) \times M$,}\\
\mathcal{U}^{(\ell)}_{0,\tau}|_{x \in \partial M}={\p_{\nu_g}   \mathcal{U}^{(\ell)}_{0,\tau}}|_{x \in \partial M} = 0, & \text{on $(0,T) \times \partial M$,}\\
\mathcal{U}^{(\ell)}_{0,\tau}|_{t = T} = 0, & \text{in $M$}.
\end{array}\right.
\end{align}
As in the Euclidean case, we choose $\kappa$ large enough that $H^{4\kappa}_{00}$ is a Banach algebra. It can then be shown that the source $f_{j}^\tau\in H^{4\kappa}_{00}$ is determined by the source-to-solution map $L_{h,q,\beta}$, up to any error $\mathcal{O}(\tau^{-K})$ in the $H^{4\kappa}$-norm. 
\subsubsection*{\textbf{Step 2}} \tblue{Next, we apply the second-order linearization technique. Choosing small parameters $\varepsilon_1,\varepsilon_2>0$, we set $\varepsilon=(\varepsilon_1,\varepsilon_2)$ and introduce the source term $f^\tau=\varepsilon_1f_{1}^{\lambda'\tau}+\varepsilon_2f_{2}^{\lambda\tau}$. Provided $\varepsilon$ is sufficiently small, we have $f^\tau \in \mathcal{H}$, leading to the observation that}
$$ -\frac{1}{2}\p_{\varepsilon_1}\p_{\varepsilon_2}L_{h_\ell,q_\ell,\beta_\ell}f^\tau|_{\varepsilon=0}=w_\ell|_{(0,T)\times\Gamma},$$
where $w_\ell$ solves 
\begin{align}
\left\{\begin{array}{rll}
\I\p_tw_\ell+\Delta^2_g w_\ell+\Delta_{h_{\ell}} w_\ell+ q_\ell w_\ell &= \beta_\ell\,\mathcal{U}^{(\ell)}_{1,\lambda'\tau}\,\mathcal{U}^{(\ell)}_{2,\lambda\tau}\\
w_\ell|_{x \in \p M} &= {\p_{\nu_g} w_\ell}|_{x \in \p M} = 0\\
w_\ell|_{t=0} &= 0.
\end{array}\right.
\end{align}
This implies that
\begin{align*}\int_{(0,T)\times M}w_\ell\ \overline{(\I\partial_t+\Delta_g^2+\Delta_{h_{\ell}}+q_{\ell})\mathcal{U}^{(\ell)}_{0,\tau}}\ \df V_g\, \df t &=\int_{(0,T)\times\Gamma} w_\ell \overline{f^\tau_0}\ \df V_g\, \df t\\ &= -\frac{1}{2}\int_{(0,T)\times\Gamma}\partial_{\varepsilon_1}\partial_{\varepsilon_2}L_{h_{\ell},q_{\ell},\beta_{\ell}}f^\tau|_{\varepsilon=0}\ \overline{f^\tau_0}\ \df V_g\,\df t.
\end{align*}
Using integration by parts, this further entails
\begin{align}\int_{(0,T)\times M}(\I\partial_t+\Delta_g^2+\Delta_{h_\ell}+q_{\ell})w_\ell \ \overline{\mathcal{U}^{(\ell)}_{0,\tau}}\ \df V_g\, \df t &=\int_{(0,T)\times M} \beta_{\ell}\ \mathcal{U}^{(\ell)}_{1,\lambda'\tau}\  \mathcal{U}^{(\ell)}_{2,\lambda\tau} \ \overline{\mathcal{U}^{(\ell)}_{0,\tau}}\ \df V_g\, \df t \nonumber\\ \label{geome_before} &= -\frac{1}{2}\int_{(0,T)\times\Gamma}\partial_{\varepsilon_1}\partial_{\varepsilon_2}L_{h_{\ell},q_{\ell},\beta_{\ell}}f^\tau|_{\varepsilon=0}\ \overline{f^\tau_0}\ \df V_g\,\df t.
\end{align}
Next, subtracting the above equations corresponding to $\ell = 1$ and $\ell = 2$, and using the identity
$L_{h_{1},q_{1},\beta_{1}} = L_{h_{2},q_{2},\beta_{2}}$, we obtain that 
\begin{align}\label{geome_after}\int_{(0,T)\times M} \left(\beta_{1}\ \mathcal{U}^{(1)}_{1,\lambda'\tau}\  \mathcal{U}^{(1)}_{2,\lambda\tau} \ \overline{\mathcal{U}^{(1)}_{0,\tau}}-\beta_{2}\ \mathcal{U}^{(2)}_{1,\lambda'\tau}\  \mathcal{U}^{(2)}_{2,\lambda\tau} \ \overline{\mathcal{U}^{(2)}_{0,\tau}}\right)\ \df V_g\, \df t = 0.
\end{align}
We now like to approximate $\mathcal{U}^{(\ell)}_{j,\tau}$ by $\eta U^{(\ell)}_{j,\tau}$ in the integral \eqref{geome_after} for $j=1,2$ and $\mathcal{U}^{(\ell)}_{0,\tau}$ by $\tilde{\eta} U^{(\ell)}_{0,\tau}$. Therefore, choose $\kappa$ sufficiently large so that  $H^{4\kappa}_{00}$ is a Banach algebra, and let $N\geq4\kappa+4$. Then, up to an error of order $\mathcal{O}(\tau^{-4})$, the integral \eqref{geome_after} becomes equivalent to the integral
\[\int_{(0,T)\times M}\eta^2\tilde{\eta} \left(\beta_{1}\ U^{(1)}_{1,\lambda'\tau}\  U^{(1)}_{2,\lambda\tau} \ \overline{U^{(1)}_{0,\tau}}-\beta_{2}\ U^{(2)}_{1,\lambda'\tau}\  U^{(2)}_{2,\lambda\tau} \ \overline{U^{(2)}_{0,\tau}}\right)\, \df V_g\, \df t =0.\]
Further choose that $\eta=\tilde{\eta}=1$ in $\supp\big(U_{1,\lambda'\tau}\  U_{2,\lambda\tau} \ \overline{U_{0,\tau}}\big)$ and since $\phi(t)$ appearing in the definition of $a_0^{(j)}$ and $b_0^{(j)}$ is arbitrary, it follows that the above integral is determined for all $t\in(0, T)$
\begin{align}\label{integral_identity_1}
\int_{M}\left(\beta_{1}\ U^{(1)}_{1,\lambda'\tau}\  U^{(1)}_{2,\lambda\tau} \ \overline{U^{(1)}_{0,\tau}}-\beta_{2}\ U^{(2)}_{1,\lambda'\tau}\  U^{(2)}_{2,\lambda\tau} \ \overline{U^{(2)}_{0,\tau}}\right)\, \df V_g\, =0.
\end{align}
We start with analyzing the integral identity \eqref{integral_identity_1}. Since $\gamma_{p,\xi_0}\cap\gamma_{p,\xi_1}\cap\gamma_{p,\xi_2}=\{p\}$, we observe that the product\[\overline{U^{(1)}_{0,\tau}}U^{(1)}_{1,\lambda'\tau}U^{(1)}_{2,\lambda\tau}=e^{i\tau\Psi}a^{(0)}_\tau a^{(1)}_{\lambda'\tau}a^{(2)}_{\lambda\tau},  \quad \overline{U^{(2)}_{0,\tau}}U^{(2)}_{1,\lambda'\tau}U^{(2)}_{2,\lambda\tau}=e^{i\tau\Psi}b^{(0)}_\tau b^{(1)}_{\lambda'\tau}b^{(2)}_{\lambda\tau}\] are supported in a neighbourhood of $p$ and the function $\Psi$ is given by $ \Psi= \overline{\psi^{(0)}}+\lambda'\psi^{(1)}+\lambda\psi^{(2)}.$ 
\subsubsection*{\textbf{Step 3.}} We claim that the function $\Psi=-\overline{\psi^{(0)}}+\lambda'\psi^{(1)}+\lambda\psi^{(2)} $ satisfies the following two conditions: $\nabla_g\Psi(p)=0$, and  $\det(\nabla_g^2 \Psi)(p)\neq 0$. \tblue{The terms $\psi^{(j)}$ are approximate Gaussian beam solutions along the geodesics $\gamma_{p,\xi_j}$ ($j=0,1,2$), which intersect exclusively at $p \in M$.}
Note that, $\nabla_g\psi^{(j)}(p)=\xi_j$ for $j=0,1,2$. This implies \[\nabla_g\Psi(p)=-\xi_0+\lambda'\xi_1+\lambda\xi_2,\] which vanishes by the definition of admissible manifolds, see Definition \ref{def_admissible}.  

\noindent \tblue{To prove the last claim, it is sufficient to demonstrate that $D^2\Im\Psi(X,X)>0$ for all non-zero $X\in T_pM$. First, observe that}
\[\Im\Psi=\Im\psi^{(0)}+\lambda'\Im\psi^{(1)}+\lambda\Im\psi^{(2)},\]
\tblue{which implies the non-negativity condition $D^2\Im\Psi(X,X)\geq0$. To see this, observe that in Fermi coordinates along $\gamma_{p,\xi_j}$, we have}
\[D^2\Im\psi^{(j)}|_{\gamma_{p,\xi_j}}=\begin{pmatrix}0&0\\0&H_j\end{pmatrix}.\] 
\tblue{Here, $H_j$ denotes the appropriate matrix solution to \eqref{!riccati} along $\gamma_{p,\xi_j}$. Applying \eqref{!psi01} together with the lower bound $\Im(\psi)(r,y)\geq c|y|^2$, we see that for $j=0,1,2$:}
\[D^2\Im\psi^{(j)}(X,X) \geq0, \,\, \forall X\in T_pM \qum{and}  \quad D^2\Im\psi^{(j)}(X,X)>0, \,\, 
\forall X\in T_pM\setminus\vspan(\xi^{(j)}).\]
The claim follows from the fact that $\xi_1$ and $\xi_2$ are linearly independent. 
 \subsubsection*{\textbf{Step 4}} In this step, we show that $\beta_1=\beta_2$ in $M.$ From  \textbf{\emph{Step 3}}, the phase function $\Psi$ has  a non-degenerate critical point at $p$, we now apply method of stationary phase from \cite{Hormander_1}*{Theorem 7.7.5} to analyze the integral identity \eqref{integral_identity_1}. 
We can expand the amplitudes $a^{(j)}_\tau$ in terms of functions $a^{(j)}_k$ and $b^{(j)}_\tau$ in terms of functions $b^{(j)}_k$ as in \eqref{eq transport}. 
Thus we recover
\begin{align}
0 &= \int_{M}e^{i\tau\Psi}\left(\beta_{1}\ a^{(1)}_{\lambda'\tau}\  a^{(2)}_{\lambda\tau} \ {a^{(0)}_{\tau}}-\beta_{2}\ b^{(1)}_{\lambda'\tau}\  b^{(2)}_{\lambda\tau} \ {b^{(0)}_{\tau}}\right)\, \df V_g\,  \nonumber\\ &= \frac{e^{i\tau\Psi(p)}C_0}{\tau^{n/2}} \Bigg[\left(\beta_1(p)a_0^{(0)}a_0^{(1)}a_0^{(2)}(p)-\beta_2(p)b_0^{(0)}b_0^{(1)}b_0^{(2)}(p)\right)  \nonumber\\ & \quad \  \left. + \frac{1}{\tau}\left(\beta_1(p)\left( \sum_{i+j+k=1} \frac{1}{\lambda'^j}\frac{1}{\lambda^k}a_i^{(0)}a_j^{(1)}a_k^{(2)}(p)\right)   -\beta_2(p)\left(\sum_{i+j+k=1} \frac{1}{\lambda'^j}\frac{1}{\lambda^k}b_i^{(0)}b_j^{(1)}b_k^{(2)}(p) \right)\right)\right.  \nonumber\\ &  \left. \quad \ + \frac{1}{\tau^2}\left(\beta_1(p)\Bigg(\sum_{i+j+k=2} \frac{1}{\lambda'^j}\frac{1}{\lambda^k}a_i^{(0)}a_j^{(1)}a_k^{(2)}(p)\Bigg) -\beta_2(p)\Bigg(\sum_{i+j+k=2}\frac{1}{\lambda'^j}\frac{1}{\lambda^k}b_i^{(0)}b_j^{(1)}b_k^{(2)}(p)\Bigg)\right)\right.   \nonumber\\ & \quad \  \left. + \frac{1}{\tau^3}\left(\beta_1(p)\Bigg(\sum_{i+j+k=3}\frac{1}{\lambda'^j}\frac{1}{\lambda^k}a_i^{(0)}a_j^{(1)}a_k^{(2)}(p)\Bigg)  -\beta_2(p)\Bigg(\sum_{i+j+k=3}\frac{1}{\lambda'^j}\frac{1}{\lambda^k}b_i^{(0)}b_j^{(1)}b_k^{(2)}(p)\Bigg)\right) + \mathcal{O}(\tau^{-4})\right] \nonumber\\ \label{Integral_term}& = I_0+\tau^{-1}I_1+\tau^{-2}I_2+\tau^{-3}I_3+ \mathcal{O}(\tau^{-4}) 
\end{align}
\tblue{where $C_0$ is a constant independent of $(\beta, h, q)$. Using the relation $a_0^{(j)}(p)=b_0^{(j)}(p)$, we recover $\beta_1(p)=\beta_2(p)$ from the first term, $I_0$. Since the point $p\in M\setminus \Gamma$ was chosen arbitrarily, we uniquely recover the function $\beta(t,x)$ everywhere.}
\subsubsection*{\textbf{Step 5}}\tblue{In this step, our goal is to prove that $h_1 = h_2$ in $M$. The proof relies on examining $I_1$ and $I_2$: the former yields the equality $h_1^{11} = h_2^{11}$, while the latter allows us to recover $h_1^{1j} - h_2^{1j}$ for all $1 \le j \le n$. Finally, to reconstruct the other components of the metric $h$, we perform the same analysis on a small variation $\gamma_{\epsilon}$ of a geodesic $\gamma$. In a suitable system of Fermi coordinates, this variation is chosen to satisfy $\gamma_{\epsilon}(0) = \gamma(0) = p$ and $\displaystyle \gamma_{\epsilon}'(0) = \frac{1}{\sqrt{1+\epsilon^2}}(e_1 + \epsilon e_j)$ for $j \ge 2$.}
\\
\noindent We denote $\beta_1(p)=\beta_2(p)=\beta(p)$. Since $\beta(p)$ is known and non-zero for almost all $p$, this together with $a_0^{(j)}(p) = b_0^{(j)}(p)$ implies
$$(a_1^{(0)}- b_1^{(0)})(p) +\frac{1}{\lambda'}(a_1^{(1)}-b_1^{(1)})(p)+\frac{1}{\lambda}(a_1^{(2)}- b_1^{(2)})(p) =0.$$
Since $\lambda'^4=1-\lambda^4$, we can obtain the quantity
$$ \lim_{\lambda\to 0}\lambda\left((a_1^{(0)}- b_1^{(0)})(p) +\frac{1}{\lambda'}(a_1^{(1)}-b_1^{(1)})(p)+\frac{1}{\lambda}(a_1^{(2)}- b_1^{(2)})(p)\right) = (a_1^{(2)}- b_1^{(2)})(p) = 0.$$
We recall from \eqref{!split_1} that $a_1^{(2)}$ and $b_1^{(2)}$ are of the form $a_1^{(2)} = c_1^{(2)}+d_1^{(2)}$ and $b_1^{(2)} = e_1^{(2)}+f_1^{(2)}$, where it holds that
\begin{align*}
&c_{1,0}^{(2)}(r)=e^{-\frac{1}{2}\int_{r_0}^r\tr(H)(s)\df s}\int_{r_0}^r\frac{\I}{4}e^{\frac{1}{2}\int_{r_0}^s\tr(H)\df \tilde{s}}\left(2\frac{d}{dr}(\Delta_g\psi)a_{0,0}^{(2)}+4\tr(H)\frac{d}{dr}a_{0,0}^{(2)}+(\tr(H))^2a_{0,0}^{(2)} \right.\\ & \left. \qquad \qquad \qquad \qquad \qquad \qquad+4\frac{d^2}{dr^2}a_{0,0}^{(2)} +\Delta_g(|\df \psi|^2_g)a_{0,0}^{(2)}+2\Delta_g a_0^{(2)}\right)(s) \df s\\
&d_{1,0}^{(2)}(r)=e^{-\frac{1}{2}\int_{r_0}^r\tr(H)(s)\df s}\int_{r_0}^r\frac{\I}{4}e^{\frac{1}{2}\int_{r_0}^s\tr(H)\df \tilde{s}}h_1^{11}a_{0,0}^{(2)}(s)\df s.\\
&e_{1,0}^{(2)}(r)=e^{-\frac{1}{2}\int_{r_0}^r\tr(H)(s)\df s}\int_{r_0}^r\frac{\I}{4}e^{\frac{1}{2}\int_{r_0}^s\tr(H)\df \tilde{s}}\left(2\frac{d}{dr}(\Delta_g\psi)b_{0,0}^{(2)}+4\tr(H)\frac{d}{dr}b_{0,0}^{(2)}+(\tr(H))^2b_{0,0}^{(2)} \right.\\ & \left. \qquad \qquad \qquad \qquad \qquad \qquad+4\frac{d^2}{dr^2}b_{0,0}^{(2)} +\Delta_g(|\df \psi|^2_g)b_{0,0}^{(2)}+2\Delta_g b_0^{(2)}\right)(s) \df s\\
&f_{1,0}^{(2)}(r)=e^{-\frac{1}{2}\int_{r_0}^r\tr(H)(s)\df s}\int_{r_0}^r\frac{\I}{4}e^{\frac{1}{2}\int_{r_0}^s\tr(H)\df \tilde{s}}h_2^{11}b_{0,0}^{(2)}(s)\df s.
\end{align*}
\tblue{Notably, the terms $a_0^{(2)}$ and $b_0^{(2)}$ do not depend on the coefficients $h$ and $q$, which means $c_{1,0}^{(2)}$ and $e_{1,0}^{(2)}$ are also independent of them. The equality $a_0^{(2)}=b_0^{(2)}$ directly yields $c_1^{(2)}= e_1^{(2)}$, allowing us to deduce the quantity}
$$(d_1^{(2)}- f_1^{(2)})(p) =0.$$
Next, using the expressions of $d_1^{(2)}$ and $f_1^{(2)}$ together with expression of $a_{0,0}^{(2)}=b_{0,0}^{(2)}$ from \eqref{!amp00}, we evaluate $d_1^{(2)}-f_1^{(2)}$ at $p$. Furthermore,  letting $p= \gamma(r_1)$ we conclude $ \int_{r_0}^{r_1}(h_1^{11}-h_2^{11})(s)\, \df s =0,$
in the Fermi coordinates along $\gamma_{p,\xi_2}$.  Next, differentiating $\int_{r_0}^{r_1}(h_1^{11}-h_2^{11})(s)\, \df s=0$  with respect to the upper limit of integration, we conclude
\begin{align}
h_1^{11}(p)=h_2^{11}(p).
\end{align}
Next, we consider $I_2$
$$\Bigg(\sum_{i+j+k=2} \frac{1}{\lambda'^j}\frac{1}{\lambda^k}a_i^{(0)}a_j^{(1)}a_k^{(2)}(p)\Bigg) -\Bigg(\sum_{i+j+k=2}\frac{1}{\lambda'^j}\frac{1}{\lambda^k}b_i^{(0)}b_j^{(1)}b_k^{(2)}(p)\Bigg)=0.$$
Since $\lambda'^4=1-\lambda^4$, we can obtain the quantity
$$\lim_{\lambda\to 0}\lambda^2\Bigg(\sum_{i+j+k=2} \frac{1}{\lambda'^j}\frac{1}{\lambda^k}\left(a_i^{(0)}a_j^{(1)}a_k^{(2)}(p) -b_i^{(0)}b_j^{(1)}b_k^{(2)}(p)\right)\Bigg) = (a_2^{(2)}- b_2^{(2)})(p) = 0.$$
We recall from \eqref{!split_2} that $a_2^{(2)}$ and $b_2^{(2)}$ are of the form $a_2^{(2)} = c_2^{(2)}+d_2^{(2)}$ and $b_2^{(2)} = e_2^{(2)}+f_2^{(2)}$, where it hold that
\begin{align*}
c^{(2)}_{2,0}(r)&=e^{-\frac{1}{2}\int_{r_0}^r\tr(H)(s)\df s}\int_{r_0}^r\frac{\I}{4}e^{\frac{1}{2}\int_{r_0}^s\tr(H)\df \tilde{s}}\left(2\frac{d}{dr}(\Delta_g\psi)a^{(2)}_{1,0}+4\tr(H)\frac{d}{dr}a^{(2)}_{1,0}+(\tr(H))^2a^{(2)}_{1,0} \right.\\ & \left. \qquad +4\frac{d^2}{dr^2}a^{(2)}_{1,0} +\Delta_g(|\df\psi|^2_g)a^{(2)}_{1,0}+2\Delta_g a^{(2)}_1+h_1^{11}a^{(2)}_{1,0}-2\I \lr \df(\Delta_g\psi), \df a^{(2)}_0 \rn_g \right.\\ & \left. \qquad - \I\Delta_g^2\psi a_{0,0} -2\I \tr(H) \Delta_g a^{(2)}_0 -2\I \Delta_g(\lr \df\psi, \df a^{(2)}_0 \rn)-2\I \frac{d}{dr}(\Delta_ga^{(2)}_0) \right)(s) \df s\\
d^{(2)}_{2,0}(r) &=e^{-\frac{1}{2}\int_{r_0}^r\tr(H)(s)\df s}\int_{r_0}^r\frac{1}{4}e^{\frac{1}{2}\int_{r_0}^s\tr(H)\df \tilde{s}}\left((\Delta_{h_1}\psi) a^{(2)}_{0,0} + 2\sum_{j=2}^n h_1^{1j} \p_ja^{(2)}_0\right)(s)\df s.\\
e^{(2)}_{2,0}(r) &=e^{-\frac{1}{2}\int_{r_0}^r\tr(H)(s)\df s}\int_{r_0}^r\frac{\I}{4}e^{\frac{1}{2}\int_{r_0}^s\tr(H)\df \tilde{s}}\left(2\frac{d}{dr}(\Delta_g\psi)b^{(2)}_{1,0}+4\tr(H)\frac{d}{dr}b^{(2)}_{1,0}+(\tr(H))^2b^{(2)}_{1,0} \right.\\ & \left. \qquad +4\frac{d^2}{dr^2}b^{(2)}_{1,0} +\Delta_g(|\df \psi|^2_g)b^{(2)}_{1,0}+2\Delta_g b^{(2)}_1+h_2^{11}b^{(2)}_{1,0}-2\I \lr \df (\Delta_g\psi), \df b^{(2)}_0 \rn_g \right.\\ & \left. \qquad- \I\Delta_g^2\psi b_{0,0} -2\I \tr(H) \Delta_g b^{(2)}_0 -2\I \Delta_g(\lr \df \psi, \df b^{(2)}_0 \rn)-2\I \frac{d}{dr}(\Delta_gb^{(2)}_0) \right)(s) \df s\\
f^{(2)}_{2,0}(r)&=e^{-\frac{1}{2}\int_{r_0}^r\tr(H)(s)\df s}\int_{r_0}^r\frac{1}{4}e^{\frac{1}{2}\int_{r_0}^s\tr(H)\df \tilde{s}}\left((\Delta_{h_2}\psi) b^{(2)}_{0,0} + 2\sum_{j=2}^n h_2^{1j} \p_jb^{(2)}_0\right)(s)\df s.
\end{align*}
In particular, $a_0^{(2)}, a_1^{(2)}, b_0^{(2)}$ and $b_1^{(2)}$ are independent of the coefficients $h^{1j}$ for $j=2,\dots,n$ and $q$, so is $c_{2,0}^{(2)}$ and $e_{2,0}^{(2)}$. Since $a_0^{(2)}=b_0^{(2)}$ and $a_1^{(2)}=b_1^{(2)}$ as $h_1^{11}=h_2^{11}$, we have $c_2^{(2)}= e_2^{(2)}$. Thus, using this to determine the quantity
$$(d_2^{(2)}- f_2^{(2)})(p) =0.$$
Next, we take $a_{0,0}(r_0)=0$ and $b_{0,0}(r_0)=0$ in condition \eqref{!amp00} to get $a_{0,0}^{(2)}=b_{0,0}^{(2)}=0$. Also, assume $a^{(2)}_{0,1}$ and $b^{(2)}_{0,1}$ as a function of only $y_j$ for $j = 2,\dots,n$ and using condition $a^{(2)}_{0,1}(r_0) = b^{(2)}_{0,1}(r_0) = 1$ in \eqref{eq_transport_v_0j}, substitute this in the expressions of $d_2^{(2)}$ and $f_2^{(2)}$ and we evaluate $d_2^{(2)}-f_2^{(2)}$ at $p$, letting $p=\gamma(r_1)$ to deduce
\begin{align}\label{h_j comp}
\frac{1}{4}\int_{r_0}^{r_1}(h_1^{1j}-h_2^{1j})(s)\df s=0.
\end{align}
in the Fermi coordinates along $\gamma_{p,\xi_2}$. We can recover $h_1^{1j}(p)-h_2^{1j}(p)=0$ for each $j = 2,\dots,n$ by differentiating \eqref{h_j comp} with respect to the upper limit of integration. Consider the small perturbations of $\dot{\gamma}(r_1)= e_1=(1,0,\dots,0)\in \mathbb{R}^{n}$, given by 
$$e'_2 = \frac{1}{\sqrt{1+\varepsilon^2}}(1,\varepsilon,0,\dots,0), \dots, e'_{n}=\frac{1}{\sqrt{1+\varepsilon^2}}(1,0,\dots,0,\varepsilon),$$
for $\varepsilon > 0$ small. By \cite[Proposition D.2]{Krupchyk_Uhlmann_magnetic}, \tblue{For sufficiently small $\varepsilon>0$, the unit-speed geodesic $\gamma_{e'_k}$ (for $k=2,\dots,n$) passing through $(p,e'_k)$ has no self-intersections and is non-tangential at the boundary. Applying the preceding discussion to $\gamma=\gamma_{e'_k}$ and evaluating the term $I_1$, we obtain}
\begin{align*}
\frac{\I}{4}\int_{r_0}^{r_1}(h_1^{11}+2\varepsilon h_1^{1k}+\varepsilon^2 h_1^{kk}-h_2^{11}-2\varepsilon h_2^{1k}-\varepsilon^2 h_2^{kk})(s)\df s=0.
\end{align*}
Since $h_1^{11}=h_2^{11}$ and $h_1^{1k}=h_2^{1k}$, we get
\begin{align}\label{h_2 comp}
\int_{r_0}^{r_1}(\varepsilon^2 h_1^{kk}-\varepsilon^2 h_2^{kk})(s)\df s=0.
\end{align}
\tblue{Differentiating \eqref{h_2 comp} with respect to its upper limit of integration allows us to recover the identity $h_1^{kk}(p)-h_2^{kk}(p)=0$. Turning our attention to the $I_2$ term and applying analogous reasoning, we find that}
\begin{align*}
\frac{1}{4}\int_{r_0}^{r_1}(h_1^{1j}+\varepsilon h_1^{kj}-h_2^{1j}-\varepsilon h_2^{kj})(s)\df s=0.
\end{align*}
Since $h_1^{1j}=h_2^{1j}$, we get 
\begin{align}\label{h_kj comp}
\frac{1}{4}\int_{r_0}^{r_1}e^{-\frac{1}{2}\int_{r_0}^s\tr(H)\df \tilde{s}}(\varepsilon h_1^{kj}-\varepsilon h_2^{kj})(s)\df s=0.
\end{align}
We can recover $h_1^{kj}(p)-h_2^{kj}(p)=0$ for each $j = 2,\dots,n$ by differentiating \eqref{h_kj comp} with respect to the upper limit of integration. Hence $h_1^{kj}(p)=h_2^{kj}(p)$ for each $k,j=2,\dots,n$. This together with $h_1^{11}(p)=h_2^{11}(p)$ and  $h_1^{1j}(p)=h_2^{1j}(p)$ for each $j=2,\dots,n$, we conclude 
$h_1(p)=h_2(p)$. Since $p\in M\setminus \Gamma$ was arbitrary, we uniquely recover the function $h(t,x)$ everywhere.
\subsubsection*{\textbf{Step 6}} We show that $q_1= q_2$ in $M$.
To achieve this, we consider $I_3$ term from \eqref{Integral_term} and utilize $\beta_1(p)=\beta_2(p)$, to deduce
\begin{align*}
\Bigg(\sum_{i+j+k=3}\frac{1}{\lambda'^j}\frac{1}{\lambda^k}a_i^{(0)}a_j^{(1)}a_k^{(2)}(p)\Bigg)- \Bigg(\sum_{i+j+k=3}\frac{1}{\lambda'^j}\frac{1}{\lambda^k}b_i^{(0)}b_j^{(1)}b_k^{(2)}(p)\Bigg)=0
\end{align*}
Since $\lambda'^4=1-\lambda^4$, we can obtain the quantity
\begin{align*}
\lim_{\lambda\to 0}\lambda^3\Bigg(\sum_{i+j+k=3}\frac{1}{\lambda'^j}\frac{1}{\lambda^k}\left(a_i^{(0)}a_j^{(1)}a_k^{(2)}(p)-b_i^{(0)}b_j^{(1)}b_k^{(2)}(p)\right)\Bigg) = (a_3^{(2)}-b_3^{(2)})(p)=0.
\end{align*}
We recall from \eqref{!split_3} that $a_3^{(2)}$ and $b_3^{(2)}$ are of the form $a_3^{(2)} = c_3^{(2)}+d_3^{(2)}$ and $b_3^{(2)} = e_3^{(2)}+f_3^{(2)}$, where it hold that
\begin{equation}\begin{aligned}
c^{(2)}_{3,0}(r)&=e^{-\frac{1}{2}\int_{r_0}^r\tr(H)(s)\df s}\int_{r_0}^r\frac{\I}{4}e^{\frac{1}{2}\int_{r_0}^s\tr(H)\df \tilde{s}}\left(2\frac{d}{dr}(\Delta_g\psi)a^{(2)}_{2,0}+4\tr(H)\frac{d}{dr}a^{(2)}_{2,0}+(\tr(H))^2a^{(2)}_{2,0} \right.\\ & \left. \qquad +4\frac{d^2}{dr^2}a^{(2)}_{2,0} +\Delta_g(|\df \psi|^2_g)a^{(2)}_{2,0}+2\Delta_g a^{(2)}_2+h_1^{11}a^{(2)}_{2,0}-2\I \lr \df (\Delta_g\psi), \df a^{(2)}_1 \rn_g - \I\Delta_g^2\psi a^{(2)}_{1,0} \right.\\ & \left. \qquad -2\I \tr(H) \Delta_g a^{(2)}_1 -2\I \Delta_g(\lr \df \psi, \df a^{(2)}_1 \rn)-2\I \frac{d}{dr}(\Delta_ga^{(2)}_1) -\I (\Delta_{h_1}\psi) a^{(2)}_{1,0} -2\I h_1^{1j} \p_ja^{(2)}_1 \right)(s) \df s\\
d^{(2)}_{3,0}(r)&= -e^{-\frac{1}{2}\int_{r_0}^r\tr(H)(s)\df s}\int_{r_0}^r\frac{\I}{4}e^{\frac{1}{2}\int_{r_0}^s\tr(H)\df \tilde{s}}  P_{h_1,q_1}a^{(2)}_{0}(s)\df s\\
e^{(2)}_{3,0}(r)&=e^{-\frac{1}{2}\int_{r_0}^r\tr(H)(s)\df s}\int_{r_0}^r\frac{\I}{4}e^{\frac{1}{2}\int_{r_0}^s\tr(H)\df \tilde{s}}\left(2\frac{d}{dr}(\Delta_g\psi)b^{(2)}_{2,0}+4\tr(H)\frac{d}{dr}b^{(2)}_{2,0}+(\tr(H))^2b^{(2)}_{2,0} +4\frac{d^2}{dr^2}b^{(2)}_{2,0}   \right.\\ & \left. \qquad +\Delta_g(|\df \psi|^2_g)b^{(2)}_{2,0}+2\Delta_g b^{(2)}_2+h_2^{11}b^{(2)}_{2,0}-2\I \lr \df (\Delta_g\psi), \df b^{(2)}_1 \rn_g - \I\Delta_g^2\psi b^{(2)}_{1,0} -2\I \tr(H) \Delta_g b^{(2)}_1 \right.\\ & \left. \qquad - 2\I \Delta_g(\lr \df \psi, \df b^{(2)}_1 \rn) -2\I \frac{d}{dr}(\Delta_gb^{(2)}_1)-\I (\Delta_{h_2}\psi) b^{(2)}_{1,0} -2\I h_2^{1j} \p_jb^{(2)}_1 \right)(s) \df s\\
f^{(2)}_{3,0}(r)&= -e^{-\frac{1}{2}\int_{r_0}^r\tr(H)(s)\df s}\int_{r_0}^r\frac{\I}{4}e^{\frac{1}{2}\int_{r_0}^s\tr(H)\df \tilde{s}} P_{h_2,q_2}b^{(2)}_{0}(s)\df s.
\end{aligned} 
\end{equation}
In particular, $a_0^{(2)}, a_1^{(2)}, a_2^{(2)}, b_0^{(2)}, b_1^{(2)}$ and $b_2^{(2)}$ are independent of the coefficient $q$, so is $c_{3,0}^{(2)}$ and $e_{3,0}^{(2)}$. Since $a_0^{(2)}=b_0^{(2)}$, $a_1^{(2)}=b_1^{(2)}$ and $a_2^{(2)}=b_2^{(2)}$ as $h_1=h_2$, we have $c_3^{(2)}= e_3^{(2)}$. Thus, using this to determine the quantity
$$(d_3^{(2)}- f_3^{(2)})(p) =0.$$
\tblue{We substitute this result into the expressions for $d_3^{(2)}$ and $f_3^{(2)}$, utilizing the identity $a_{0,0}^{(2)}=b_{0,0}^{(2)}$ from \eqref{!amp00}. Evaluating the difference $d_3^{(2)}-f_3^{(2)}$ at the point $p=\gamma(r_1)$ then yields}
\begin{align}\label{q comp}
\frac{\I}{4}  \int_{r_0}^{r_1}\left(q_1-q_2\right)(s)\df s =0
\end{align}
\tblue{in Fermi coordinates along $\gamma_{p,\xi_2}$. By differentiating \eqref{h_j comp} with respect to its upper limit of integration, we deduce that $q_1(p)-q_2(p)=0$. Since the point $p\in M\setminus \Gamma$ was chosen arbitrarily, it follows that the function $q(t,x)$ is uniquely recovered everywhere.}
\end{proof}
\section*{Conflict of interest statement} All the authors declare that they have no known competing financial interests or personal relationships that could have appeared to influence the work reported in this manuscript.
\section*{acknowledgements} 
SKS  and AP are  supported by IIT Bombay
seed grant (RD/0524-IRCCSH0-021) and ANRF Early Career Research Grant (ECRG) (RD/0125-
ANRF000-016).  
\section*{Data availability statement}
Data availability is not applicable to this article as no new data were created or analysed in this
study.

\bibliography{bibfile}

\begin{bibdiv}
\begin{biblist}

\bib{AJS_bihar_nonlinear}{article}{
      author={Agrawal, Divyansh},
      author={Jaiswal, Ravi~Shankar},
      author={Sahoo, Suman~Kumar},
       title={The linearized partial data {C}alder\'on problem for biharmonic operators},
        date={2024},
        ISSN={0362-546X,1873-5215},
     journal={Nonlinear Anal.},
      volume={244},
       pages={Paper No. 113544, 12},
         url={https://doi.org/10.1016/j.na.2024.113544},
      review={\MR{4734533}},
}

\bib{Ibtissem_dynamical_Schroedinger}{article}{
      author={Ben~A\"icha, Ibtissem},
       title={Stability estimate for an inverse problem for the {S}chr\"odinger equation in a magnetic field with time-dependent coefficient},
        date={2017},
        ISSN={0022-2488,1089-7658},
     journal={J. Math. Phys.},
      volume={58},
      number={7},
       pages={071508, 21},
         url={https://doi.org/10.1063/1.4995606},
      review={\MR{3680328}},
}

\bib{Bellassoued_Ben_IPI}{article}{
      author={Bellassoued, Mourad},
      author={Ben~Fraj, Oumaima},
       title={Stability estimates for time-dependent coefficients appearing in the magnetic {S}chr\"odinger equation from arbitrary boundary measurements},
        date={2020},
        ISSN={1930-8337,1930-8345},
     journal={Inverse Probl. Imaging},
      volume={14},
      number={5},
       pages={841\ndash 865},
         url={https://doi.org/10.3934/ipi.2020039},
      review={\MR{4128445}},
}

\bib{BG_JFAA}{article}{
      author={Bhattacharyya, Sombuddha},
      author={Ghosh, Tuhin},
       title={Inverse boundary value problem of determining up to a second order tensor appear in the lower order perturbation of a polyharmonic operator},
        date={2019},
        ISSN={1069-5869,1531-5851},
     journal={J. Fourier Anal. Appl.},
      volume={25},
      number={3},
       pages={661\ndash 683},
         url={https://doi.org/10.1007/s00041-018-9625-3},
      review={\MR{3953481}},
}

\bib{BG_MAA}{article}{
      author={Bhattacharyya, Sombuddha},
      author={Ghosh, Tuhin},
       title={An inverse problem on determining second order symmetric tensor for perturbed biharmonic operator},
        date={2022},
        ISSN={0025-5831,1432-1807},
     journal={Math. Ann.},
      volume={384},
      number={1-2},
       pages={457\ndash 489},
         url={https://doi.org/10.1007/s00208-021-02276-6},
      review={\MR{4476229}},
}

\bib{Bhattacharyya-Kumar}{article}{
      author={Bhattacharyya, Sombuddha},
      author={Kumar, Pranav},
       title={Local data inverse problem for the polyharmonic operator with anisotropic perturbations},
        date={2024},
        ISSN={0266-5611,1361-6420},
     journal={Inverse Problems},
      volume={40},
      number={5},
       pages={Paper No. 055004, 22},
      review={\MR{4723844}},
}

\bib{BansalKrishnanPattar}{article}{
      author={Bansal, Rajat},
      author={Krishnan, Venkateswaran~P.},
      author={Pattar, Rahul~Raju},
       title={Determination of lower order perturbations of a polyharmonic operator in two dimensions},
        date={2025},
     journal={Journal of Inverse and Ill-posed Problems},
      volume={33},
      number={1},
       pages={1\ndash 9},
         url={https://doi.org/10.1515/jiip-2023-0067},
}

\bib{BKS_MRT_polyharmonic}{article}{
      author={Bhattacharyya, Sombuddha},
      author={Krishnan, Venkateswaran~P.},
      author={Sahoo, Suman~K.},
       title={Momentum ray transforms and a partial data inverse problem for a polyharmonic operator},
        date={2023},
        ISSN={0036-1410,1095-7154},
     journal={SIAM J. Math. Anal.},
      volume={55},
      number={4},
       pages={4000\ndash 4038},
         url={https://doi.org/10.1137/22M1500617},
      review={\MR{4631015}},
}

\bib{BKSU_biharmonic_nonlinear}{article}{
      author={Bhattacharyya, Sombuddha},
      author={Krupchyk, Katya},
      author={Sahoo, Suman~Kumar},
      author={Uhlmann, Gunther},
       title={Inverse problems for third-order nonlinear perturbations of biharmonic operators},
        date={2025},
        ISSN={0360-5302,1532-4133},
     journal={Comm. Partial Differential Equations},
      volume={50},
      number={3},
       pages={407\ndash 440},
         url={https://doi.org/10.1080/03605302.2024.2444972},
      review={\MR{4870992}},
}

\bib{Baudouin_schrodinger}{article}{
      author={Baudouin, Lucie},
      author={Puel, Jean-Pierre},
       title={Uniqueness and stability in an inverse problem for the {S}chr\"odinger equation},
        date={2002},
        ISSN={0266-5611,1361-6420},
     journal={Inverse Problems},
      volume={18},
      number={6},
       pages={1537\ndash 1554},
         url={https://doi.org/10.1088/0266-5611/18/6/307},
      review={\MR{1955903}},
}

\bib{Calderon_problem}{incollection}{
      author={Calder\'on, Alberto-P.},
       title={On an inverse boundary value problem},
        date={1980},
   booktitle={Seminar on {N}umerical {A}nalysis and its {A}pplications to {C}ontinuum {P}hysics ({R}io de {J}aneiro, 1980)},
   publisher={Soc. Brasil. Mat., Rio de Janeiro},
       pages={65\ndash 73},
      review={\MR{590275}},
}

\bib{Catalin_Ghosh_Nakamura_aniso_poros_media}{article}{
      author={C\^{a}rstea, C\u{a}t\u{a}lin~I.},
      author={Ghosh, Tuhin},
      author={Nakamura, Gen},
       title={An {I}nverse {B}oundary {V}alue {P}roblem for the {I}nhomogeneous {P}orous {M}edium {E}quation},
        date={2025},
        ISSN={0036-1399},
     journal={SIAM J. Appl. Math.},
      volume={85},
      number={1},
       pages={278\ndash 293},
         url={https://doi.org/10.1137/23M1623197},
      review={\MR{4860199}},
}

\bib{Catalin_Ghosh_Uhlmann_poros_media}{article}{
      author={C\^{a}rstea, C\u{a}t\u{a}lin~I.},
      author={Ghosh, Tuhin},
      author={Uhlmann, Gunther},
       title={An inverse problem for the porous medium equation with partial data and a possibly singular absorption term},
        date={2023},
        ISSN={0036-1410},
     journal={SIAM J. Math. Anal.},
      volume={55},
      number={1},
       pages={162\ndash 185},
         url={https://doi.org/10.1137/21M1465573},
      review={\MR{4538897}},
}

\bib{CKS_SIAM}{article}{
      author={Choulli, Mourad},
      author={Kian, Yavar},
      author={Soccorsi, Eric},
       title={Stable determination of time-dependent scalar potential from boundary measurements in a periodic quantum waveguide},
        date={2015},
        ISSN={0036-1410,1095-7154},
     journal={SIAM J. Math. Anal.},
      volume={47},
      number={6},
       pages={4536\ndash 4558},
         url={https://doi.org/10.1137/140986268},
      review={\MR{3427046}},
}

\bib{CLLO_jde}{article}{
      author={C\^arstea, C\u at\u alin~I.},
      author={Lassas, Matti},
      author={Liimatainen, Tony},
      author={Oksanen, Lauri},
       title={An inverse problem for the {R}iemannian minimal surface equation},
        date={2024},
        ISSN={0022-0396,1090-2732},
     journal={J. Differential Equations},
      volume={379},
       pages={626\ndash 648},
         url={https://doi.org/10.1016/j.jde.2023.10.039},
      review={\MR{4660625}},
}

\bib{Dos_Jems}{article}{
      author={Dos Santos~Ferreira, David},
      author={Kurylev, Yaroslav},
      author={Lassas, Matti},
      author={Salo, Mikko},
       title={The {C}alder\'{o}n problem in transversally anisotropic geometries},
        date={2016},
        ISSN={1435-9855},
     journal={J. Eur. Math. Soc. (JEMS)},
      volume={18},
      number={11},
       pages={2579\ndash 2626},
         url={https://doi.org/10.4171/JEMS/649},
      review={\MR{3562352}},
}

\bib{DOS}{article}{
      author={Dos Santos~Ferreira, D.},
      author={Kenig, C.},
      author={Sj\"{o}strand, J.},
      author={Uhlmann, G.},
       title={Determining a magnetic schr\"{o}dinger operator from partial cauchy data},
        date={2007},
     journal={Comm. Math. Phys.},
      volume={2},
       pages={467\ndash 488},
}

\bib{Nikolas_Stefanov_nonlinear}{article}{
      author={Eptaminitakis, Nikolas},
      author={Stefanov, Plamen},
       title={Weakly nonlinear geometric optics for the {W}estervelt equation and recovery of the nonlinearity},
        date={2024},
        ISSN={0036-1410,1095-7154},
     journal={SIAM J. Math. Anal.},
      volume={56},
      number={1},
       pages={801\ndash 819},
         url={https://doi.org/10.1137/22M1543379},
      review={\MR{4688692}},
}

\bib{Eskin_JMP}{article}{
      author={Eskin, G.},
       title={Inverse problems for the {S}chr\"odinger equations with time-dependent electromagnetic potentials and the {A}haronov-{B}ohm effect},
        date={2008},
        ISSN={0022-2488,1089-7658},
     journal={J. Math. Phys.},
      volume={49},
      number={2},
       pages={022105, 18},
         url={https://doi.org/10.1063/1.2841329},
      review={\MR{2392842}},
}

\bib{Evans_Book}{book}{
      author={Evans, Lawrence~C.},
       title={Partial differential equations},
     edition={Second},
      series={Graduate Studies in Mathematics},
   publisher={American Mathematical Society, Providence, RI},
        date={2010},
      volume={19},
        ISBN={978-0-8218-4974-3},
         url={https://doi.org/10.1090/gsm/019},
      review={\MR{2597943}},
}

\bib{Evans_book_KAM}{article}{
      author={Evans, Lawrence~Craig},
       title={New identities for weak {KAM} theory},
        date={2017},
        ISSN={0252-9599,1860-6261},
     journal={Chinese Ann. Math. Ser. B},
      volume={38},
      number={2},
       pages={379\ndash 392},
         url={https://doi.org/10.1007/s11401-017-1074-9},
      review={\MR{3615495}},
}

\bib{FATY_ann_pde}{article}{
      author={Feizmohammadi, Ali},
      author={Liimatainen, Tony},
      author={Lin, Yi-Hsuan},
       title={An inverse problem for a semilinear elliptic equation on conformally transversally anisotropic manifolds},
        date={2023},
        ISSN={2524-5317,2199-2576},
     journal={Ann. PDE},
      volume={9},
      number={2},
       pages={Paper No. 12, 54},
         url={https://doi.org/10.1007/s40818-023-00153-w},
      review={\MR{4610907}},
}

\bib{FO_jde}{article}{
      author={Feizmohammadi, Ali},
      author={Oksanen, Lauri},
       title={An inverse problem for a semi-linear elliptic equation in {R}iemannian geometries},
        date={2020},
        ISSN={0022-0396,1090-2732},
     journal={J. Differential Equations},
      volume={269},
      number={6},
       pages={4683\ndash 4719},
         url={https://doi.org/10.1016/j.jde.2020.03.037},
      review={\MR{4104456}},
}

\bib{FO_semilinear_elliptic}{article}{
      author={Feizmohammadi, Ali},
      author={Oksanen, Lauri},
       title={An inverse problem for a semi-linear elliptic equation in {R}iemannian geometries},
        date={2020},
        ISSN={0022-0396},
     journal={J. Differential Equations},
      volume={269},
      number={6},
       pages={4683\ndash 4719},
         url={https://doi.org/10.1016/j.jde.2020.03.037},
      review={\MR{4104456}},
}

\bib{GK_applicable_analysis}{article}{
      author={Ghosh, Tuhin},
      author={Krishnan, Venkateswaran~P.},
       title={Determination of lower order perturbations of the polyharmonic operator from partial boundary data},
        date={2016},
        ISSN={0003-6811,1563-504X},
     journal={Appl. Anal.},
      volume={95},
      number={11},
       pages={2444\ndash 2463},
         url={https://doi.org/10.1080/00036811.2015.1092522},
      review={\MR{3546596}},
}

\bib{Hormander_1}{book}{
      author={H\"ormander, Lars},
       title={The analysis of linear partial differential operators. {I}},
      series={Classics in Mathematics},
   publisher={Springer-Verlag, Berlin},
        date={2003},
        ISBN={3-540-00662-1},
         url={https://doi.org/10.1007/978-3-642-61497-2},
        note={Distribution theory and Fourier analysis, Reprint of the second (1990) edition [Springer, Berlin; MR1065993 (91m:35001a)]},
      review={\MR{1996773}},
}

\bib{Hintz_survey}{article}{
      author={Hintz, Peter},
      author={Uhlmann, Gunther},
      author={Zhai, Jian},
       title={The {D}irichlet-to-{N}eumann map for a semilinear wave equation on {L}orentzian manifolds},
        date={2022},
        ISSN={0360-5302,1532-4133},
     journal={Comm. Partial Differential Equations},
      volume={47},
      number={12},
       pages={2363\ndash 2400},
         url={https://doi.org/10.1080/03605302.2022.2122837},
      review={\MR{4526896}},
}

\bib{Isakov_Nachman_TAMS}{article}{
      author={Isakov, Victor},
      author={Nachman, Adrian~I.},
       title={Global uniqueness for a two-dimensional semilinear elliptic inverse problem},
        date={1995},
        ISSN={0002-9947,1088-6850},
     journal={Trans. Amer. Math. Soc.},
      volume={347},
      number={9},
       pages={3375\ndash 3390},
         url={https://doi.org/10.2307/2155015},
      review={\MR{1311909}},
}

\bib{Isakov_Sylvester_CPAM}{article}{
      author={Isakov, Victor},
      author={Sylvester, John},
       title={Global uniqueness for a semilinear elliptic inverse problem},
        date={1994},
        ISSN={0010-3640,1097-0312},
     journal={Comm. Pure Appl. Math.},
      volume={47},
      number={10},
       pages={1403\ndash 1410},
         url={https://doi.org/10.1002/cpa.3160471005},
      review={\MR{1295934}},
}

\bib{isakov1993uniqueness}{article}{
      author={Isakov, V.},
       title={On uniqueness in inverse problems for semilinear parabolic equations},
        date={1993},
        ISSN={0003-9527},
     journal={Arch. Rational Mech. Anal.},
      volume={124},
      number={1},
       pages={1\ndash 12},
         url={https://doi.org/10.1007/BF00392201},
      review={\MR{1233645}},
}

\bib{JNS2023}{misc}{
      author={Johansson, David},
      author={Nurminen, Janne},
      author={Salo, Mikko},
       title={Inverse problems for semilinear elliptic pde with a general nonlinearity $a(x,u)$},
        date={2023},
}

\bib{JNS_low_regu}{misc}{
      author={Johansson, David},
      author={Nurminen, Janne},
      author={Salo, Mikko},
       title={Inverse problems for semilinear elliptic equations with low regularity},
        date={2025},
         url={https://arxiv.org/abs/2505.05278},
}

\bib{KARPMAN1994355}{article}{
      author={Karpman, V.~I.},
       title={Solitons of the fourth order nonlinear schrödinger equation},
        date={1994},
        ISSN={0375-9601},
     journal={Physics Letters A},
      volume={193},
      number={4},
       pages={355\ndash 358},
}

\bib{Karpman1995PRL}{article}{
      author={Karpman, V.~I.},
        date={1995},
     journal={Physical Review Letters},
      volume={74},
       pages={2455\ndash },
}

\bib{Lassas2001boundary}{book}{
      author={Katchalov, Alexander},
      author={Kurylev, Yaroslav},
      author={Lassas, Matti},
       title={Inverse boundary spectral problems},
      series={Chapman \& Hall/CRC Monographs and Surveys in Pure and Applied Mathematics},
   publisher={Chapman \& Hall/CRC, Boca Raton, FL},
        date={2001},
      volume={123},
        ISBN={1-58488-005-8},
         url={https://doi.org/10.1201/9781420036220},
      review={\MR{1889089}},
}

\bib{Kian_Kru_Uhlmann}{article}{
      author={Kian, Yavar},
      author={Krupchyk, Katya},
      author={Uhlmann, Gunther},
       title={Partial data inverse problems for quasilinear conductivity equations},
        date={2023},
        ISSN={0025-5831},
     journal={Math. Ann.},
      volume={385},
      number={3-4},
       pages={1611\ndash 1638},
         url={https://doi.org/10.1007/s00208-022-02367-y},
      review={\MR{4566701}},
}

\bib{Klibanov}{article}{
      author={Klibanov, Michael~V.},
       title={Inverse problems and {C}arleman estimates},
        date={1992},
        ISSN={0266-5611,1361-6420},
     journal={Inverse Problems},
      volume={8},
      number={4},
       pages={575\ndash 596},
         url={https://doi.org/10.1088/0266-5611/8/4/009},
      review={\MR{1178231}},
}

\bib{Kian_forum}{article}{
      author={Kian, Yavar},
      author={Liimatainen, Tony},
      author={Lin, Yi-Hsuan},
       title={On determining and breaking the gauge class in inverse problems for reaction-diffusion equations},
        date={2024},
        ISSN={2050-5094},
     journal={Forum Math. Sigma},
      volume={12},
       pages={Paper No. e25, 42},
         url={https://doi.org/10.1017/fms.2024.18},
      review={\MR{4710715}},
}

\bib{KRU2}{article}{
      author={Krupchyk, Katsiaryna},
      author={Lassas, Matti},
      author={Uhlmann, Gunther},
       title={Determining a first order perturbation of the biharmonic operator by partial boundary measurements},
        date={2012},
        ISSN={0022-1236},
     journal={J. Funct. Anal.},
      volume={262},
      number={4},
       pages={1781\ndash 1801},
         url={https://doi.org/10.1016/j.jfa.2011.11.021},
      review={\MR{2873860}},
}

\bib{KRU1}{article}{
      author={Krupchyk, Katsiaryna},
      author={Lassas, Matti},
      author={Uhlmann, Gunther},
       title={Inverse boundary value problems for the perturbed polyharmonic operator},
        date={2014},
        ISSN={0002-9947},
     journal={Trans. Amer. Math. Soc.},
      volume={366},
      number={1},
       pages={95\ndash 112},
         url={https://doi.org/10.1090/S0002-9947-2013-05713-3},
      review={\MR{3118392}},
}

\bib{KLU_invention}{article}{
      author={Kurylev, Yaroslav},
      author={Lassas, Matti},
      author={Uhlmann, Gunther},
       title={Inverse problems for {L}orentzian manifolds and non-linear hyperbolic equations},
        date={2018},
        ISSN={0020-9910,1432-1297},
     journal={Invent. Math.},
      volume={212},
      number={3},
       pages={781\ndash 857},
         url={https://doi.org/10.1007/s00222-017-0780-y},
      review={\MR{3802298}},
}

\bib{Venky_Rohit}{article}{
      author={Krishnan, Venkateswaran~P.},
      author={Mishra, Rohit~Kumar},
       title={Microlocal analysis of a restricted ray transform on symmetric {$m$}-tensor fields in {$\Bbb{R}^n$}},
        date={2018},
        ISSN={0036-1410,1095-7154},
     journal={SIAM J. Math. Anal.},
      volume={50},
      number={6},
       pages={6230\ndash 6254},
         url={https://doi.org/10.1137/18M1178530},
      review={\MR{3885753}},
}

\bib{KMS_fractional_nonlinear}{article}{
      author={Kow, Pu-Zhao},
      author={Ma, Shiqi},
      author={Sahoo, Suman~Kumar},
       title={An inverse problem for semilinear equations involving the fractional {L}aplacian},
        date={2023},
        ISSN={0266-5611,1361-6420},
     journal={Inverse Problems},
      volume={39},
      number={9},
       pages={Paper No. 095006, 27},
         url={https://doi.org/10.1088/1361-6420/ace9f4},
      review={\MR{4629230}},
}

\bib{KARPMAN2000194}{article}{
      author={Karpman, V.I},
      author={Shagalov, A.G},
       title={Stability of solitons described by nonlinear schrödinger-type equations with higher-order dispersion},
        date={2000},
        ISSN={0167-2789},
     journal={Physica D: Nonlinear Phenomena},
      volume={144},
      number={1},
       pages={194\ndash 210},
         url={https://www.sciencedirect.com/science/article/pii/S0167278900000786},
}

\bib{Kenig_Salo_Survey}{incollection}{
      author={Kenig, Carlos},
      author={Salo, Mikko},
       title={Recent progress in the {C}alder\'on problem with partial data},
        date={2014},
   booktitle={Inverse problems and applications},
      series={Contemp. Math.},
      volume={615},
   publisher={Amer. Math. Soc., Providence, RI},
       pages={193\ndash 222},
         url={http://dx.doi.org/10.1090/conm/615/12245},
      review={\MR{3221605}},
}

\bib{Kian_Soccorsi_Holder_stability_Scrodinger}{article}{
      author={Kian, Yavar},
      author={Soccorsi, Eric},
       title={H\"older stably determining the time-dependent electromagnetic potential of the {S}chr\"odinger equation},
        date={2019},
        ISSN={0036-1410,1095-7154},
     journal={SIAM J. Math. Anal.},
      volume={51},
      number={2},
       pages={627\ndash 647},
         url={https://doi.org/10.1137/18M1197308},
      review={\MR{3919408}},
}

\bib{kuchment_cone_trans}{article}{
      author={Kuchment, Peter},
      author={Terzioglu, Fatma},
       title={Inversion of weighted divergent beam and cone transforms},
        date={2017},
        ISSN={1930-8337},
     journal={Inverse Probl. Imaging},
      volume={11},
      number={6},
       pages={1071\ndash 1090},
         url={https://doi.org/10.3934/ipi.2017049},
      review={\MR{3708176}},
}

\bib{Kian_Tetlow_Holder_stability_dynamical_Scrodinger}{article}{
      author={Kian, Yavar},
      author={Tetlow, Alexander},
       title={H\"older-stable recovery of time-dependent electromagnetic potentials appearing in a dynamical anisotropic {S}chr\"odinger equation},
        date={2020},
        ISSN={1930-8337,1930-8345},
     journal={Inverse Probl. Imaging},
      volume={14},
      number={5},
       pages={819\ndash 839},
         url={https://doi.org/10.3934/ipi.2020038},
      review={\MR{4128444}},
}

\bib{Kian_Uhl_arma}{article}{
      author={Kian, Yavar},
      author={Uhlmann, Gunther},
       title={Recovery of nonlinear terms for reaction diffusion equations from boundary measurements},
        date={2023},
        ISSN={0003-9527,1432-0673},
     journal={Arch. Ration. Mech. Anal.},
      volume={247},
      number={1},
       pages={Paper No. 6, 20},
         url={https://doi.org/10.1007/s00205-022-01831-y},
      review={\MR{4531030}},
}

\bib{Kruchyk_Uhl}{article}{
      author={Krupchyk, Katya},
      author={Uhlmann, Gunther},
       title={Inverse problems for nonlinear magnetic {S}chr\"{o}dinger equations on conformally transversally anisotropic manifolds},
        date={2023},
        ISSN={2157-5045},
     journal={Anal. PDE},
      volume={16},
      number={8},
       pages={1825\ndash 1868},
         url={https://doi.org/10.2140/apde.2023.16.1825},
      review={\MR{4657146}},
}

\bib{Krupchyk_Uhlmann_magnetic}{article}{
      author={Krupchyk, Katya},
      author={Uhlmann, Gunther},
       title={Inverse problems for nonlinear magnetic {S}chr\"odinger equations on conformally transversally anisotropic manifolds},
        date={2023},
        ISSN={2157-5045,1948-206X},
     journal={Anal. PDE},
      volume={16},
      number={8},
       pages={1825\ndash 1868},
         url={https://doi.org/10.2140/apde.2023.16.1825},
      review={\MR{4657146}},
}

\bib{LLLS_JMPA}{article}{
      author={Lassas, Matti},
      author={Liimatainen, Tony},
      author={Lin, Yi-Hsuan},
      author={Salo, Mikko},
       title={Inverse problems for elliptic equations with power type nonlinearities},
        date={2021},
        ISSN={0021-7824,1776-3371},
     journal={J. Math. Pures Appl. (9)},
      volume={145},
       pages={44\ndash 82},
         url={https://doi.org/10.1016/j.matpur.2020.11.006},
      review={\MR{4188325}},
}

\bib{LLLS_Revista}{article}{
      author={Lassas, Matti},
      author={Liimatainen, Tony},
      author={Lin, Yi-Hsuan},
      author={Salo, Mikko},
       title={Partial data inverse problems and simultaneous recovery of boundary and coefficients for semilinear elliptic equations},
        date={2021},
        ISSN={0213-2230,2235-0616},
     journal={Rev. Mat. Iberoam.},
      volume={37},
      number={4},
       pages={1553\ndash 1580},
         url={https://doi.org/10.4171/rmi/1242},
      review={\MR{4269409}},
}

\bib{LLPT_apde}{article}{
      author={Lassas, Matti},
      author={Liimatainen, Tony},
      author={Potenciano-Machado, Leyter},
      author={Tyni, Teemu},
       title={Stability and {L}orentzian geometry for an inverse problem of a semilinear wave equation},
        date={2025},
        ISSN={2157-5045,1948-206X},
     journal={Anal. PDE},
      volume={18},
      number={5},
       pages={1065\ndash 1118},
         url={https://doi.org/10.2140/apde.2025.18.1065},
      review={\MR{4904383}},
}

\bib{LLS_Math_ann}{article}{
      author={Lassas, Matti},
      author={Liimatainen, Tony},
      author={Salo, Mikko},
       title={The {P}oisson embedding approach to the {C}alder\'on problem},
        date={2020},
        ISSN={0025-5831,1432-1807},
     journal={Math. Ann.},
      volume={377},
      number={1-2},
       pages={19\ndash 67},
         url={https://doi.org/10.1007/s00208-019-01818-3},
      review={\MR{4099622}},
}

\bib{LLST_fractional_power}{article}{
      author={Liimatainen, Tony},
      author={Lin, Yi-Hsuan},
      author={Salo, Mikko},
      author={Tyni, Teemu},
       title={Inverse problems for elliptic equations with fractional power type nonlinearities},
        date={2022},
        ISSN={0022-0396},
     journal={J. Differential Equations},
      volume={306},
       pages={189\ndash 219},
         url={https://doi.org/10.1016/j.jde.2021.10.015},
      review={\MR{4332042}},
}

\bib{Lion-Magenes2}{book}{
      author={Lions, J.-L.},
      author={Magenes, E.},
       title={Probl\`emes aux limites non homog\`enes et applications. {V}ol. 1},
      series={Travaux et Recherches Math\'ematiques},
   publisher={Dunod, Paris},
        date={1968},
      volume={No. 17},
      review={\MR{247243}},
}

\bib{LOSST_NLS_2025}{article}{
      author={Lassas, Matti},
      author={Oksanen, Lauri},
      author={Sahoo, Suman~Kumar},
      author={Salo, Mikko},
      author={Tetlow, Alexander},
       title={Coefficient determination for nonlinear {S}chr\"odinger equations on manifolds},
        date={2025},
        ISSN={0036-1410,1095-7154},
     journal={SIAM J. Math. Anal.},
      volume={57},
      number={4},
       pages={4425\ndash 4458},
         url={https://doi.org/10.1137/24M1704257},
      review={\MR{4941933}},
}

\bib{lai2024partialdatainverseproblems}{misc}{
      author={Lai, Ru-Yu},
      author={Uhlmann, Gunther},
      author={Yan, Lili},
       title={Partial data inverse problems for the nonlinear magnetic schr\"odinger equation},
        date={2024},
         url={https://arxiv.org/abs/2411.06369},
}

\bib{Lai_arma}{article}{
      author={Lai, Ru-Yu},
      author={Uhlmann, Gunther},
      author={Zhou, Hanming},
       title={Recovery of coefficients in semilinear transport equations},
        date={2024},
        ISSN={0003-9527,1432-0673},
     journal={Arch. Ration. Mech. Anal.},
      volume={248},
      number={4},
       pages={Paper No. 62, 41},
         url={https://doi.org/10.1007/s00205-024-02007-6},
      review={\MR{4760979}},
}

\bib{Lai_Zhou}{article}{
      author={Lai, Ru-Yu},
      author={Zhou, Ting},
       title={Partial data inverse problems for nonlinear magnetic {S}chr\"{o}dinger equations},
        date={2023},
        ISSN={1073-2780},
     journal={Math. Res. Lett.},
      volume={30},
      number={5},
       pages={1535\ndash 1563},
         url={https://doi.org/10.4310/mrl.2023.v30.n5.a10},
      review={\MR{4747871}},
}

\bib{nurminen_sahoo}{misc}{
      author={Nurminen, Janne},
      author={Sahoo, Suman~Kumar},
       title={An inverse problem for a nonlinear biharmonic operator},
        date={2025},
         url={https://arxiv.org/abs/2504.06624},
}

\bib{Janne_ms_nonlinear_ana}{article}{
      author={Nurminen, Janne},
       title={An inverse problem for the minimal surface equation},
        date={2023},
        ISSN={0362-546X,1873-5215},
     journal={Nonlinear Anal.},
      volume={227},
       pages={Paper No. 113163, 19},
         url={https://doi.org/10.1016/j.na.2022.113163},
      review={\MR{4503820}},
}

\bib{Janne_ms_nonlinearity}{article}{
      author={Nurminen, Janne},
       title={An inverse problem for the minimal surface equation in the presence of a {R}iemannian metric},
        date={2024},
        ISSN={0951-7715,1361-6544},
     journal={Nonlinearity},
      volume={37},
      number={9},
       pages={Paper No. 095029, 22},
         url={https://doi.org/10.1088/1361-6544/ad6949},
      review={\MR{4785481}},
}

\bib{PSU_book}{book}{
      author={Paternain, Gabriel},
      author={Salo, Mikko},
      author={Uhlmann, Gunther},
       title={Geometric inverse problems---with emphasis on two dimensions},
      series={Cambridge Studies in Advanced Mathematics},
   publisher={Cambridge University Press, Cambridge},
        date={2023},
      volume={204},
        ISBN={978-1-316-51087-2},
      review={\MR{4520155}},
}

\bib{Rakesh_Salo_Siam}{article}{
      author={Rakesh},
      author={Salo, Mikko},
       title={Fixed angle inverse scattering for almost symmetric or controlled perturbations},
        date={2020},
        ISSN={0036-1410,1095-7154},
     journal={SIAM J. Math. Anal.},
      volume={52},
      number={6},
       pages={5467\ndash 5499},
         url={https://doi.org/10.1137/20M1319309},
      review={\MR{4170189}},
}

\bib{Rakesh_Salo_IP}{article}{
      author={Rakesh},
      author={Salo, Mikko},
       title={The fixed angle scattering problem and wave equation inverse problems with two measurements},
        date={2020},
        ISSN={0266-5611,1361-6420},
     journal={Inverse Problems},
      volume={36},
      number={3},
       pages={035005, 42},
         url={https://doi.org/10.1088/1361-6420/ab23a2},
      review={\MR{4068234}},
}

\bib{simon2024gaussianbeamsinverseproblems}{misc}{
      author={St-Amant, Simon},
       title={Gaussian beams and inverse problems for connections at high fixed frequency},
        date={2024},
         url={https://arxiv.org/abs/2402.13854},
}

\bib{Barreto_Stefanov_nonlinear}{article}{
      author={S\'a{}~Barreto, Ant\^onio},
      author={Stefanov, Plamen},
       title={Recovery of a general nonlinearity in the semilinear wave equation},
        date={2024},
        ISSN={0921-7134,1875-8576},
     journal={Asymptot. Anal.},
      volume={138},
      number={1-2},
       pages={27\ndash 68},
         url={https://doi.org/10.3233/asy-231890},
      review={\MR{4765428}},
}

\bib{SS_linearized}{article}{
      author={Sahoo, Suman~Kumar},
      author={Salo, Mikko},
       title={The linearized {C}alder\'on problem for polyharmonic operators},
        date={2023},
        ISSN={0022-0396,1090-2732},
     journal={J. Differential Equations},
      volume={360},
       pages={407\ndash 451},
         url={https://doi.org/10.1016/j.jde.2023.03.017},
      review={\MR{4562046}},
}

\bib{sylvester1987global}{article}{
      author={Sylvester, John},
      author={Uhlmann, Gunther},
       title={A global uniqueness theorem for an inverse boundary value problem},
        date={1987},
        ISSN={0003-486X,1939-8980},
     journal={Ann. of Math. (2)},
      volume={125},
      number={1},
       pages={153\ndash 169},
         url={https://doi.org/10.2307/1971291},
      review={\MR{873380}},
}

\bib{Sun_EJDE}{article}{
      author={Sun, Ziqi},
       title={An inverse boundary-value problem for semilinear elliptic equations},
        date={2010},
        ISSN={1072-6691},
     journal={Electron. J. Differential Equations},
       pages={No. 37, 5},
      review={\MR{2602870}},
}

\bib{Uhl_survey}{article}{
      author={Uhlmann, G.},
       title={Electrical impedance tomography and {C}alder\'{o}n's problem},
        date={2009},
        ISSN={0266-5611},
     journal={Inverse Problems},
      volume={25},
      number={12},
       pages={123011, 39},
         url={https://doi.org/10.1088/0266-5611/25/12/123011},
      review={\MR{3460047}},
}

\bib{Uhl_Zhai_math_annalen}{article}{
      author={Uhlmann, Gunther},
      author={Zhai, Jian},
       title={Determination of the density in a nonlinear elastic wave equation},
        date={2024},
        ISSN={0025-5831,1432-1807},
     journal={Math. Ann.},
      volume={390},
      number={2},
       pages={2825\ndash 2858},
         url={https://doi.org/10.1007/s00208-024-02797-w},
      review={\MR{4801841}},
}

\end{biblist}
\end{bibdiv}
\bibliographystyle{siam}
\end{document}